\documentclass[11pt]{article}

\usepackage{geometry,amssymb,amsmath,amsthm,graphicx,xcolor}
\definecolor{mygreen}{RGB}{10,100,10}
\usepackage{tikz}
\usetikzlibrary{cd,calc}
\usetikzlibrary{intersections}
\usepackage[pdfencoding=auto, psdextra]{hyperref}
\hypersetup{
    colorlinks = true,
    linkcolor = mygreen,
    anchorcolor = mygreen,
    citecolor = mygreen,
    filecolor = mygreen,
    urlcolor = mygreen
    }

\allowdisplaybreaks

\newtheoremstyle{reduced_space}{}{-0.5\baselineskip}{}{}{\bfseries}{}{.5em}{}
\theoremstyle{reduced_space}

\newtheorem{thm}{Theorem}[section]
\newtheorem{prp}[thm]{Proposition}
\newtheorem{lmm}[thm]{Lemma}
\newtheorem{crl}[thm]{Corollary}

\newtheorem{fact}[thm]{Fact}

\theoremstyle{definition}
\newtheorem{dfn}[thm]{Definition}
\newtheorem{eg}[thm]{Example}

\theoremstyle{remark}
\newtheorem{rmk}[thm]{Remark}

\def\eref#1{(\ref{#1})}

\def\ov#1{\overline{#1}}

\def\wt#1{\widetilde{#1}}
\def\tl#1{\tilde{#1}}
\def\tn#1{\textnormal{#1}}

\def\ch#1{\check{#1}}

\def\wh#1{\widehat{#1}}

\def\lra{\longrightarrow}
\def\xra#1{\xrightarrow{{#1}}}

\def\sx{\!\times\!}

\def\cA{\mathcal A}

\def\ccM#1{\ov{C}_{#1}(M,\i)}
\def\ccE#1{\ov{C}_{#1}(\pi)}
\def\ccEt{{\ov{C}_2(\pi)}}
\def\ccEft{{\ov{C}^{E(\Ga)}_2(\pi)}}
\def\ccEp#1{\ov{C}_{#1}(\pi')}
\def\ccEpp#1{\ov{C}_{#1}(\pi'')}
\def\tS{\wt{S}}
\def\tcM{X_{\Ga}}

\def\rxg{{\mathring{X}_\Ga}}

\def\cq{{\ov{C}_2(\pi)/\!\!\sim_F}}
\def\cqp#1{{\ov{C}_2(\pi#1)/\!\!\sim_{F#1}}}
\def\Sp{{S^{d-1}_\pi}}

\def\vgsa#1{{\ov{V}^{#1}_{\!A}}}

\def\com{{q^*\ch\om}}
\def\um{{\underline{m}}}
\def\ul{{\underline{l}}}
\def\pr{{\tn{pr}}}

\def\cG{\mathcal G}

\def\cJ{\mathcal J}

\def\cN{\mathcal N}

\def\R{\mathbb R}

\def\cS{\mathcal S}
\def\cT{\mathcal T}
\def\cU{\mathcal U}
\def\cZ{\mathcal Z}
\def\Z{\mathbb Z}

\def\al{\alpha}

\def\de{\delta}
\def\ep{\epsilon}
\def\ga{\gamma}
\def\io{\iota}

\def\la{\lambda}
\def\si{\sigma}
\def\om{\omega}

\def\De{\Delta}
\def\Ga{\Gamma}

\def\Om{\Omega}

\def\Th{\Theta}

\def\fG{\mathfrak G}

\def\dim{\tn{dim}}

\def\id{\tn{id}}

\def\Diff{\tn{Diff}}
\def\Homeo{\tn{Homeo}}
\def\eqodd{\sim_{\tn{odd}}}
\def\eqeven{\sim_{\tn{even}}}
\def\odd{\tn{odd}}
\def\even{\tn{even}}
\def\sgn{\tn{sgn}}
\def\vertex{\tn{vertex}}
\def\edge{\tn{edge}}

\def\Au{\tn{Aut}^u}
\def\rsa{\mathring{\cS}^\Ga_A}
\def\rrsa{\ringring{\cS}^\Ga_A}

\def\rrrsa{\ringringring{\cS}^\Ga_A}

\def\quo{{\textnormal{quo}}}

\def\i{\infty}

\def\eset{\emptyset}
\def\prt{\partial}

\def\bu{\bullet}

\newcommand\ringring[1]{%
  {
   \mathop{\kern0pt #1}\limits^{
     \vbox to-1.85ex{
       \kern-2ex 
       \hbox to 0pt{\hss\normalfont\kern.1em \r{}\kern-.45em \r{}\hss}%
       \vss 
     }
   }
  }
}
\newcommand\ringringring[1]{%
  {
   \mathop{\kern0pt #1}\limits^{
     \vbox to-1.85ex{
       \kern-2ex 
       \hbox to 0pt{\hss\normalfont\kern0.05em \r{}\kern-.5em\r{}\kern-.5em\r{}\hss}%
       \vss 
     }
   }
  }
}
\begin{document}

\title{Kontsevich's Characteristic Classes as Topological Invariants of Configuration Space Bundles}
\author{Xujia Chen}
\date{\today}

\maketitle

\begin{abstract}
Kontsevich's characteristic classes are invariants of framed smooth fiber bundles with homology sphere fibers. It was shown by Watanabe that they can be used to distinguish smooth $S^4$-bundles that are all trivial as topological fiber bundles. 
In this article we show that this ability of Kontsevich's classes 
is a manifestation of the following principle: the ``real blow-up'' construction on a smooth manifold essentially depends on its smooth structure and thus, given a smooth manifold (or smooth fiber bundle) $M$, the topological invariants of spaces constructed from $M$ by real blow-ups could potentially differentiate smooth structures on $M$. The main theorem says that Kontsevich's characteristic classes of a smooth framed bundle $\pi$ are determined by the topology of the 2-point configuration space bundle of $\pi$ and framing data. 
\end{abstract}
\tableofcontents
\setlength{\parskip}{\baselineskip}
\setlength{\parindent}{0cm}

\section{Introduction}
\label{intro_sec}

Given the following data: 
\begin{itemize}
\vspace{-.6cm}
\itemsep0em 
\item a smooth fiber bundle $E\xra{\pi}B$ whose fibers are homology spheres, 
\item a smooth section $s_\i\!:\!B\to E$ and a trivialization $t$ of $\pi$ in a neighborhood $U$ of $s_\i(B)$,
\item a vertical framing $F$ on $E-s_\i(B)$ which is standard (i.e.~looks like the standard framing on $\R^n$ near $\i$) with respect to $t$ in $U$;
\end{itemize}
\vspace{-.6cm}
Kontsevich's characteristic classes are a collection of cohomology classes in $H^*(B;\R)$, parameterized by some combinatorial data (``graph homology''). They were introduced by Kontsevich \cite{Kontsevich} and have been exploit by various authors thereafter; see \cite{Lescop} for a good introduction. In \cite{Watanabe}, Watanabe constructed smooth (trivialized near a section and framed, in the above sense) $S^4$-bundles with non-trivial Kontsevich's characteristic classes, implying that as smooth fiber bundles (with fixed trivialization near a section) they are non-trivial, while as topological fiber bundles they are trivial. 

We would like to understand why Kontsevich's characteristic classes are able to differentiate smooth fiber bundles that are topologically the same. These classes are constructed by considering the ``configuration space bundles'' associated to $E\xra{\pi}B$, which are obtained by doing a sequence of real (oriented) blow-up operations fiberwise, and then doing some sort of intersection in the total space of the configuration space bundle to get an intersection number. Since intersection theoretical invariants usually do not depend on the smooth structure, while the real blow-up operations do, it is plausible that different smooth structures on the original bundle $E\xra{\pi}B$ yield different topological structures on the induced configuration space bundles, and Kontsevich's characteristic classes only depend on the topological bundle structure of the configuration space bundles (together with some information from the framing). The purpose of the present article is to make this statement precise and to give a detailed proof. The main theorem \ref{main_thm} says that the topological information from the 2-point configuration space bundle of $E\xra{\pi}B$, together with framing, determines Kontsevich's characteristic classes of $E\xra{\pi}B$. 

\begin{rmk}
The homotopy type of the real oriented blow up of a manifold $X$ along a submanifold $Y$, $Bl_YX$, does not depend on the smooth structure on $X$, since it is just $X\!-\!Y$; but the topological structure of $\prt Bl_YX$ as a sphere bundle over $Y$ and how $\prt Bl_YX$ is attached to $X\!-\!Y$ do depend on the smooth structure in an essential way; and a framing on the normal bundle of $Y$ can be used to capture this structure. 
\end{rmk}

For example, the simplest Kontsevich's characteristic classes, when it is a number ($\Th$-graph invariants), can be viewed as the triple intersection number of a cohomology class (called ``propagator class'' in Section \ref{propagator_sec}) in a space $\ov{C}_2(\pi)/\!\!\sim_F$, where $\ov{C}_2(\pi)$ is the total space of the 2-point configuration space bundle associated to $E\xra{\pi}B$, and $\sim_F$ is an equivalence relation on $\prt^v\ov{C}_2(\pi)$ (the vertical boundary of $\ov{C}_2(\pi)$) -- it uses the framing data $F$ to ``pinch'' $\prt^v\ov{C}_2(\pi)$, making it lower-dimensional. 
\footnote{This idea of constructing $\cq$ and viewing the propagator as a cohomology class in it has already been explored in the early work of Kuperberg and Thurston \cite{KT}.}

\subsection{Statement of main result}\label{thm_sec}

Throughout Sections 1-5, let $M$ be a closed smooth $d$-dimensional manifold whose $R$-homology groups are the same as that of the $d$-sphere, where $R=\Z$ or $\R$, and let $\i$ be a fixed point in $M$. Let $\ccM{2}$ be the configuration space of 2 ordered, not equal and not $\i$, points in $M$, compactified by the Fulton-MacPherson compactification.
More precisely, $\ov{C}_2(M,\i)$ is obtained from $M\sx M$ by first blowing up $\i\sx\i$ and then blowing up the strict transforms of $\i\sx M$, $M\sx\i$ and the diagonal $\De$. (All the blow-ups we use are oriented.) Denote by $f_+,f_-:\ccM{2}\to M$ the two forgetful maps lifting
$$M^2
\lra M, \quad f_+(x_1,x_2)=x_2,\quad f_-(x_1,x_2)=x_1,$$
respectively. So, $(f_-,f_+):\ccM{2}\to M\sx M$ is the blow down map. 

Denote by $\Diff_+(M)$ (resp. $\Homeo_+(M)$) the group of orientation-preserving diffeomorphisms (resp. homeomorphisms) of $M$, with the Whitney (resp. compact-open) topology, and 
define 
\begin{gather*}
\Diff_+(M,N_\i):=\big\{g\in\Diff_+(M)\big|\ \exists \tn{ neighborhood } U \ni\i \tn{ such that } g|_U=\id\big\}\\
\Homeo_+(M,N_\i):=\big\{g\in\Homeo_+(M)\big|\ \exists \tn{ neighborhood } U \ni\i \tn{ such that } g|_U=\id\big\}\\
\cG:=\big\{(\tl{g},g)\in\Homeo(\ccM{2})\sx\Homeo_+(M,N_\i)\big|\ g\circ f_\pm=f_\pm\circ\tl g\}.
\end{gather*}
By a smooth $(M,\i)$-bundle $E\xra{\pi}B$ we mean a fiber bundle with typical fiber $(M,\i)$ and structure group $\Diff_+(M,N_\i)$. It has a canonical section $s_\i$ and a canonical germ of trivializations 
\[
\begin{tikzcd}
t:B\sx U \ar[rr,"\sim"] \ar[rd,"\tn{projection to } B"'] & & \wt{U} \ar[ld,"\pi"] \\
& B &
\end{tikzcd}
\]
of some neighborhood $\wt U\supset s_\i(B)$ ($U$ is some neighborhood of $\i$ in $M$). Abusing notation we also denote the image $s_\i(B)$ by $s_\i$. 

A framing $F$ on a smooth $(M,\i)$-bundle $E\xra{\pi}B$ is a continuous choice of basis for the vertical tangent space at every point of $E\!-\!s_\i$,
such that there are neighborhoods $E\supset\wt U\supset s_\i$, $M\supset U\supset\i$, a diffeomorphism $(U,\i)\approx((\R^d-0)\sqcup\i,\i)$, satisfying that $t^*(F|_{\wt U})$ is the standard framing on $\R^d$ under this diffeomorphism. 

Let $\ccE{2}\to B$ be the associated $\ccM{2}$-bundle, $f_+,f_-:\ccE{2}\to E$ the two forgetful maps and $\prt^v\ccE{2}$ consists of boundaries of every fiber. So 
$$\prt^v\ccE{2}=(f_-,f_+)^{-1}\big(\De(\pi)\cup s_\i\!\!\times_{\!B}\!E\cup E\!\times_{\!B}\!s_\i\big),$$
where $\times_B$ denotes fiber product over $B$ and $\De(\pi)\subset E\!\times_{\!B}\!E$ is the fiberwise diagonal.
The framing $F$ induces a map $(f_-,f_+)^{-1}(\De(\pi))\to S^{d-1}$ which at each point $x\in E\!-\!s_\i$ maps $(f_-,f_+)^{-1}(x,x)=S\cN^v_{\De(\pi)}E|_x\approx ST^v_xE$ to $S^{d-1}$ using $F:T^v_xE\approx \R^d$; here $ST^vE$ denotes the sphere bundle of the vertical tangent bundle of $E$ and $S\cN^v_{\De(\pi)}E$ denotes the sphere bundle of the vertical normal bundle of $\De(\pi)$ in $E$. 
Using the trivialization $t$, this map can be extended to a map $\prt^v\ccE{2}\to S^{d-1}$ that we still denote by $F$, abusing notation; see for example \cite[\S2.3]{WatanabeAdd} for the detailed definition. 

Kontsevich's invariants, as cohomology classes in $B$ parameterized by 
graph homology, are defined for smooth framed $(M,\i)$-bundles with smooth base $B$ and $R=\R$.
In this article we need to restrict to the case of having only trivalent graphs\footnote{
The general definition of Kontsevich's invariants don't need to assume this, but for us, because of a nuanced technicality (more precisely, we need the very last sentence of Section \ref{graph_sec} to hold, which is needed in the proof of Lemma \ref{cancel_lmm}), the proof of Theorem \ref{main_thm} only works if all the vertices of the graph have the same valency. Since it is customary to consider trivalent graphs, we will just assume that.
}, 
so, throughout this article, Kontsevich's invariant for a bundle $\pi:E\to B$ is a map
$$\{\tn{formal sum of trivalent graphs, closed in graph homology}\}\lra H^*(B;R).$$

One of the by-products of this article is to extend Kontsevich's invariants to the case where $B$ is any paracompact Hausdorff space and $R=\Z$. 
The main theorem is

\begin{thm}\label{main_thm}
Let $E'\xra{\pi'}B'$, $E''\xra{\pi''}B''$ be smooth $(M,\i)$-fiber bundles and $s'_\i,s''_\i$ their canonical sections. Let $F',F''$ be framings on $\pi',\pi''$, respectively. If there exist \emph{continuous} maps 
$$\tl h: \ccEp{2}\lra\ccEpp{2},\ h:E'\lra E'',\ h_B:B'\lra B'',\ h_S:S^{d-1}\xrightarrow{\tn{homeomorphism}
} S^{d-1}$$
such that 
the following diagrams commute 
\[
\begin{tikzcd}
\ccEp{2}\rar["\tl h"]\dar["f'_+",shift left=1]\dar["f'_-"',shift right=1] & \ccEpp{2}\dar["f''_+",shift left=1]\dar["f''_-"',shift right=1] \\
E'\rar["h"]\dar["\pi'"] & E''\dar["\pi''"] \\
B'\uar[bend left, "s'_\i",pos=0.45] \rar["h_B"] & B''\uar[bend left,"s''_\i",pos=0.45]
\end{tikzcd}
\qquad
\begin{tikzcd}
\prt^v\ccEp{2}\rar["\tl{h}|_{\small \prt^v\ccEp{2}}"]\dar["F'"] &[20pt] \prt^v\ccEpp{2}\dar["F''"] \\
S^{d-1}\rar["h_S"] & S^{d-1}
\end{tikzcd}
\]
and for every point $b\in B$, $\tl{h},h$ restrict to orientation-preserving homeomorphisms on the fibers over $b$, 
then the Kontsevich's characteristic classes of $\pi''$ pull back to those of $\pi'$ by $h_B$. 
\end{thm}

\begin{rmk}\label{Gfr_rmk}
Without the commutativity condition of the diagram on the right, the assumption of this theorem is the same as saying $(\tl h,h)$ is a $\cG$-bundle map. 
See Section \ref{rmk_sec} for some discussions about the condition in this theorem. In particular, Theorem \ref{restatement_thm} is a restatement and extension of Theorem \ref{main_thm}. 
\end{rmk} 

\begin{rmk}
Theorem 1.4 in \cite{LinXie} (which appeared a few months after the present paper) implies Theorem \ref{main_thm}, at least in the case $M$ is a sphere. 
See Theorem \ref{restatement_thm}---a restatement of Theorem \ref{main_thm} (and the second to last paragraph above it, as well as Proposition \ref{frame_prp}), for why this is the case. 
\end{rmk}


\subsection{Outline of the proof}\label{outline_sec}

We re-construct Kontsevich's characteristic classes in a way that all the definitions are made using only the topological bundle structure on $\ov{C}_2(\pi),\pi$ and the maps $f_\pm$, avoiding using the smooth structure in definitions. This will make Theorem \ref{main_thm} automatic. Sections 2-4 are devoted to this re-construction, which is really just a translation of the original construction. In Section \ref{equiv_sec} we show that the new definition is equivalent to the original one. 

To make such a re-construction, the natural strategy is to translate the original construction using differential forms into the language of some topological cohomology theory---here we use \v{C}ech cochains---and thus avoid using the smooth structure. However, the cochains in all topological cohomology theories are rather cumbersome to work with, so we need to find the appropriate spaces so that we actually work with cohomology classes. The seemingly unmotivated definitions in Section \ref{space_sec} are for this purpose. 
This approach is very similar to the work of Kuperberg and Thurston \cite{KT} (and some later works, e.g. \cite{Koytcheff}), but with some major differences; see Remark \ref{KTK_rmk} below. 

We describe in a bit more detail how the re-construction is done below. 
First, in this paragraph, we briefly recall how the original construction roughly goes. 
Let $E\xra{\pi}B$ be a smooth $(M,\i)$-bundle with smooth, compact base $B$, and $F$ a framing on $\pi$. 
Let $\Ga$ be a trivalent graph that is closed in graph homology\footnote{
Throughout this paper we work with a single graph $\Ga$ instead of a formal sum of graphs, but this is just for simplicity. All the arguments can be modified to work if we assume $\Ga$ is a formal sum of graphs instead; see Remark \ref{graphsum_rmk}.
}; denote its vertex set and edge set by $V(\Ga),E(\Ga)$, respectively. Denote by $\ov{C}_{V(\Ga)}(\pi)$ the Fulton-MacPherson compactification of 
$$C_{V(\Ga)}(\pi):=\big\{(x_v\in E)_{v\in V(\Ga)}\big|\,x_v\notin s_\i,\pi(x_v)=\pi(x_w),x_v\neq x_w,\forall\,v\neq w\in V(\Ga)\big\}.$$
It is a manifold with boundaries and corners. For every edge $e$ of $\Ga$, there is a forgetful map $f_e\!:\!\ov{C}_{V(\Ga)}(\pi)\to \ov{C}_2(\pi)$ forgetting everything but the two points labeled by the vertices adjacent to $e$. Take a closed $(d\!-\!1)$-form $\om$ (called {\it propagator}) on $\ov{C}_2(\pi)$ satisfying $\om|_{\prt^v\ov{C}_2(\pi)}=F^*\tn{vol}$ for some form $\tn{vol}$ on $S^{d-1}$ such that $\int_{S^{d-1}}\tn{vol}=1$. 
Then the desired characteristic class (with parameter $\Ga$) is defined to be the class represented by the push-forward of $\bigwedge_ef_e^*\om_e$ to $B$. This pushed-forward form is not automatically closed or independent of the choice of $\om$; all the trouble here is that $\ov{C}_{V(\Ga)}(\pi)$ has boundary. The codimension-1 boundary strata of $\ov{C}_{V(\Ga)}(\pi)$ are in correspondence with subsets $A\subset\{\i\}\sqcup V(\Ga)$ having at least 2 elements. Denote by $\ov{\cS}_A$ the closed boundary stratum corresponding to $A$; it represents the configurations where the points with labels in $A$ all coincide and ``bubble off to a screen''. These boundary strata are divided into 4 types, and treated separately. 
\begin{enumerate}
\vspace{-.5cm}
\itemsep0em 
\item[1)] The subgraph $\Ga_A$ of $\Ga$ spanned by vertices in $A$ has a zero- or univalent vertex, and is not of the form as in the 4-th type below. Then $\ov{\cS}_A$ is contracted by the map to $B$ and thus does not contribute.
\item[2)] $\Ga_A$ has a bivalent vertex, then there is an involution on $\ov{\cS}_A$, making the contribution from $\ov{\cS}_A$ cancel with itself. 
\item[3)] $\i\in A$ or $A=V(\Ga)$. This case relies on the framing $F$, the trivialization of $\pi$ near $s_\i$, and that $\om$ is defined to be compatible with $F$. A dimension count shows that $\ov{\cS}_A$ does not contribute either. 
\item[4)] $\Ga_A$ has two vertices connected by an edge, then these $\ov{\cS}_A$ correspond to boundary terms of $\Ga$ in graph cohomology. Since $\Ga$ is closed, their total contribution is 0. 
\end{enumerate}

In our re-construction, we no longer use differential form propagators; instead, define the space $\cq$ obtained from $\ccE{2}$ by contracting $\prt^v\ccE{2}$ to $S^{d-1}$ using the framing $F$, and the propagator can be naturally replaced by a {\it propagator class} $\Om\in H^{d-1}(\cq)$. Denote by $q:\ccE{2}\to\cq$ the quotient map. 

Our construction also treats the boundary strata of $\ccE{V(\Ga)}$ type by type; each type is treated in the same spirit as the original arguments. 
\begin{enumerate}
\vspace{-.5cm}
\itemsep0em 
\item[1)] Instead of using $\ov{C}_{V(\Ga)}(\pi)$, we use a ``weaker'' compactification of $C_{V(\Ga)}(\pi)$, denoted by $\ov{C}_\Ga(\pi)$. 
When marked points coincide, $\ov{C}_{V(\Ga)}(\pi)$ records the collapsing directions of every pair of points, as well as the relative collapsing speed of each triple of points. 
But $\ov{C}_\Ga(\pi)$ only records the collapsing directions of pairs of points whose labels in $V(\Ga)$ are connected by an edge of $\Ga$; $\ov{C}_\Ga(\pi)$ also does not record the relative collapsing speed of triples of points. 
Indeed, $\ccE{\Ga}$ is defined to be the closure of the image of
$$C_{V(\Ga)}(\pi)\xrightarrow{(f_e)_{e\in E(\Ga)}}\ccE{2}^{E(\Ga)}.$$ 
We denote the image of each $\ov{\cS}_A$ in $\ccE{\Ga}$ by $\ov{\cS}^\Ga_A$. 
Those $\ov\cS_A\subset\ccM{V(\Ga)}$ for $A$ of type 1 are ``contracted'', i.e., $\ov{\cS}^\Ga_A\subset\ccE{\Ga}$ has codimension 2 or higher. 

\item[2\&4)] The original arguments in these two cases tell us that each type 2 $\ov{\cS}_A$ has an involution and type 4 $\ov{\cS}_A$'s can be paired-up and cancel each other. Morally, we want to define a new space obtained from $\ccE{\Ga}$ by gluing each type 2 $\ov{\cS}^\Ga_A$ to itself by the involution, and glue each type 4 pair $\ov{\cS}^\Ga_{A_1},\ov{\cS}^\Ga_{A_2}$ to each other, thus they are no longer boundaries. However, the gluing procedure involves much technicality; so we instead just use the space 
$$X_\Ga(\pi):=\bigcup_{\si\in\tS_{E(\Ga)}}\si\cdot\ov{C}_\Ga(\pi)\subset\ov{C}_2(\pi)^{E(\Ga)},$$
where $\tS_{E(\Ga)}$ is a slight generalization of the permutation group of the set $E(\Ga)$, which acts on $\ccE{2}^{E(\Ga)}$ by permuting the factors. 
Taking the union of all the translates of $\ov{C}_\Ga(\pi)$ by $\wt{S}_{E(\Ga)}$ makes the type 2 and type 4 boundary strata of them coincide and cancel with each other. 
We will explain in Section \ref{XGa_sec} that after removing a codimension 2 subset $T_2(\pi)$ from $X_\Ga(\pi)$, it is (fiberwise) a ``manifold with boundary and bindings'' (see Figure \ref{binding_fig} for what binding means), where the pages adjacent to a binding, when counted with sign, sum up to 0. Since $X_\Ga(\pi)$ is a subspace of $\ov{C}_2(\pi)^{E(\Ga)}$, for each edge $e$ of $\Ga$, we still have the forgetful map $f_e:X_\Ga(\pi)\to\ccE{2}$. 

\begin{figure}[h]
\centering
\begin{minipage}{.35\textwidth}
\centering
\begin{tikzpicture}
\draw [thick] (0,0) -- (0,2);
\draw (0,2) -- (-1.5,2) -- (-1.5,0);
\draw [name path=l2, color=lightgray] (-1.5,0) -- (0,0);
\draw (0,0) -- (1,-1.5);
\draw [name path=l4] (1,-1.5) -- (1,0.5);
\draw (1,0.5) --(0,2);
\draw (0,2) -- (-1, .5);
\draw [name path=l3] (-1, .5) -- (-1,-1.5);
\draw (-1,-1.5) -- (0,0);
\draw (0,2) -- (1.3,2.5) -- (1.3,.5);
\draw [name path=l1, color=lightgray](1.3,.5) -- (0,0);
\path [name intersections={of=l2 and l3,by=E1}];
\draw (-1.5,0) -- (E1);
\path [name intersections={of=l1 and l4,by=E2}];
\draw (1.3,.5) -- (E2);
\end{tikzpicture}
\caption[caption]{\tabular[t]{@{}l@{}}A binding with \\ 4 adjacent pages\endtabular}
\label{binding_fig}
\end{minipage}%
\begin{minipage}{.7\textwidth}
\centering
\includegraphics[width=.8\textwidth]{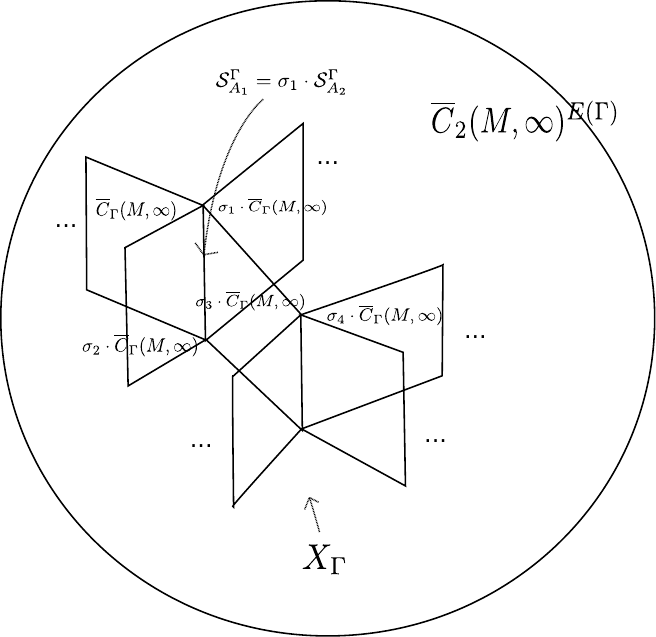}
\caption{An illustration of $X_\Ga$}
\label{X_fig}
\end{minipage}
\end{figure}

\item[3)] Denote by $S(\pi)\subset X_\Ga$ the union of all type 3 $\ov{\cS}^\Ga_A$. Then the boundary of $X_\Ga(\pi)-T_2(\pi)$ is contained in $S(\pi)$. We show that the cup product
$\Om_\Ga:=\cup_{e\in E(\Ga)}f_e^*q^*\Om$
actually lands in the relative group $H^{*}(X_\Ga(\pi),S(\pi))$. 
\end{enumerate}
Since $X_\Ga(\pi)$ can be roughly thought of as a manifold with boundary $S(\pi)$ and bindings with pages summing up to 0, $H^*(X_\Ga(\pi),S(\pi))$ is not trivial and $\Om_\Ga$ contains all the information we need. We push it forward to $B$ (cohomology push-forward is defined in Section \ref{cohomologypushforward_sec} using Leray-Serre spectral sequence) to obtain a class in $H^*(B)$, which is the desired Kontsevich's characteristic class with parameter $\Ga$. 

\begin{rmk}\label{KTK_rmk}
Our approach above to modify $\ccE{V(\Ga)}$ is very similar to that of \cite{KT} and \cite{Koytcheff}.
The similarities include: considering the space $\cq$ and view the propagator as a cohomology class of it; 
using a smaller configuration space to deal with boundary strata of type (1); and using a gluing construction to deal with boundary strata of type (2) and (4). 
The main difference is about the construction of the ``smaller configuration space''. 
In \cite{KT,Koytcheff}, 
it is constructed by blowing up less diagonals in $M^{V(\Ga)}$ (denote by $M$ the 3-manifold whose configuration space is to be constructed), thus the ``smaller configuration space'' obtained is a smooth manifold, and when doing the gluing, how the corners glue together nicely is analyzed. 
Here, we construct the ``smaller configuration space'' in a much simpler way (the closure of the image of the forgetful maps), and define the glued space $X_\Ga$ in a very simple way as well. 
As a trade-off, the downside of this approach is that these spaces are not smooth, therefore we have to make an effort to analyze that the not-necessarily-smooth part is of at least codimension-2. 
Yet it is because the definition of $X_\Ga(\pi)$ is so simple in our approach---essentially, only the topology of $\ov{C}_2(\pi)$ is used---that we are able to prove Theorem \ref{main_thm} which only involves the 2-point configuration space instead of the $n$-point configuration spaces for bigger $n$. 
\end{rmk}

\subsection{Some auxiliary notation}
For a set $S$ let $|S|$ denote the number of elements in $S$. 
For a set $I$ and a space $X$, denote $X^I=\prod_{i\in I}X$; for $n\in\Z^{>0}$, denote $X^n=\underbrace{X\sx X\ldots\sx X}_{\tn{n-times}}$.  Denote by $\De_{\tn{big}}\subset X^I$ the big diagonal. For a map $f:X\to Y$, denote 
$$f^I:X^I\to Y^I,\qquad f^I\big((x_i)_{i\in I}\big)=(f(y_i))_{i\in I}.$$
By the real blow-up of a smooth manifold $X$ along a submanifold $Y$ we mean the oriented blow-up: replacing $Y$ by the sphere normal bundle of $Y$ in $X$. 

We will be using \v{C}ech cochians and \v{C}ech cohomology. For a space $X$ with an open cover $\cU$, and a coefficient ring $R$, we write $\ch{C}^*_{\cU}(X;R)$ for the $R$-module of \v{C}ech cochains on $X$ with respect to $\cU$, valued in the constant sheaf $R\times X$. 
By a skew-symmetric \v{C}ech cochain $\al\in\ch{C}^n_{\cU}(X;R)$ we mean $\al$ is such that $$\al(U_0,\ldots,U_i,U_{i+1},\ldots,U_n)=-\al(U_0,\ldots,U_{i+1},U_i,\ldots,U_n)$$
for all $U_0,\ldots,U_n\in\cU$ and $0\le i<n$. 
If $A\subset X$ is closed, we denote $\ch{C}^*_{\cU}(X,A;R)=\ch{C}^*_{\cU}(X;R|_{X-A})$, where $R|_{X-A}$ is the sheaf obtained by restricting $R\times X$ to $X-A$ and then extending it by 0 to $X$. 
So, $\ch{C}^*_{\cU}(X,A;R)$ can be identified with a sub-module of  $\ch{C}^*_{\cU}(X;R)$: it consists of \v{C}ech cochains that vanish on $A$.
Denote by $H^*_{\cU}(X;R)$ and $H^*_{\cU}(X,A;R)$ the cohomology of $\ch{C}^*_{\cU}(X;R),\ch{C}^*_{\cU}(X,A;R)$, respectively. 

If $\cU'$ is a refinement of $\cU$, then there is a well-defined map $H^*_{\cU}(X;R)\to H^*_{\cU'}(X;R)$ (resp. $H^*_{\cU}(X,Y;R)\to H^*_{\cU'}(X,Y;R)$), independent of the choice of refinements. 
Write $H^*(X;R)$ (resp. $H^*(X,Y;R)$) for the cohomology of $X$ (resp. the pair $(X,Y)$) with $R$ coefficients; it is the direct limit of $H^*_{\cU}(X;R)$ (resp. $H^*_{\cU}(X,Y;R)$) as $\cU$ gets more and more refined. 
In Sections \ref{space_sec} and \ref{propagator_sec}, the coefficient ring $R$ is assumed to be $\Z$ or $\R$, and is often suppressed from notation; in Section \ref{equiv_sec}, $R=\R$. 

We will also be using compactly-supported \v{C}ech cohomology (and very occasionally \v{C}ech cohomology with some other support $\Phi$; usually $\Phi$ is the collection of compact subsets inside a subspace). 
The last two paragraphs hold in these cases as well. The notation in the compactly-supported case will be an added subscript ``$c$'', e.g., $\ch{C}^*_{c,\cU}(X;R)$, $H^*_{c,\cU}(X;R)$, $H^*_{c}(X;R)$. (And the notation in the case with support $\Phi$ will be an added subscript ``$\Phi$''.)

We follow \cite[\S2.4]{Steenrod} for the definition of fiber bundles. When we say $p:B\to X$ is a fiber bundle with fiber $Y$ and structure group $G$, we mean a fiber bundle as in \cite[Definition 2.3]{Steenrod}, with given (maximal) coordinate functions that we do not explicitly mention. Given a $G$-action on some other space $Y'$ (resp. and a $G$-equivariant map $f\!:Y'\to Y$), by {\it the associated bundle of $p$ with fiber $Y'$} we mean a fiber bundle $p':B'\to X$ with fiber $Y'$ and structure group $G$ (resp. and a $G$-bundle map (see \cite[Definition 2.5]{Steenrod} for definition) $\tilde{f}:B'\to B,p\circ\tilde{f}=p'$) built up by gluing coordinate charts in the same way as $p$. Then $p'$ (resp. $(p',\tilde{f})$) is unique up to equivalence.

In Section \ref{thm_sec} we defined a smooth $(M,\i)$-bundle to be a fiber bundle with typical fiber $(M,\i)$ and structure group $\Diff_+(M,N_\i)$. If $B$ is a smooth manifold, then it is the same as saying (by smooth approxiamation theorems we can make the transition maps 
smooth) 
$E\xra{\pi}B$ is a smooth submersion between smooth manifolds, with a smooth section $s_\i\!:\!B\to E$, such that for all $b\in B$, $(\pi^{-1}(b),s_\i(b))$ is diffeomorphic to $(M,\i)$, together with
a neighborhood $\wt U\subset E$ of $s_\i(B)$ and a trivialization $t\!:B\sx (U,\i)\xra{\sim}(\wt U,s_\i)$ where $U\subset M$ is some neighborhood of $\i$. 

Throughout this article we assume the reader is familiar with the Fulton-MacPherson compactification -- having the picture in their mind; see \cite{FM} or \cite{Lescop} for reference. 
Some familiarity with the original definition of Kontsevich's characteristic classes would help (see any one of  \cite{Kontsevich}, \cite{Lescop}, \cite[Section 2]{WatanabeAdd} for reference), but it will not be needed until Section \ref{equiv_sec}. 

\subsection{Acknowledgments}
I thank Peter Kronheimer for many very helpful discussions throughout the course of the work. I thank Fabian Gundlach for helpful discussions regarding Example \ref{eg} and for providing the proof of Lemma \ref{lattice_lmm}. 
I would also like to thank the anonymous referees who carefully read this paper and made many valuable comments and suggestions.

\section{Information from the graph}
\label{graph_sec}

Say a graph is {\it directed} if its edges have directions, {\it ordered} if both of its vertex set and edge set are ordered. If a vertex or edge is the $i$-th one in the ordering, we call $i$ its {\it label}. Given such a graph $\Ga$, we denote by $V(\Ga)$ its vertex set and $E(\Ga)$ its edge set.
Denote by $e^\Ga_{i}$ the $i$-th edge of $\Ga$, $v^\Ga_{i}$ the $i$-th vertex of $\Ga$. Conversely, given an edge $e$ or vertex $v$ of $\Ga$, we denote by $o_\Ga(e),o_\Ga(v)$ its label. (So $o_\Ga(e^\Ga_{i})=o_\Ga(v^\Ga_{i})=i$.) For an edge $e$ of $\Ga$, denote by $v_+(e),v_-(e)$ the output and input vertex connected to $e$, respectively. 

Suppose $\Ga_1,\Ga_2$ are directed, ordered graphs, $\al:\Ga_1\to\Ga_2$ an undirected, unordered graph isomorphism, then we denote by $\sgn(\al,\vertex),\sgn(\al,\edge)\in\{+1,-1\}$ the permutation signs of $\al$ on the set of vertices and edges, respectively, and denote \begin{gather*}\sgn(\al,\to)=(-1)^{\tn{number of edges whose direction is reversed by }\al}.\\
\sgn_d(\al):=\begin{cases}\sgn(\al,\edge),\quad \tn{ for }d\in\Z \tn{ even};\\
\sgn(\al,\vertex)+\sgn(\al,\to),\quad \tn{ for }d\in\Z \tn{ odd}. 
\end{cases}
\end{gather*}
For a graph $\Ga$, define $\Au(\Ga)$ to be the group of automorphisms of $\Ga$ as an undirected, unlabeled graph. Define
$$|\Au(\Ga)|^\pm_d=\sum_{\al\in\Au(\Ga)}\sgn_d(\al)\in\Z.$$

For a set $I$, denote by $\tS_I$ the group of bijections from $\sqcup_{i\in I}\{i^+,i^-\}$ to itself satisfying that, if $i^+$ is mapped to $j^+$ or $j^-$, then $i^-$ is also mapped to $j^+$ or $j^-$. There is an obvious map from $\tS_I$ to $S_I$, the symmetry group of $I$ as a set. For $\si\in\tS_I$, we denote by $\sgn(\si)$ the sign of its image in $S_I$; denote $\sgn'(\si):=(-1)^{|\{i\in I|\si(i^+)=i^-\}|}$. 

Given an element $\al\in\Au(\Ga)$, we denote by $\al_V\in S_{V(\Ga)}$ and $\al_E\in S_{E(\Ga)}$ the permutations of the vertices and edges induced by $\alpha$, respectively.  
Define 
\begin{gather*}
\psi_\Ga:\Au(\Ga)\lra\tS_{E(\Ga)},\qquad
\psi_\Ga(\al)(e^+)=\al_E(e)^\pm\ \tn{ if }\ \al_V(v_+(e))=\al_V(v_\pm(e)).
\end{gather*}
Intuitively, one can think of $e^\pm$ as the two half-edges of $e$. 

If $\Ga$ has no isolated vertex, this is an injective group homomorphism. 

\subsection{Quick review of graph homology}
\label{graph_homology_sec}

Let $\fG$ be the free abelian group generated by directed, ordered graphs that are also non-empty, connected and such that every vertex is at least trivalent. Define two equivalence relations on $\fG$, $\eqodd$ and $\eqeven$, generated by the following:
for $\Ga_1,\Ga_2$ directed, ordered graphs, if there exists an isomorphism $\phi:\Ga_1\to\Ga_2$ as unoriented, unordered graphs, then 
$$\Ga_1\eqodd\big(\sgn(\phi,\vertex)\cdot\sgn(\phi,\to)\big)\,\Ga_2,\quad \Ga_1\eqeven\sgn(\phi,\edge)\,\Ga_2.$$

For a directed, ordered graph $\Ga$ and an edge $e$ of $\Ga$, we define $\Ga/e$ to be the graph obtained from $\Ga$ by contracting $e$, 
with edge directions unchanged, vertices ordered in the same way as $\Ga$, except for the new vertex, 
now put in the very front; edges ordered in the same way as $\Ga$; as below. 
\vspace{-.7cm}
\begin{center}
\begin{tikzpicture}[scale=.6]
\draw [-stealth] (0,0) -- (.5,.5);
\draw (.5,.5) -- (1,1);
\draw [-stealth] (0,0) -- (-.5,0);
\draw (-.5,0) -- (-1,0);
\draw (0,0) -- (-.1,-.5);
\draw [stealth-] (-.1,-.5) -- (-.2,-1);
\draw [-stealth] (1,1) -- (.9,1.5);
\draw (.9,1.5) -- (.8,2);
\draw (1,1) -- (1.5,1);
\draw [stealth-] (1.5,1) -- (2,1);
\node at (.3,.6) {$e$};
\node at (0.8,2.2) {...};
\node at (2.3,1) {...};
\node at (-1.3,0) {...};
\node at (-.2,-1.1) {...};

\draw [|->] (2.5,0.5) -- (3.3,0.5);

\draw [-stealth] (5.2,0.5) -- (4.7,0.5);
\draw (4.7,0.5) -- (4.2,0.5); 
\draw [-stealth] (5,-.5) -- (5.1,0);
\draw (5.1,0) -- (5.2,.5);
\draw [-stealth] (5.2,.5) -- (5.1,1);
\draw (5.1,1) -- (5,1.5);
\draw [-stealth] (6.2,.5) -- (5.7,.5);
\draw (5.7,.5) -- (5.2,.5);
\node at (3.9,.5) {...};
\node at (5,-.6) {...};
\node at (5,1.6) {...};
\node at (6.5,.5) {...};
\end{tikzpicture}
\end{center}

Define $\Z$-linear maps $$\de_{\odd}:\fG/\!\eqodd\lra\fG/\!\eqodd,\qquad \de_{\even}:\fG/\!\eqeven\lra\fG/\!\eqeven$$ 
induced by: for a directed, ordered graph $\Ga$, 
$$\de_{\odd}(\Ga):=\sum_{e\in E(\Ga)}(-1)^{o_\Ga(v_+(e))-o_\Ga(v_-(e))}\,\Ga/e,\qquad \de_{\even}(\Ga):=\sum_{e\in E(\Ga)}(-1)^{o_\Ga(e)}\,\Ga/e.$$
It can easily be seen that $\de_\odd^2,\de_\even^2=0$. Graph homology is defined to be the homology of $(\fG/\!\eqodd,\de_\odd)$ and $(\fG/\!\eqeven,\de_\even)$. 

The main takeaway we need is the following statement, which follows directly from definition: if $\Ga_1+\ldots+\Ga_n\in\fG, [\de_\odd(\sum_{i=1}^n\Ga_i)]_{\eqodd}\!=0$ or $[\de_\even(\sum_{i=1}^n\Ga_i)]_{\eqeven}\!=0$ (``[\ ]'' denotes taking equivalence class with respect to $\eqodd,\eqeven$), respectively, then there exists an (often not unique) pairing between the edges of $\Ga_1,\ldots,\Ga_n$, such that, if $e_{\Ga_a,i},e_{\Ga_b,j}$ is a pair, then there is an undirected, unordered graph isomorphism 
$$\alpha:\Ga_a/e^{\Ga_a}_i\lra\Ga_b/e^{\Ga_b}_{j}$$
such that 
$$\sgn(\al,\vertex)\cdot\sgn(\al,\to)\cdot(-1)^{o_{\Ga_a}(v_+(e^{\Ga_a}_i))-o_{\Ga_a}(v_-(e^{\Ga_a}_i))+o_{\Ga_b}(v_+(e^{\Ga_b}_{j}))-o_{\Ga_b}(v_-(e^{\Ga_b}_{j}))}=-1$$
or 
$$\sgn(\al,\edge)\cdot(-1)^{o_{\Ga_a}(e^{\Ga_a}_i)+o_{\Ga_b}(e^{\Ga_b}_{j})}=-1,$$
respectively.

\subsection{Information from the graph}
\label{fixagraph_sec}

Let $d\ge3$ be an integer. Denote by $\bar d\in\{\odd, \even\}$ the parity of $d$. Let $\Ga$ be a connected, directed, ordered graph such that all vertices are at least trivalent, 
and $[\de_{\bar d}(\Ga)]_{\sim_{\bar d}}=0$. (It is easy to see that all the arguments in this section work if we assume $\Ga$ is a formal sum of graphs instead.) 
For example, the tetrahedron graph -- the graph with 4 vertices and 1 edge between each pair of vertices, arbitrarily directed and ordered -- satisfies these conditions. 

For $V'\subset V(\Ga)$, we denote by $\Ga_{V'}$ the subgraph of $\Ga$ spanned by the vertices in $V'$, and by $\Ga/V'$ the graph obtained from $\Ga$ by contracting $\Ga_{V'}$ to one single vertex $[V']_v$. The order of edges in $\Ga_{V'},\Ga/V'$ and the order of vertices in $\Ga_{V'}$ are the same as in $\Ga$. The order of vertices in $\Ga/V'$ is defined as: $[V']_v$ is the first one, and the rest are ordered as in $\Ga$. 

Suppose $A\subset\{\i\}\sqcup V(\Ga)$, $|A|\ge2$.  Define 
\[\Ga_A=\begin{cases}\Ga_{A}, \tn{ if }\i\notin A,\\
\Ga/(A\!-\!\i), \tn{ if } \i\in A;
\end{cases}\quad
\Ga/\Ga_A=\begin{cases}
\Ga/A, \tn{ if }\i\notin A,\\
\Ga_{A-\i},\tn{ if } \i\in A.
\end{cases}\]

We say $A$ is of 
\begin{itemize}
\vspace{-.5cm}
\itemsep=0em
\item type 1, if $\Ga_A$ has a zero-valent or univalent vertex, and $|A|\ge3$ or $\Ga_A$ has no edge; 
\item type 2, if $\Ga_A$ has a bivalent vertex but no uni- or zero-valent vertex; 
\item type 3, if all vertices of $\Ga_A$ are at least trivalent,
\item type 4, if $\Ga_A$ has exactly 2 vertices with 1 edge connecting them. 
\end{itemize}
\vspace{-.45cm}
Notice that since all vertices of $\Ga$ are at least trivalent, $\i\in A\implies A$ is of type 3. 

Suppose $A$ as above is of type 2. We fix $v_A$ a bivalent vertex in $A$. Denote by $e_A^1,e_A^2\in E(\Ga)$ the two edges connected to $v_A$. Define
$\si_A\in\tS_{E(\Ga)}$ as follows: 
\[\begin{cases}
\si_A(e^\pm)=e^\pm \ \tn{ if }\  e\neq e_A^1,e_A^2;\\
\si_A({e^1_A}^\pm)={e^2_A}^\mp \ \tn{ if }e^1_A,e^2_A\tn{ both start or both end at }v_A;\\
\si_A({e^1_A}^\pm)={e^2_A}^\pm \ \tn{ if one of }e^1_A,e^2_A\tn{ starts at }v_A,\tn{ the other ends at }v_A. 
\end{cases}\]

Now we look at $A$ of type 4 above. Since $[\de_{\bar d}(\Ga)]_{\sim_{\bar d}}=0$, there exists a pairing (given $\Ga$, let us choose it once and for all, and call it {\it $\Ga$-pairing}) between type 4 $A$'s such that, if $A_1,A_2$ are paired (call them a {\it $\Ga$-pair}), then there exists an undirected, unordered graph isomorphism $\al_{A_1A_2}\!:\Ga/A_1\to\Ga/A_2$, as in the last paragraph of Section~\ref{graph_homology_sec}. Denote by $e_1,e_2$ the edge in $\Ga$ between the two vertices in $A_1,A_2$, respectively. We call $e_1,e_2$ a {\it $\Ga$-pair} as well, and denote $\al_{A_1A_2}$ also by $\al_{e_1e_2}$. We define $\si_{e_1e_2}\in\tS_{E(\Ga)}$ to be: 
\[\begin{cases}
\si_{e_1e_2}({e_1}^\pm)={e_2}^\pm;\\
\si_{e_1e_2}({e^\Ga_i}^\pm)={e^\Ga_j}^\pm, \tn{ if }\al_{A_1A_2}(e^{\Ga}_i)=e^{\Ga}_j, \tn{ preserving direction;} \\
\si_{e_1e_2}({e^\Ga_i}^\pm)={e^\Ga_j}^\mp, \tn{ if }\al_{A_1A_2}(e^{\Ga}_{i})=e^{\Ga}_j, \tn{ reversing direction.} 
\end{cases}\]
Note that we implicitly identified edges in $E_\Ga\!-\!\{e_1\}$ (resp. $E_\Ga\!-\!\{e_2\}$) with edges in $\Ga/A_1$ (resp. $\Ga/A_2$). Evidently $\si_{e_2e_1}=\si^{-1}_{e_1e_2}$. 

If $\Ga$ is trivalent, then $[A_1]_v,[A_2]_v$ are the only vertices of valency~4 in $\Ga/A_1,\Ga/A_2$, respectively. So $\al_{A_1A_2}$ must map $[A_1]_v$ to $[A_2]_v$. 

\section{Defining various spaces, all having a \texorpdfstring{$\cG$}{G}-action}\label{space_sec}


Let $(M,\i)$ be as in Section \ref{thm_sec}. Denote by $\De\subset M\sx M$ the diagonal.
Denote by $C_I(M)=M^I\!-\!\De_{\tn{big}}$ the configuration space of distinct marked points on $M$ labeled by $I$. If $I$ is ordered, let $C_I(M)$ be oriented by the product orientation on $M^I$. Denote by $\ov{C}_I(M)$ its Fulton-MacPherson compactification \cite{FM}. It can be constructed from $M^I$ by a sequence of real blow ups along various diagonals. For example, $\ov{C}_{\{1,2\}}(M)$ is the real blow up of $M\sx M$ along $\De$. The space $\ov{C}_I(M)$ has the structure of a smooth manifold with boundaries and corners. For $I'\subset I$ we have a smooth forgetful map $f_{I'}:\ov{C}_I(M)\lra\ov{C}_{I'}(M)$ lifting the map 
$$M^I\lra M^{I'},\quad (x_i)_{i\in I}\lra (x_i)_{i\in I'}.$$
If $I'$ has only one element $i$, we also denote $f_i=f_{I'}$. 

Denote $C_I(M,\i)=\{(x_i)_{i\in I}\in C_I(M)|\,x_i\neq\i,\forall i\}$, oriented in the same way as $C_I(M)$. 
Denote by $\ccM{I}$ its Fulton-MacPherson compactification, which is defined to be the preimage of $\i$ under the forgetful map $f_{\{*\}}:\ov{C}_{I\sqcup\{*\}}(M)\to M.$ 
For simplicity we write $\ccM{n}:=\ccM{\{1,\ldots,n\}}$. For example, $\ccM{2}$ is described in Section \ref{thm_sec}. Define $\tau:\ccM{2}\to\ccM{2}$ to be the map swapping the two marked points.  

For $A\subset\{\i\}\sqcup I$, $|A|\ge2$, we denote by $\mathring\cS_A\subset\ccM{I}$ the (open) boundary stratum corresponding to that the marked points with labels in $A$ coincide. Denote by $\ov\cS_A$ its closure in $\ccM{I}$. 

\subsection{Defining various spaces}\label{defspace_sec}

Recall in Section \ref{thm_sec} we defined
$$\cG:=\big\{(\tl{g},g)\in\Homeo(\ccM{2})\sx\Homeo_+(M,N_\i)\big|\ g\circ f_\pm=f_\pm\circ\tl g\}.$$
It is easy to see that the action of $\cG$ commutes with the point-swapping map $\tau$. 
We call the action of $\cG$ on $\ccM{2}^I$ by acting simultaneously on every factor the {\it diagonal action}. All the actions of $\cG$ we talk about below are action by homeomorphisms. 

Let $\Ga$ be a graph as in Section~\ref{fixagraph_sec}. Assume $\Ga$ is trivalent. 
For an edge $e$ of $\Ga$, we denote by $f_e:\ccM{V(\Ga)}\to\ccM{2}$ the forgetful map lifting 
$$M^{V(\Ga)}\lra M^{2},\quad(x_v)_{v\in V(\Ga)}\lra(x_{v^\Ga_-(e)},x_{v^\Ga_+(e)}).$$

\begin{dfn} Denote $f_\Ga=(f_e)_{e\in E(\Ga)}\!:\!\ccM{V(\Ga)}\to\ccM{2}^{E(\Ga)}$. Define \begin{gather*}C_\Ga(M,\i)\subset\ccM{\Ga}\subset\ccM{2}^{E(\Ga)}\\
\ccM{\Ga}=\tn{image}(f_\Ga),\ C_\Ga(M,\i)=f_\Ga(C_{V(\Ga)}(M,\i)).
\end{gather*}
\end{dfn}
Since $\Ga$ is connected, $f_\Ga|_{C_{V(\Ga)}(M,\i)}$ is an embedding. Thus it gives a diffeomorphism $C_{V(\Ga)}(M,\i)\to C_\Ga(M,\i)$. Since $V(\Ga)$ is ordered, this also gives $C_\Ga(M,\i)$ an orientation. 
Since $C_\Ga(M,\i)$ can be written as 
$$\big\{(z_e)_{e\in E(\Ga)}\in\ccM{2}^{E(\Ga)}\big|\,\forall e,e'\in E(\Ga),\forall s,s'\in\{+,-\}, f_s(z_e)=f_{s'}(z_{e'})\text{ iff }v_s(e)=v_{s'}(e')\big\},$$
it is invariant under the diagonal $\cG$-action.


\begin{lmm}\label{1_lmm}
$\ccM{\Ga}$ is the closure of $C_\Ga(M,\i)$ in $\ccM{2}^{E(\Ga)}$. 
\end{lmm}
\begin{proof}
Since $\ccM{V(\Ga)}$ is compact, $f_\Ga(\ccM{V(\Ga)})$ is closed and thus contains the closure of $C_\Ga(M,\i)$. On the other hand, 
$C_{V(\Ga)}(M,\i)$ is dense in $\ccM{V(\Ga)}$, so $C_\Ga(M,\i)=f_\Ga\big(C_{V(\Ga)}(M,\i)\big)$ is dense in $f_\Ga(\ccM{V(\Ga)})$. 
\end{proof}

Abusing notation, we denote by the projection 
$\ccM{\Ga}\to \ccM{2}$ to the $e$-th coordinate still by $f_{e}$ and denote the inclusion map $\ccM{\Ga}\to\ccM{2}^{E(\Ga)}$ still by $f_\Ga$. It follows directly from Lemma~\ref{1_lmm} that $\ccM{\Ga}$ is invariant under the diagonal action of $\cG$ on $\ccM{2}^{E(\Ga)}$.

\begin{dfn}\label{Saction_dfn}
For $I$ an ordered set, define an action $\phi$ of $\tS_{I}$ on $\ccM{2}^I$ by diffeomorphisms: for $\si\in\tS_I$, 
\begin{gather*}
\phi(\si):\ccM{2}^I\lra\ccM{2}^I,\qquad
\phi(\si)(z_i)_{i\in I}=\bigg(\begin{cases}z_j,\tn{ if }\si(i^\pm)=j^\pm\\\tau(z_j),\tn{ if }\si(i^\pm)=j^\mp
\end{cases}\bigg)_{i\in I}.
\end{gather*}
\end{dfn}
It is clear from definition that every $\phi(\si)$ is equivariant with respect to the diagonal action of $\cG$ on $\ccM{2}^I$. 

\begin{dfn}
$$\tcM:=\bigcup_{\si\in\tS_{E(\Ga)}}\phi(\si)\big(\ccM{\Ga}\big)\subset\ccM{2}^{E(\Ga)}$$
We still denote the inclusion map $\tcM\to\ccM{2}^{E(\Ga)}$ by $f_\Ga$ and its $e$-th factors by $f_{e}$. 
\end{dfn}
It follows from the $\cG$-invariance of $\ccM{\Ga}$ and the $\cG$-equivariance of each $\phi(\si)$ that $\tcM$ is invariant under the diagonal action of $\cG$ on $\ccM{2}^{E(\Ga)}$. We therefore define the action of $\cG$ on $\tcM$ to be the restriction of the diagonal action. The maps $f_{e}$ are clearly $\cG$-equivariant. 

It seems that $\tcM$ (and its bundle version) is the appropriate space that accommodates Kontsevich's characteristic classes. It is defined this way -- taking the $|\wt{S}_{E(\Ga)}|$ copies of $\ccM{\Ga}$ inside of $\ccM{2}^{E(\Ga)}$ and taking their union -- because this makes the type 2 and 4 boundary strata of the many copies of $\ccM{\Ga}$ cancel with each other; see Figure \ref{X_fig}. On the other hand, being a subspace of $\ccM{2}^{E(\Ga)}$ automatically makes it (and all its subspaces) metrisable, as a topological space, which is needed for the various arguments regarding covering dimension and \v{C}ech cohomology below. 

We first make a remark that in the definition of $X_\Ga$ as a union, the ``main stratum'' parts either coincide or do not intersect. This is the content of Lemma \ref{chambers_lmm} below. 

Recall $\Au(\Ga)$ is the group of automorphisms of $\Ga$ as an undirected, unlabeled graph. An element $\al\in\Au(\Ga)$ consists of permutations $\al_V\in S_{V(\Ga)}$ and $\al_E\in S_{E(\Ga)}$. 
Denote by $\ga:\Au(\Ga)\lra\Diff\big(\ccM{V(\Ga)}\big)$ the action of $\Au(\Ga)$ on $\ccM{V(\Ga)}$ by permuting marked points according to $\al_V$, namely, $\ga(\al)$ lifts the map 
$$M^{V(\Ga)}\lra M^{V(\Ga)},\quad (x_v)_{v\in V(\Ga)}\lra(x_{\al_V(v)})_{v\in V(\Ga)}.$$
Then (recall $\psi_\Ga:\Au(\Ga)\to\tS_{E(\Ga)}$ defined right before Section \ref{graph_homology_sec})
\begin{equation}\label{permutation_eqn}
f_\Ga\circ\ga(\al)=\phi(\psi_\Ga(\al))\circ f_\Ga.
\end{equation}

\begin{lmm}\label{chambers_lmm}
For $\si\in\tS_{E(\Ga)}$, if $\si\notin\tn{image}(\psi_\Ga)$, then $\phi(\si)\big(C_\Ga(M,\i)\big)\cap \ccM{\Ga}=\eset$; if $\si=\psi_\Ga(\al)$ for some $\al\in\Au(\Ga)$, then $\phi(\si)\big(C_\Ga(M,\i)\big)=C_\Ga(M,\i)$ and $\phi(\si)$ changes its orientation by $\sgn(\al,\vertex)^d$. 
\end{lmm}
\begin{proof}
This lemma is intuitively quite simple: an edge $e$ of $\Gamma$ contains the information of the two vertices $v_\pm(e)$; so, since every vertex of $\Gamma$ is connected to some edge, all the vertices of $\Gamma$ (and hence $\Gamma$ itself) can be recovered from its edges. 
We spell out the details below: 

For simplicity we prove the lemma in the case $\Ga$ has no repeated edges; it can easily be generalized to the other cases as well. 
Given an element $z=(z_e)_{e\in E(\Ga)}\in\ccM{2}^{E(\Ga)}$, we define its {\it set of vertex positions} $\bigcup_e\big\{f_+(z_e),f_-(z_e)\big\}\subset M$. It does not change under the $\phi(\si)$ action which only permutes factors. 
For $x\!\in\!\ccM{V(\Ga)}$, we denote by
$\{f_v(x)\}_{v\in V(\Ga)}$ the set of vertex positions of $f_\Ga(x)$, 
which has exactly $|V(\Ga)|$ (distinct) elements if $x\in C_{V(\Ga)}(M,\i)$ and less otherwise. Thus, if we suppose $\phi(\si)\big(C_\Ga(M,\i)\big)\cap \ccM{\Ga}\neq\eset$, and
if $x,y\in C_{V(\Ga)}(M,\i)$ are such that $f_\Ga(x)=\phi(\si)f_\Ga(y)$, then there is a permutation $\al_V\in S_{V(\Ga)}$ such that $f_{\al_V(v)}(y)=f_v(x)$. Since $\Ga$ has no repeated edge, there is a unique $\al\in\Au(\Ga)$ with $\al_V$ as given. So $\ga(\al)(x)=y$. So, 
$$f_\Ga(x)=\phi(\si)f_\Ga(y)=\phi(\si)\big(f_\Ga(\ga(\al)(x))\big)=\phi(\si)\circ\phi(\psi_\Ga(\al))(f_\Ga(x))=\phi(\si\psi_\Ga(\al))(f_\Ga(x)).$$
Since $\phi(\si\psi_\Ga(\al))$ acts by permuting factors of $\ccM{2}^{E(\Ga)}$ composed with $\tau$'s, and no two factors of $f_\Ga(x)$ are the same even modulo $\tau$, because $\Ga$ has no repeated edges, we must have $\si\psi_\Ga(\al)=\id$. So $\si=\psi_\Ga(\al^{-1})$. 
This proves the lemma except for the statement about orientation. The orientation statement is straightforward. 
\end{proof}

\begin{dfn}\label{cluster_dfn}
For $A\subset\{\i\}\sqcup V(\Ga),\,|A|\ge2$, 
Define
\begin{align*}
C_{V(\Gamma):A}(M,&\infty)=\\
&\begin{cases}
\{(x_v)_{v\in V(\Ga)}\in (M-\{\infty\})^{V(\Ga)} \big|\, x_v=x_w \tn{ iff } (v,w\in A \tn{ or }v=w)\}, \tn{ if }\infty\notin A, \\
p_{V(\Ga)-A}^{-1}(C_{V(\Ga)-A}(M,\infty))\cap p_{A}^{-1}(\infty,\ldots,\infty), \tn{ if }\infty\in A,
\end{cases}
\end{align*}
where $p_{V(\Ga)-A}:M^{V(\Ga)}\to M^{V(\Ga)-A}$ and $p_A: M^{V(\Ga)}\to M^{A-\{\infty\}}$ are the projections. 
Then we have the following commutative diagram 
$$
\begin{tikzcd}
\mathring{\cS}_A \ar[r,"f_\Ga"] \ar[d] & \ccM{2}^{E(\Ga)} \ar[d,"(f_-{,}f_{+})^{E(\Ga)}"] \\
C_{V(\Ga):A}(M,\infty) \ar[r,"f'_\Ga"] & (M^2)^{E(\Ga)}.
\end{tikzcd}
$$
Define 
$$\cS^\Ga_A=\big(((f_-,f_+)^{E(\Ga)})^{-1}f'_\Ga(C_{V(\Ga):A}(M,\infty))\big) \cap \ccM{\Ga}.$$
Denote by $\ov\cS^\Ga_{A}$ the closure of $\cS^\Ga_{A}$. 
For an edge $e$ of $\Ga$, denote $\cS^\Ga_e:=\cS^\Ga_{\{v^\Ga_+(e),v^\Ga_-(e)\}}$. 
\end{dfn}

Intuitively, elements of $\cS^\Ga_A$ are configurations of points such that those in $A$ all coincide in $M$ and no other two points coincide. 
Notice that $\ov\cS_A=f_\Ga^{-1}(\ov\cS^\Ga_A),\ \mathring{\cS}_A\subset f_\Ga^{-1}(\cS^\Ga_A)\subset\ov\cS_A$, and the inclusions are often strict. 

Recall that (see Figure \ref{X_fig}) in $X_\Ga$ as a union of many copies of $\ccM{\Ga}$ inside of $\ccM{2}^{E(\Ga)}$, we want those codimension-1 strata, $\cS^\Ga_A$, of type 2 or 4 in the various copies of $\ccM{\Ga}$ to glue together to form bindings. 
For example, if the binding has only two pages, it should locally be two topological manifolds with boundary gluing together along their boundaries, forming a topological manifold. 
This is only true if we remove the non-nice, ``singular'' parts of those $\cS^\Ga_A$ and only leave the nice part, which is what we call $\rrsa$ in Definition \ref{ringGa_dfn} below. 

\begin{dfn}\label{ringGa_dfn}
For $A\subset\{\i\}\sqcup V(\Ga),|A|\ge2$, define $\mathring\cS^\Ga_A\subset\cS^\Ga_A$ to be the set of points such that locally $\ccM{\Ga}$ is a topological manifold with boundary, i.e., 
\begin{align*}
\mathring\cS^\Ga_A:=\Big\{&x\in\cS^\Ga_A\Big|\ \exists\ U\subset\ccM{\Ga}\tn{ open neighborhood of }x,\tn{ and } \\ &\tn{ homeomorphism } \nu\!:\!U\to\R^{d|V(\Ga)|\!-\!1}\!\sx\R^{\ge0}
\tn{ such that }\nu(\cS^\Ga_A\cap U)=\R^{d|V(\Ga)|\!-\!1}\!\sx\{0\}\Big\}.
\end{align*}
Moreover, define 
$$\rrsa=\rsa-\bigcup_{\si\in\tS_{E(\Ga)}}\bigcup_{\begin{subarray}{c}A'\subset\{\i\}\sqcup V(\Ga)\\|A'|\ge2\end{subarray}}\phi(\si)(\ov\cS^\Ga_{A'}-\mathring{\cS}^\Ga_{A'}).$$
\end{dfn}

In Definition \ref{strata_dfn} below we give names to the various parts of $X_\Ga$. The ``good'' part, consisting of the main strata and the ``non-singular'' parts of the type 2 and 4 codimension-1 strata, will function like a topological manifold with bindings as in Figure \ref{X_fig}.
The remaining of $X_\Ga$ can again be divided into two parts: 
$S$, consisting of all $\ov{\cS}^\Ga_A$ of type 3, and $T_2$, consisting of everything else. 
We will show that neither $S$ or $T_2$ causes an issue in our argument, but for different reasons: we will show that $T_2$ is of codimension at least 2 (hence the ``2'' in the name ``$T_2$''), therefore not contributing to an intersection-theoretical argument; although $S$ has a larger dimension, it can be shown (in Section \ref{landinginrelative_subsec}) that the cohomology class we care about is actually in the relative cohomology $H^*(X_\Ga,S)$. 

\begin{dfn}\label{strata_dfn}
Define the following subsets of $\tcM$:  
\begin{gather*}
X_\Ga^\tn{good}=\bigcup_{\si\in\tS_{E(\Ga)}}\phi(\si)\Big(C_\Ga(M,\i)\cup\bigcup_{A\tn{ of type 2 or 4}}\rrsa\Big),
\qquad T_1=\tcM-X_\Ga^\tn{good},\\
S=\bigcup_{\si\in\tS_{E(\Ga)}}\bigcup_{A\tn{ of type 3}}\phi(\si)(\ov{\cS}^\Ga_A),\qquad T_2=\bigcup_{\si\in\tS_{E(\Ga)}}\phi(\si)\bigcup_{A}(\ov\cS^\Ga_A-\rsa).
\end{gather*}
$A$ in the above formulas are subsets of $\{\infty\}\sqcup V(\Gamma)$ such that $|A|\ge2$. 
\end{dfn}
Notice that if $A$ is of type 1, then $\rsa=\eset$ since $\ov{\cS}_A\subset \ccM{V(\Ga)}$ is contracted by $f_\Ga$. So
$T_1\supset T_2\supset T_1\!-\!S$. Also $T_2,T_1,S$ are invariant under the $\tS_{E(\Ga)}$-action $\phi$. And $T_1,T_2$ are closed, since every point in $\rrsa$ (resp. $\rsa$) has a neighborhood lying in $C_\Ga(M,\i)\cup\rrsa$ (resp. $C_\Ga(M,\i)\cup\rsa$). 
It is evident from the above three definitions that for every $A$,  $\cS^\Ga_A,\ov\cS^\Ga_A,\rsa,\rrsa\subset\ccM{\Ga}$ are invariant under the $\cG$-action, and therefore $S,T_1,T_2$ are invariant under the $\cG$-action too. 

Recall the {\it covering dimension} of a topological space $X$ is the biggest integer $N$ satisfying: any open cover of $X$ has a refinement $\cU$ such that if $U_0,\ldots,U_{N+1}\in\cU,U_i\neq U_j\,\forall i,j$, then $U_0\cap\ldots\cap U_{N+1}=\eset$. Below we write $\dim_t(X)$ for the covering dimension of $X$ (``t'' stands for ``topological''). 
Notice that everything we have defined so far are subspaces of $\ccM{2}^{E(\Ga)}$, thus are all metrizable topological spaces.

In 
Section \ref{structureX_sec} below, we show that 
\begin{itemize}
\vspace{-.5cm}
\itemsep=0em
\item $\dim_t(T_2)\le d|V(\Ga)|-2,\ \dim_t(T_1)\le d|V(\Ga)|-1$;
\item $H^{d|V(\Ga)|}(X_\Ga,S;R)$ admits a non-trivial $\cG$-equivariant map to $R$. (Recall $R=\Z$ or $\R$ is the coefficient ring; we take the trivial $\cG$-action on $R$.) This statement is the goal of Section \ref{space_sec} and is all what we need to re-construct Kontsevich's classes. 
In all what follows we will call this map $\rho$. 
\end{itemize}

\subsection{Structure of \texorpdfstring{$\tcM$}{X}}\label{structureX_sec}

We list some basic properties of the covering dimension that will be used repeated later in this subsection. Let $X,X'$ be non-empty metrisable spaces. 
\begin{itemize}
\vspace{-.5cm}
\itemsep=0em
\item If $Y\subset X$ is closed, then $\dim_t(Y)\le\dim_t(X)$. This follows from definition. 
\item If $Y_1,\ldots,Y_n\subset X$ are closed, $\dim_t(Y_i)\le m,\forall\,i$, then $\dim_t(Y_1\cup\ldots\cup Y_n)\le m$; see \cite[Theorem 9-10]{DT}. 
\item $\dim_t(X\sx X')\le\dim_t(X)+\dim_t(X')$; see \cite[Theorem 12-14]{DT}.
\end{itemize}

\begin{lmm}\label{dim_lmm}
Let $f:X\lra Y$ be a smooth map between smooth manifolds (possibly with boundary and corners). Assume $X$ is compact. Denote 
$$A_r=\{x\in X\,\big|\,\text{rank}(d_xf)\le r\},\qquad fA_r=f(A_r).$$
Then $\dim_t(fA_r)\leq r.$ Specifically, $\dim_t(f(X))\le\dim(X)$. 
\end{lmm}
\begin{proof}
This follows from two celebrated theorems. 
By \cite[Corollary on Page 169]{Sard}, there exist countably many charts $(U_i\subset Y,\phi_i:U_i\to\R^n)_{i=1}^\i$ of $Y$ such that $fA_r\subset\bigcup_iU_i$ and the Hausdorff dimension of $\phi_i(U_i\cap fA_r)$ is at most $r$ for all $i$. By \cite{Szpilrajn}, the covering dimension of a metrisable, separable space (which $\phi_i(U_i\cap fA_r)$ is) is no bigger than its Hausdorff dimension. So, $\dim_t(U_i\cap fA_r)\le r$ for all $i$. Since $X$ is compact and $A_r$ is closed in $X$, $fA_r$ is compact, and so it is covered by the charts $U_1,\ldots,U_m$ for some $m$. By shrinking each $U_i$ a little bit, we have closed subsets $\{V_i\subset U_i\subset Y\}_{i=1}^m$ which still cover $fA_r$. And $\dim_t(V_i\cap fA_r)\le r$ for all $i\le m$. Since $V_i\cap fA_r$ are closed subsets of $fA_r$ and $fA_r$ is their union, $\dim_t(fA_r)\le r$. 
\end{proof}

\begin{crl}
$\dim_t(X_\Ga)\leq d|V(\Ga)|,\ \dim_t(S)\leq d|V(\Ga)|-1.$
\end{crl}


\begin{dfn}\label{Vga_dfn}
Given a graph $\Ga'$, define $\ov V_{\Ga'}\subset(S^{d-1})^{E(\Ga')}$ to be the image of the map 
$$f^{\R^d}_{\Ga'}=(f'_e)_{e\in E(\Ga')}:\ov{C}^{\quo}_{V(\Ga')}(\R^d)\lra
(S^{d-1})^{E(\Ga')},$$
where $\ov{C}^\quo_{V(\Ga')}(\R^d)$ is the Fulton-MacPherson space of configurations of $V(\Ga')$-marked points in $\R^d$, modulo translation and scaling (so $\dim(\ov C_{V(\Ga')}(\R^d))\!=\!d|V(\Ga')|\!-\!d\!-\!1$; ``quo'' stands for ``quotient''), and $f'_e$ is the unique map induced from 
$$(\R^d)^{V(\Ga')}-\De_{\tn{big}}\lra S^{d-1},\qquad (x_v)_{v\in V(\Ga')}\lra \frac{x_{v_+(e)}-x_{v_-(e)}}{\big|x_{v_+(e)}-x_{v_-(e)}\big|},$$
i.e., $f'_e$ is the direction between the points marked by the vertices adjacent to $e$. 
Notice the $GL(d)$ action on $\R^d$ induces $GL(d)$-actions on $S^{d-1},\ov{C}^\quo_{V(\Ga')}(\R^d),\ov{V}_{\Ga'}$. 
\end{dfn}

Let $C^\quo_{V(\Ga')}(\R^d)$ be the quotient of $(\R^d)^{V(\Ga')}\!-\!\De_{\tn{big}}$ by translation and scaling. 

\begin{lmm}\label{regularvalueunique_lmm}
If $x,y\in C^\quo_{V(\Ga')}(\R^d), x\neq y$, $f^{\R^d}_{\Ga'}(x)=f^{\R^d}_{\Ga'}(y)$, then $d_xf^{\R^d}_{\Ga'}$, $d_yf^{\R^d}_{\Ga'}$ are not injective. 
\end{lmm}
\begin{proof}
Since $C^\quo_{V(\Ga')}(\R^d)$ is a quotient, let us take representatives $x'\in(\R^d)^{V(\Ga')}$ of $x$ and $y'\in(\R^d)^{V(\Ga')}$ of $y$ such that the marked point labeled by $v^{\Ga'}_1$ is the origin and the marked point labeled by $v^{\Ga'}_2$ has norm 1. It is easy to see that for all $0<\la<1$, if $\la x'+(1\!-\!\la)y'\in (\R^d)^{V(\Ga')}\!-\!\De_{\tn{big}}$, then $f^{\R^d}_{\Ga'}([\la x'+(1\!-\!\la)y'])=f^{\R^d}_{\Ga'}(x)=f^{\R^d}_{\Ga'}(y)$. So the differential of $f^{\R^d}_{\Ga'}$ at $x$ or $y$ is 0 in at least one direction. 
\end{proof}

For $A\subset\{\i\}\sqcup V(\Ga),|A|\ge2$, recall the definitions of $\Ga_A,\Ga/\Ga_A$ from Section~\ref{fixagraph_sec}. 
From the construction of Fulton-MacPherson compactification, we have 
\begin{fact}\label{bundleGaA_lmm}
$\ov\cS_A$ is a fiber bundle over $\ov C_{V(\Ga/A)}(M,\i)$ with fiber $\ov{C}^\quo_{V(\Ga_A)}(\R^d)$ and structure group $SL(d)$. 
\end{fact} 
Notice $\ov{\cS}^\Ga_A=f_\Ga(\ov\cS_A)$. 
\begin{crl}\label{bundleA_crl}
$\ov\cS^\Ga_{A}$ is a fiber bundle over $\ov C_{\Ga/{A}}(M,\i)$ with fiber $\ov V_{\Ga_{A}}$ and structure group $SL(d)$. 
The map $f_\Ga|_{\ov\cS_A}$ is an $SL(d)$-bundle map covering $f_{\Ga/A}\!:\!\ov{C}_{V(\Ga/A)}(M,\i)\to\ov C_{\Ga/{A}}(M,\i)$, which is $f^\R_{\Ga_A}$ on each fiber. 
\end{crl}

Notice that the covering dimension of the total space of a fiber bundle with compact base is no more than the sum of the covering dimensions of its fiber and base: over each chart, the bundle is a product, so this is true; and since covering dimension does not increase when taking finite unions of closed subsets, this is true for the whole total space.

\begin{crl}
Suppose $A\subset\{\i\}\sqcup V(\Ga)$ is of type~1, then $\dim_t(\ov{\cS}^\Ga_{A})\leq d|V(\Ga)|-2$. 
\end{crl}
\begin{proof}
Since $\Ga_A$ has a zero valent vertex $v_0$ or a uni-valent vertex $v_1$ adjacent to $v'_1$, $f^{\R^d}_{\Ga_A}:\ov{C}^\quo_{V(\Ga_A)}(\R^d)\to \ov{V}_{\Ga_A}$ factors through $f_{v_0}:\ov{C}^\quo_{V(\Ga_A)}(\R^d)\to \ov{C}^\quo_{V(\Ga_A)-\{v_0\}}(\R^d)$ which forgets the point labeled by $v_0$, or through $f_{v_1}:\ov{C}^\quo_{V(\Ga_A)}(\R^d)\to S^{d-1}\sx\ov{C}^\quo_{V(\Ga_A)-\{v_1\}}(\R^d)$ which forgets the distance between the point labeled by $v_1,v'_1$, respectively. So, $\dim_t(\ov{V}_{\Ga_A})<\dim(\ov{C}^\quo_{V(\Ga_A)}(\R^d))$. 
\end{proof}

\begin{lmm}\label{small_lmm}
$\dim_t(\ov\cS^\Ga_A-\rsa)\leq d|V(\Ga)|-2$ for all $A\subset\{\i\}\sqcup V(\Ga),|A|\ge2$.  
\end{lmm}
\begin{proof}
Denote by $p_A\!:\!\ov\cS_A\to\ov C_{V(\Ga/A)}(M,\i),p^\Ga_A\!:\!\ov\cS^\Ga_A\to\ccM{\Ga/A}$ the fiber bundles in Fact \ref{bundleGaA_lmm} and Corollary~\ref{bundleA_crl}, respectively.
Denote 
$$Z_A=\{x\in\ov\cS_A\,\big|\,\tn{rank}\ d_x(f_\Ga|_{p_A^{-1}(p_A(x))})<d|V(\Ga_A)|-d-1\},\quad R_A=f_\Ga(\ov{\cS}_A-Z_A).$$
(Notice that $p_A^{-1}(p_A(x))$ is just the fiber of $p_A$ containing $x$.) In other words, $Z_A$ consists of points $x\in\ov\cS_A$ at which $f_\Ga|_{p_A^{-1}(p_A(x))}$ is not an immersion. 
$\ov\cS^\Ga_A-R_A$ is also a fiber bundle over $\ov{C}_{\Ga/A}(M,\i)$ with fiber
$$f^{\R^d}_{\Ga_A}\big(\{x\in\ov{C}^\quo_{V(\Ga_A)}(\R^d)\,\big|\,\tn{rank}\ d_x f^{\R^d}_{\Ga_A}<d|V(\Ga_A)|-d-1\}\big),$$
so $\dim_t(\ov\cS^\Ga_A\!-\!R_A)\leq d|V(\Ga)|\!-\!2$ by Lemma \ref{dim_lmm}, and $\ov\cS^\Ga_A\!-\!R_A$ is closed in $\ov\cS^\Ga_A$. 
Since $\ov\cS_A\!-\!\mathring{\cS}_A$ is covered by codimension-2 or higher strata of $\ccM{V(\Ga)}$, $\dim_t\big(f_\Ga(\ov\cS_A\!-\!\mathring{\cS}_A)\big)\leq d|V(\Ga)|\!-\!2$. We claim that
\begin{equation}\label{regvalue_eqn}
\ov\cS^\Ga_A-\big(f_\Ga(\ov\cS_A\!-\!\mathring{\cS}_A)\cup(\ov\cS^\Ga_A\!-\!R_A)\big)\subset\rsa.
\end{equation}
If this is true, then $\ov\cS^\Ga_A-\rsa$ is contained in the union of two closed subsets of $\dim_t\le d|V(\Ga)|-2$; and since $\ov\cS^\Ga_A-\rsa$ is itself closed in $\ov\cS^\Ga_A$ (that $\rsa$ is open in $\ov\cS^\Ga_A$ follows easily from the definition of $\rsa$), we are done. Let $x$ be in the LHS of \eref{regvalue_eqn}. 
Since $x\in R_A$ and $x\notin f_\Ga(\ov{\cS}_A\!-\!\mathring\cS_A)$, by Lemma \ref{regularvalueunique_lmm}, $f_\Ga^{-1}(x)$ consists of a single element $y\in\mathring{\cS}_A$. 
We next show that $f_\Ga$ is a homeomorphism onto its image in a neighborhood $U\!\subset\!\ov{C}_{V(\Ga)}(M,\i)$ of $y$. 
Since $p_A(\mathring\cS_A)\!=\!C_{V(\Ga/A)}(M,\i)\!=\!C_{\Ga/A}(M,\i)$, $f_\Ga|_{\mathring\cS_A}$ is locally the product of $f^{\R^d}_{\Ga_A}$ with a diffeomorphism. So $f_\Ga|_{\mathring\cS_A}$ is an immersion at $y$. Since $f_\Ga|_{C_{V(\Ga)}(M,\i)}$ is injective, $f_\Ga$ is injective in an open neighborhood $U\!\subset\!\ccM{V(\Ga)}$ of $y$. Since $f_\Ga$ is a closed map ($\ccM{V(\Ga)}$ is compact), $f_\Ga|_{U}$ is a homeomorphism onto $f_\Ga(U)$. Since $f_\Ga(\ccM{V(\Ga)}-U)$ is a closed subset of $\ccM{\Ga}$ not containing $x$, there is a neighborhood $V\subset\ccM{\Ga}$ of $x$ such that $V\cap f_\Ga(\ccM{V(\Ga)}-U)=\eset$, so $V\subset f_\Ga(U)$. This shows that $\ccM{\Ga}$ has the structure of a topological manifold with boundary in the neighborhood $V$ of $x$, completing the proof. 
\end{proof}
\begin{crl}\label{dimT12_crl}
$\dim_t(T_2)\le d|V(\Ga)|-2,\ \dim_t(T_1)\le d|V(\Ga)|-1$.
\end{crl}


\begin{lmm}\label{refinecover_lmm}
Let $Y$ be a compact metrizable space with  $\dim_t(Y)=n$,
and $Y_2\subset Y_1\subset Y$ be closed subspaces such that $\dim_t(Y_1)\leq n-1$,$\dim_t(Y_2)\leq n-2$. Then, for every open cover $\cU$ of $Y$, there exists a refinement $\cU'$ of $\cU$ such that 
\begin{itemize}
\vspace{-.5cm}
\item[$(*)$]
there are open neighborhoods $N_{Y_1}$ of $Y_1$, $N_{Y_2}$ of $Y_2$, such that for all $U_0,\ldots,U_n\in\cU'$, pairwise distinct,
$$(U_0\cap\ldots\cap U_n)\cap N_{Y_1}=\eset,\quad(U_0\cap\ldots\cap U_{n-1})\cap N_{Y_2}=\eset.$$
\end{itemize}
\vspace{-.5cm}
Hence, if $S\subset Y_1$ is a closed subset such that $Y_1-Y_2\subset S$, then there are canonical isomorphisms 
$$H^n(Y,S)\approx H^n(Y,Y_1)\approx H^n_c(Y-Y_1).$$
\end{lmm}
\begin{proof}
The proof of \cite[Lemma 21.2.1]{KM} goes through almost verbatim here and gives us the first statement. (We first use \cite[Proposition 12-9 (1)(3)]{DT} where $Y_1,Y_2$ are plugged in as $C_1,C_2$, and then use \cite[Proposition 9-3]{DT}.) 
The second statement easily follows using standard arguments in \v{C}ech cohomology. For the first isomorphism: the restriction map $\ch{C}^i_\cU(Y,Y_1)\to\ch{C}^i_\cU(Y,S)$ is an equality for all $i\ge n-1$ and all open covers $\cU$ of $Y$ satisfying $(*)$, so $\varinjlim_{\cU\tn{ satisfying }(*)}H^n_\cU(Y,Y_1)=\varinjlim_{\cU\tn{ satisfying }(*)}H^n_\cU(Y,S)$; since every open cover of $Y$ has a refinement satisfying $(*)$, these two limits are equal to $H^n(Y,Y_1)$, $H^n(Y,S)$, respectively. 
For the second isomorphism: this follows from \cite[Proposition 12-3]{Bredon}. Alternatively, let $\Phi$ be the collection of compact subsets of $Y\!-\!Y_1$, viewed as subsets of $Y$, and it is not hard to check directly that the natural restriction maps $H^n_\Phi(Y)\to H^n_c(Y-Y_1)$, $H^n_\Phi(Y)\to H^n(Y,Y_1)$ are isomorphisms. 
\end{proof}
\begin{crl}\label{StoT_crl}
$H^{d|V(\Ga)|}(\tcM,S)\approx H^{d|V(\Ga)|}(\tcM,T_1)\approx H^{d|V(\Ga)|}_c(\tcM-T_1)$ via $\cG$-equivariant isomorphisms. 
\end{crl}

The goal of the rest of Section \ref{structureX_sec} is to construct a non-trivial map $\rho: H^{d|V(\Ga)|}_c(X_\Ga-T_1)\to R$, and the method is to realize $X_\Ga-T_1$ as the image of a proper map from an oriented topological manifold of dimension $d|V(\Ga)|$. 
An alternative approach for constructing $\rho$ is Section \ref{XGa_sec} (which we will only sketch), and it is a better and more canonical approach. But the method here is less technical and easier to write, so we use it instead. 

Recall at the end of Section~\ref{fixagraph_sec} we defined, for every $A\subset\{\i\}\sqcup V(\Ga)$ of type~2, $\si_A\in\tS_{E(\Ga)}$, and for every $\Ga$-pair $A_1,A_2\in E(\Ga)$, $\si_{A_1A_2}\in\tS_{E(\Ga)}$. Thus, we have $\phi(\si_A),\phi(\si_{A_1A_2}):\ccM{2}^{E(\Ga)}\lra\ccM{2}^{E(\Ga)}$ as in Definition \ref{Saction_dfn}. 


\begin{lmm}\label{cancel_lmm}
$\phi(\si_A)(\ov{\cS}^\Ga_A)=\ov{\cS}^\Ga_A,\ 
\phi(\si_{A_1A_2})(\ov{\cS}^\Ga_{A_1})=\ov{\cS}^\Ga_{A_2}$. 
\end{lmm}
\begin{proof}
Let $A$ be of type~2 with chosen bivalent vertex $v_A$, and vertices $v^1_A,v^2_A$ adjacent to it. 
There is a dense open subset $\mathring{\cS}'_A\subset\mathring\cS_A$ on which we can define an involution $\phi'_A:\mathring\cS'_A\to\mathring\cS'_A$ which fixes all other marked points and reflects the point labeled by $v_A$ along the mid-point of the line segment between the points labeled by $v^1_A,v^2_A$, on the screen that these marked points lie on. (This argument is in Kontsevich's original paper \cite[Lemma~2.1]{Kontsevich}; see e.g.  \cite[Figure~16]{WatanabeAdd} for a nice picture. Notice here we use $\mathring\cS'_A$ because $\phi'_A$ is not well-defined on $\mathring\cS_A$, due to cases when the new position of the point labeled by $v_A$ coincides with other points.) Clearly $\phi(\si_A)\circ f_\Ga=f_\Ga\circ\phi'_A$. So $$\phi(\si_A)(\ov{\cS}^\Ga_A)=\phi(\si_A)(f_\Ga(\ov\cS_A))=\ov{\phi(\si_A)(f_\Ga(\mathring\cS'_A))}= \ov{f_\Ga(\phi'_A(\mathring\cS'_A))}=\ov{f_\Ga(\mathring\cS'_A)}=\ov{\cS}^\Ga_A.$$

Let $A_1,A_2$ be a $\Ga$-pair. Since for $A$ of type 4, $\mathring\cS_{A}$ is an $S^{d-1}$-bundle over $C_{V(\Ga/A)}(M,\i)$ with fiber over $x$ canonically identified with $ST_{f_{[A]_v}(x)}M$, 
we can define $\phi'_{A_1A_2}:\mathring\cS_{A_1}\to\mathring\cS_{A_2}$ by lifting the map $C_{V(\Ga/A_1)}(M)\to C_{V(\Ga/A_2)}(M)$ switching marked points in the same way as $\al_{A_1A_2}$ (defined by the end of Section~\ref{fixagraph_sec}) maps vertices of $\Ga/A_1$ to vertices of $\Ga/A_2$. Since $[A_1]_v$ is mapped to $[A_2]_v$, the fibers are canonically identified. By the definition of $\si_{A_1A_2}$, $\phi(\si_{A_1A_2})\circ f_\Ga=f_\Ga\circ\phi'_{A_1A_2}$. So 
$$\phi(\si_{A_1\!A_2})(\ov{\cS}^\Ga_{A_1})\!=\!\phi(\si_{A_1\!A_2})(f_\Ga(\ov\cS_{A_1}))\!=\!\ov{\phi(\si_{A_1\!A_2})(f_\Ga(\mathring\cS_{A_1}))}\!=\! \ov{f_\Ga(\phi'_{A_1\!A_2}(\mathring\cS_{A_1}))}\!=\!\ov{f_\Ga(\mathring\cS_{A_2})}\!=\!\ov{\cS}^\Ga_{A_2}.$$
\end{proof}


\begin{crl}
$\phi(\si_A)(\rrsa)=\rrsa,\ \phi(\si_{A_1A_2})(\ringring\cS^\Ga_{A_1})=\ringring\cS^\Ga_{A_2}.$
\end{crl}
Denote $$C'_\Ga(M,\i):=C_\Ga(M,\i)\cup\hspace{-.3cm}\bigcup_{A\tn{ of type 2 or 4}}\hspace{-.3cm}\rrsa\ \subset\ \ccM{\Ga}.$$ 
Then by the definition of $\rrsa$, $C'_{\Ga}(M,\i)$ is a topological manifold with boundary. 
\begin{dfn}\label{Xtl_dfn}
Take $2^{|E(\Ga)|}|E(\Ga)|!$ copies of $C'_\Ga(M,\i)$, labeled by elements in $\tS_{E(\Ga)}$. We write $C'_\Ga(M,\i)^{@\si}$ for the copy labeled by $\si\in\tS_{E(\Ga)}$, and similarly write $(\rrsa)^{@\si},x^{@\si}$, etc., for its subspaces and elements. We orient $C'_\Ga(M,\i)^{@\si}$ by twisting the orientation on $C_\Ga(M,\i)$ by $(-1)^{(d-1)\sgn(\si)+d\sgn'(\si)}$. 
Define
$$\wt{X}_\Ga:=\Big(\bigsqcup_{\si\in\tS_{E(\Ga)}}C'_{\Ga}(M,\i)^{@\si}\Big)\Big/\sim_\Ga,$$
where $\sim_\Ga$ is the following equivalence relation (gluing boundary components pairwise):
\begin{itemize}
\vspace{-.5cm}
\item $\forall \ A\subset\{\i\}\sqcup V(\Ga)$ of type~2, $\forall \ x\in\rrsa$, $\forall\ \si\in\tS_{E(\Ga)}$,
$x^{@\si}\sim_\Ga (\phi(\si_A)(x))^{@\si\si_A^{-1}}$;
\item $\forall\ A_1,A_2$ a $\Ga$-pair, $\forall\ x\in\ringring\cS^\Ga_{A_1}$, $\forall\ \si\in\tS_{E(\Ga)}$, 
$x^{@\si}\sim_\Ga (\phi(\si_{\!A_1\!A_2})(x))^{@\si\si_{\!A_1\!A_2}^{-1}}$. 
\end{itemize}
Moreover, define 
\begin{gather*}
\wt{f}=(\wt{f}_{e})_{e\in E(\Ga)}:\wt{X}_\Ga\lra\ccM{2}^{E(\Ga)},\\
\wt{f}|_{C'_\Ga(M,\i)^{@\si}}=\phi(\si)\circ f_\Ga.
\end{gather*}
It is well-defined since $\phi(\si\si_A^{-1})\circ\phi(\si_A)=\phi(\si),\,\phi(\si\si_{A_1A_2}^{-1})\circ\phi(\si_{A_1A_2})=\phi(\si)$.
\end{dfn}
It can be easily seen that $\tn{image}(\wt{f})=\tcM-T_1$. It follows from Lemma \ref{chambers_lmm} and the definition above that $\wt{f}|_{\bigsqcup_{\si}C_{\Ga}(M,\i)^{@\si}}$ is a covering map (onto its image) of degree $|\Au(\Ga)|^\pm_d$. 
Since $\phi(\si)$ is $\cG$-equivariant for all $\si\in\tS_{E(\Ga)}$, the diagonal $\cG$-action on $\ccM{2}^{E(\Ga)}$ lifts to an action of $\cG$ on $\wt{X}_\Ga$, so that $\wt{f}$ is equivariant. 
\begin{lmm}\label{XGa_lmm}
$\wt{X}_\Ga$ is an oriented topological manifold of dimension $d|V(\Ga)|$. 
\end{lmm}
\begin{proof}
That it is a topological manifold of dimension $d|V(\Ga)|$ follows from that, $\phi(\si_A):\rrsa\to\rrsa$ and $\phi(\si_{A_1A_2}):\ringring\cS^\Ga_{e_1}\to\ringring\cS^\Ga_{e_2}$ are homeomorphisms and $\sim_\Ga$ glues together these boundary components of $\bigsqcup_{\si\in\tS_{E(\Ga)}}C'_{\Ga}(M,\i)^{@\si}$ pairwise. It is not difficult to verify that $\phi(\si_A),\phi(\si_{A_1A_2})$ are orientation-reversing, if $(\rrsa)^{@\si}$'s are oriented as boundaries of $C'_\Ga(M,\i)$. 
\end{proof}

The purpose of the following lemma is to show that $\wt{f}$ induces a map 
$$\wt{f}^*: H^{d|V(\Ga)|}_c(X_\Ga-T_1)\lra H^{d|V(\Ga)|}_c(\wt X_\Ga).$$

\begin{lmm}\label{ftlproper_lmm}
$\wt{f}:\wt{X}_\Ga\lra X_\Ga-T_1$ is a proper map. 
\end{lmm}
\begin{proof}
Let $K\subset X_\Ga-T_1$ be compact. To show $\wt{f}^{-1}(K)$ is compact, let $\{x_n\}_{n=1}^\i$ be a sequence of points in $\wt{f}^{-1}(K)$. There is a subsequence (still call it $\{x_n\}$) and some $\sigma\in\wt{S}_{E(\Ga)}$ such that $x_n\in C'_\Ga(M,\i)^{@\si}$ for all $n$. After possibly passing to a subsequence, $\wt{f}(x_n)$ converges to some $y\in\phi(\si)\big(\ov{C}_\Ga(M,\i)\big)\cap K$, since $\wt{f}(x_n)\in\phi(\si)(C'_\Ga(M,\i))$ and $\ov{C}_\Ga(M,\i)$ is closed. But $y\notin\phi(\si)(\ov{\cS}^\Ga_A-\rrsa)$ for any $A$, since $\phi(\si)(\ov{\cS}^\Ga_A-\rrsa)\subset T_1$. Thus, $y\in\phi(\si)(C'_\Ga(M,\i))$. Since $\wt{f}$ maps $C'_\Ga(M,\i)^{@\si}$ homeomorphically onto $\phi(\si)(C'_\Ga(M,\i))$, by the definition of $C'_\Ga(M,\i)$, $\{x_n\}$ converges to the unique element $x\in \wt{f}^{-1}(y)\cap C'_\Ga(M,\i)^{@\si}$. 
\end{proof}
\begin{dfn}\label{rho_dfn}
Define $\rho:H^{d|V(\Ga)|}(X_\Ga,S;R)\lra R$ to be the composition
$$H^{d|V(\Ga)|}(X_\Ga,S)\xra{\tn{Crl }\ref{StoT_crl}} H^{d|V(\Ga)|}_c(X_\Ga-T_1)\xra{\wt{f}^*}H^{d|V(\Ga)|}_c(\wt X_\Ga)\lra R,$$
where the last arrow is by taking cap product with the fundamental class of $\wt{X}_\Ga$ (in the sense of Borel-Moore homology). 
\end{dfn}
Since the $\cG$-action on $\wt X_\Ga$ is orientation-preserving, the last map is $\cG$-equivariant (where $R$ is equipped with the trivial action). So, $\cG$ acts on all the objects involved in this definition and all maps involved are equivariant, implying that $\rho$ is $\cG$-equivariant. 

\begin{rmk}\label{graphsum_rmk}
If $\Ga$ is not a single graph but a formal sum of graphs, $\sum_{i=1}^m\Ga_i$, the arguments in Section \ref{space_sec} need to be modified as follows. 
First, note that if two graphs share a common boundary term, then they much have the same number of edges, therefore without loss of generality we can assume all the $\Ga_i$ have the same number of edges, say $n$. 
Then, let us fix, for each $\Ga_i$, a bijection between $E(\Ga_i)$ and $\{1,\ldots,n\}$. Since we will soon be summing over all permutations of edges, this choice does not matter. 
Our ``ambient space'' in this case will be $\ccM{2}^n$, in place of $\ccM{2}^{E(\Ga)}$ in the single-graph case. 
Now, for each $\Ga_i$, we can define $C_{\Ga_i}(M,\infty),\ov{C}_{\Ga_i}(M,\infty)\subset \ccM{2}^n$, just like in the single-graph case, using the chosen bijection. 
Similar to Lemma \ref{chambers_lmm}, it is easy to see that the elements in  $\{\phi(\sigma)(C_{\Ga_i}(M,\infty))\}_{i\in\{1,\ldots,m\},\sigma\in \wt{S}_{\{1,\ldots,n\}}}$ are disjoint from each other. 
We define 
$$X_{\sum_{i=1}^m\Ga_i}=\bigcup_{\sigma\in\wt{S}_{\{1,\ldots,n\}},i\in\{1,\ldots,m\}}\phi(\sigma)\big(\ov{C}_{\Ga_i}(M,\infty)\big)\subset\ccM{2}^n.$$
The rest of the arguments in the paper all generalize to this case in a straight-forward way. 
\end{rmk}

\subsection{Digression: what \texorpdfstring{$\tcM$}{X} looks like}\label{XGa_sec}

This subsection can be skipped. It is here to justify Figure \ref{X_fig}: after removing a codimension 2 subset from $X_\Ga$, it looks like a manifold with boundary and bindings, where each binding component is connected to an even number of pages, summing up to 0 when counted with sign. This statement would also allow us to define $\rho$ in a different (more canonical) way than in Section \ref{structureX_sec}.  
But working out everything precisely is quite technically involved, so we will only sketch such an approach in this subsection. 

By Lemma \ref{chambers_lmm}, 
$$\tcM= \bigsqcup_{[\si]\in\tS_{E(\Ga)}/\tn{image}(\psi_\Ga)}\phi(\si)\big(C_\Ga(M,\i)\big)\sqcup\bigcup_{\si\in\tS_{E(\Ga)}}\phi(\si)\big(\ccM{\Ga}-C_\Ga(M,\i)\big).$$
Denote the first term above by $\rxg$. 

First we define ``binding points'' (those $p$ in the definition below). 

\begin{dfn}\label{binding_dfn}
For $A\subset\{\i\}\sqcup V(\Ga),|A|\ge2$, define $\rrrsa\subset\cS^\Ga_A$ to be the set of points $p$ satisfying: there exists a neighborhood $U\subset \ccM{2}^{E(\Ga)}$ of $p$, such that for every $\si\in\tS_{E(\Ga)}$, either
\begin{enumerate} 
\vspace{-.5cm}
\itemsep0em 
\item[(1)] $\phi(\si)\big(\ccM{\Ga}\big)\cap U=\eset$, or 
\item[(2)]
$p=\si(q)$ for some $A'\subset\{\i\}\sqcup V(\Ga),|A'|\ge2$, $q\in\cS^\Ga_{A'}$, and 
\begin{enumerate}
\item[$\small\bu$] there is a homeomorphism 
$$\nu: U\cap\phi(\si)\big(\ccM{\Ga}\big)\lra \R^{N-1}\times\R^{\ge0}, \text{ s.t. }\nu\big(U\cap\phi(\si)(\cS^\Ga_{A'})\big)=\R^{N-1}\times\{0\}$$
(i.e., $\phi(\si)\big(\ccM{\Ga}\big)$ is a topological manifold with boundary near $p$);
\item[$\small\bu$] $U\cap\phi(\si)(\cS^\Ga_{A'})=U\cap\cS^\Ga_A.$
\end{enumerate}
\vspace{-.5cm}
\end{enumerate}
For $p\in\rrrsa$, define the {\it signed count of pages at $p$} to be the signed count of elements in $\tS_{E(\Ga)}$: those $\si$ of case (1) above are counted with 0; those $\si$ of case (2) above are counted with $\pm1$: $+1$ if the boundary orientations of $\cS^\Ga_A$ near $p$, as boundary of $\phi(\si)\big(\ccM{\Ga}\big)$ and as boundary of $\ccM{\Ga}$, agree; $-1$ if they disagree. 
\end{dfn}

\begin{lmm}
If $p\in\rrrsa$ where $A$ is of type 2 or 4, then the signed count of pages at $p$ is always 0. 
\end{lmm}
\begin{proof}[Sketch of Proof]
For type 2 $A$s, the pages come in pairs of opposite signs, see Lemma \ref{cancel_lmm}; for type 4 $A$s, the pages sum up to 0 because $\Ga$ is closed in graph homology. 
\end{proof}


In this subsection, $\rrrsa$ replaces the role played by $\rrsa$ in the previous section. 
We define $S,T_1,T_2$ verbatim as in Definition \ref{strata_dfn}, just with $\rrsa$ replaced by $\rrrsa$. The statements in the paragraph below Definition \ref{strata_dfn} still hold.  
We next show that that analogues of Corollaries \ref{dimT12_crl},\ref{StoT_crl} still hold. 

\begin{lmm}\label{goodcohomology_lmm}
There is a $\cG$-equivariant surjective map $H^{d|V(\Ga)|}_c(X_\Ga-T_1;R)\to R$, where $\cG$ acts on $R$ trivially. 
\end{lmm}
\begin{proof}[Sketch of Proof]
$X_\Ga-T_1$ consists of two parts: $\rxg$ is an open subset of $X_\Ga-T_1$ which is also an $d|V(\Ga)|$-dimensional (oriented) topological manifold, and $Y:=X_\Ga-T_1-\rxg$ is a closed subset of $X_\Ga-T_1$ which is also an $(d|V(\Ga)|-1)$-dimensional topological manifold (this follows from the definition of book binding points). It also follows from the definition of book binding points that $X_\Ga-T_1$ is locally contractible. So we have the long exact sequence of compactly supported cohomology, 
$$\cdots\lra H^{d|V(\Ga)|-1}_c(Y)\xra{\de}H^{d|V(\Ga)|}_c(\rxg)\lra H^{d|V(\Ga)|}_c(X_\Ga-T_1)\lra H^{d|V(\Ga)|}_c(Y)\lra\cdots,$$
where the last term is 0. 
Denote by $J_1,J_2$ the set of connected components of $Y$ and $\rxg$, respectively, then 
$$H^{d|V(\Ga)|-1}_c(Y)\approx R^{\oplus J_1},\qquad H^{d|V(\Ga)|}_c(\rxg)\approx R^{\oplus J_2},$$
and $\de$ is the coboundary map. So, by Lemma 3.27, the image of $\de$ is contained in 
$\{(r_i)_{i\in J_2}\,\big|\,\sum_{i}r_i=0\}.$ Therefore, the map 
$$H^{d|V(\Ga)|}_c(\rxg)\lra R,\qquad (r_i)_{i\in J_2}\lra\sum_{i\in J_2}r_i$$
induces a surjective map from the quotient $H^{d|V(\Ga)|}_c(\rxg)/\tn{image}(\de)\approx H^{d|V(\Ga)|}_c(X_\Ga-T_1)$ to $R$, as desired. 
%
\end{proof}

\begin{lmm}\label{T2small_lmm}
$\dim_t(T_2)\le d|V(\Ga)|-2$.
\end{lmm}

Together with Lemma \ref{refinecover_lmm}, Lemma \ref{T2small_lmm} implies that $H^{d|V(\Ga)|}(X_\Ga,S)\approx H^{d|V(\Ga)|}_c(X_\Ga-T_1)$, and we can thus define $\rho: H^{d|V(\Ga)|}(X_\Ga,S)\to R$ using the map in Lemma \ref{goodcohomology_lmm}. The rest of this subsection is devoted to the following
\begin{proof}[Sketch of Proof of Lemma \ref{T2small_lmm}]
Recall $T_2$ consists of points in $X_\Ga-\rxg$ that are not binding points. So, we need to analyze, for $A,A'\in V(\Ga)\sqcup\{\i\}$ and $\si\in\wt{S}_{E(\Ga)}$, how ${\cS}^\Ga_A$ and $\phi(\si)({\cS}^\Ga_{A'})$ intersect. Suppose $A,A',\si$ are such that they do intersect. Then, by the same reasoning as in Lemma \ref{chambers_lmm}, there exists an unordered, unoriented graph isomorphism $\Ga/\Ga_{A'}\to\Ga/\Ga_A$ whose edge permutation is given by the restriction of $\si$ to $E(\Ga/\Ga_{A'})$. So $\si$ also restricts to a bijection $E(\Ga_{A'})\to E(\Ga_A)$. Abusing notation we still denote by $\phi(\si)$ the map $(S^{d-1})^{E(\Ga_{A'})}\to(S^{d-1})^{E(\Ga_{A})}$, permuting factors according to $\si$ and composing with the antipodal map when there is a negative sign.
Recall the notation ``$f^{\R^d}_{\Ga'}$'' in Definition \ref{Vga_dfn}.
Denote $${V}^\si_{A'}=\phi(\si)\big(f^{\R^d}_{\Ga_{A'}}(C^\quo_{V(\Ga_{A'})}(\R^d))\big)\subset(S^{d-1})^{E(\Ga_A)},\qquad V_A=f^{\R^d}_{\Ga_{A}}(C^\quo_{V(\Ga_{A})}(\R^d))\subset(S^{d-1})^{E(\Ga_A)}.$$
Then, ${\cS}^\Ga_A\cap\phi(\si)({\cS}^\Ga_{A'})$ is a fiber bundle over ${C}_{\Ga/A}(M,\i)$ with fiber ${V}_A\cap{V}^\si_{A'}\subset (S^{d-1})^{E(\Ga_A)}$. 

\begin{lmm}
Let $\Ga'$ be a graph. Then $f^{\R^d}_{\Ga'}(C^\quo_{V(\Ga')}(\R^d))\subset(S^{d-1})^{E(\Ga')}\subset (\R^d)^{E(\Ga')}$, where $S^{d-1}$ is viewed as the unit sphere in $\R^d$, is semi-algebraic. 
\end{lmm}
\begin{proof}
It is the image of the composition of a linear map  $(\R^d)^{V(\Ga)}\!-\De_{\tn{big}}\to(\R^d-\{0\})^{E(\Ga')}$ with the projection map $(\R^d-\{0\})^{E(\Ga')}\to (S^{d-1})^{E(\Ga')}$. Both of these maps' graphs are semi-algebraic. So the image is also semi-algebraic by Tarski–Seidenberg Theorem. 
\end{proof}
By the above lemma, $V^\si_{A'}$ and $V_A$ are semi-algebraic. They are also open subsets of $(S^{d-1})^{E(\Ga_A)}$. Let $Y^\si_{A'},Y_A\subset (S^{d-1})^{E(\Ga_A)}$ be minimal algebraic sets containing $V^\si_{A'},V_A$, respectively. Then $V^\si_{A'}\subset Y^\si_{A'},V_A\subset Y_A$ are open (in Euclidean topology), and the Krull dimensions of both $Y_A$ and $Y^\si_{A'}$ are $d|V(\Ga_A)|-d-1$. Denote by $Z(A,A',\si)$ the union of irreducible components of $Y^\si_{A'}\cap Y_A$ whose Krull dimension is less than $d|V(\Ga_A)|-d-1$. Suppose $p\in V_A\cap V^\si_{A'} - Z(A,A',\si)$, then $p$ is in some irreducible component $Y$ of $Y_A\cap Y^\si_{A'}$ whose Krull dimension is $d|V(\Ga_A)|-d-1$, so $Y$ must also be an irreducible component of both $Y_A$ and $Y^\si_{A'}$. Since $p\notin Z(A,A',\si)$, there exists a neighborhood $U_p\subset(S^{d-1})^{E(\Ga_A)}$ of $p$ such that $V_A\cap U_p=Y\cap U_p=V^\si_{A'}\cap U_p$. 

Now, define $\wt{Z}(A,A',\si)$ to be the sub-fiber bundle of $\cS^\Ga_A$ over $C_{\Ga/A}(M,\i)$ whose fibers are $Z(A,A',\si)\cap V_A$. 
For a given $A\subset V(\Ga)\sqcup\{\i\}$, define 
$\cS_A^{\Ga,\tn{rmv}}=\bigcup_{\si,A'}\wt{Z}(A,A',\si)$, where ``rmv'' stands for ``remove''. It can be shown that the Krull dimension of an algebraic subset of $\R^n$ equals to its covering dimension in Euclidean topology, so $\dim_t(\cS^{\Ga,\tn{rmv}}_A)\le d|V(\Ga)|-2$. By the conclusion of the previous paragraph, every point in $\ov\cS^\Ga_A$ which is not in (1) $\cS_A^{\Ga,\tn{rmv}}$ or (2) $\cS^\Ga_A-\rrsa$ (as in Definition) or (3) the image of some codimsion at least 2 stratum of $\ov C_{V(\Ga)}(M,\i)\big)$ under $f_\Ga$, 
is a binding point, by the definition of binding points. By Lemma \ref{small_lmm}, $\dim_t(\cS^\Ga_A-\rrsa)\le d|V(\Ga)|-2$. So, the union of the above three sets has covering dimension at most $d|V(\Ga)|-2$. This completes the proof of the lemma.

\end{proof}

\section{Re-constructing Kontsevich's characteristic classes}\label{propagator_sec} 
Recall $R=\Z$ or $\R$ and all cohomology in this section are with $R$-coefficients, which we omit. 
Let $\Ga$ be as in Section \ref{fixagraph_sec}; assume also that $\Ga$ is trivalent. Let $M$ be as in Section \ref{thm_sec}. Assume $d\ge3$. 
Let $\pi\!:E\to B$ be a framed smooth $(M,\i)$-fiber bundle as in Section \ref{thm_sec}. 
We assume $B$ can be given a CW-structure. Using CW-approximation, Definition \ref{K_dfn} and thus Corollary \ref{K_crl} generalize to cases where $B$ is just a paracompact Hausdorff space.

Let 
\[
\begin{tikzcd}[column sep=-.1cm]
C_2(\pi)\dar&\ov{C}_2(\pi)\dar&\ov{C}^{E(\Ga)}_2(\pi)\dar&\ov{C}_\Ga(\pi)\dar& \cS_A(\pi)\dar & \cS^\Ga_A(\pi)\dar & \ov\cS^\Ga_A(\pi)\dar & X_\Ga(\pi)\dar["\pi_X"] &T_1(\pi)\dar&T_2(\pi)\dar&S(\pi)\dar&\wt{X}_\Ga(\pi)\dar["\pi_{\wt X}"]\\
B&B&B&B&B&B&B&B&B&B&B&B
\end{tikzcd}
\]
be the associated bundles of $\pi$ with fibers $C_2(M,\i)$, $\ccM{2}$, $\ccM{2}^{E(\Ga)}$, $\ccM{\Ga}$, $\cS_A$, $\cS^\Ga_A$, $\ov\cS^\Ga_A$, $\tcM$, $T_1$, $T_2$, $S$, $\wt{X}_\Ga$, respectively. (All of these spaces, except for $\wt{X}_\Ga$, are defined in Section \ref{defspace_sec}; $\wt{X}_\Ga$ is defined in Definition \ref{Xtl_dfn}.) Correspondingly, the maps $f_\pm,f_\Ga,\wt{f}$ in Section~\ref{defspace_sec} induce bundle maps. Abusing notation, we still denote them by $f_\pm,f_\Ga,\wt{f}$. Notice that $\ccEft$ denotes the fiber product while $\ccEt^{E(\Ga)}$ denotes the direct product of the total space, ignoring the fiber bundle structure. 

\begin{lmm}\label{Gbundlemaps_lmm}
Under the condition of Theorem \ref{main_thm}, $(\tl h,h)$ induces $\cG$-bundle maps between the $\pi'$ and $\pi''$ version of all the bundles above. 
\end{lmm}
\begin{proof}
By the existence of $\wt{h}$, $h$ is a $\cG$-bundle map. Since every space defined in Section \ref{defspace_sec} has induced $\cG$-action and the maps between them defined in Section \ref{defspace_sec} are all $\cG$-equivariant, $$\tl{h}^{E(\Ga)}:\ov{C}_2(\pi')^{E(\Ga)}\lra\ov{C}_2(\pi'')^{E(\Ga)}$$
restricts to $\cG$-bundle maps
\begin{gather*}\tl{h}^{E(\Ga)}\!\!:\!\ov{C}^{E(\Ga)}_2\!(\pi')\to\ov{C}^{E(\Ga)}_2\!(\pi''),\ \tl{h}_\Ga\!:\!\ccEp{\Ga}\to\ccEpp{\Ga},\\
\tl{h}_X\!:\!X_\Ga(\pi')\to X_\Ga(\pi''),\ \tl{h}_{S}\!:\!S(\pi')\to S(\pi''),\quad\tn{etc..}
\end{gather*}
\end{proof}

The framing on $\pi$ induces a map $F:\prt^v\ov{C}_2(\pi)\to S^{d-1}$, as in \cite[\S2.4.3]{WatanabeAdd} (it is called $p(\tau_E)$ there). 
Define $\sim_F$ to be the following equivalence relation on $\ov{C}_2(\pi)$: 
$$x\sim_F y\ \tn{ if }\ x=y\in C_2(\pi)=\ov{C}_2(\pi)\!-\!\prt^v\ov{C}_2(\pi)\ \tn{ or }\ F(x)=F(y),\ x,y\in\prt^v\ov{C}_2(\pi). $$
Denote by $q\!:\!\ov{C}_2(\pi)\to\ov{C}_2(\pi)/\!\sim_F$ the quotient map by $\sim_F$, where the target is equipped with the quotient topology. Let us denote $S^{d-1}_\pi=q(\prt^v\ov{C}_2(\pi))$ and denote $(\cq)-S^{d-1}_\pi$ still by $C_2(\pi)$, then $\cq=C_2(\pi)\sqcup S^{d-1}_\pi$ as a set. 
It is not hard to see that $\cq$ is Hausdorff and $S^{d-1}_\pi$ is a deformation retract of some neighborhood of it (hint: since $\ov{C}_2(\pi)$ restricted to each cell of $B$ is a manifold with compact boundary, using a cell-by-cell construction we can find a collar neighborhood of $\prt^v\ov{C}_2(\pi)$ in $\ov{C}_2(\pi)$). Notice that the orientation on $M$ specifies an orientation on $S^{d-1}_\pi$. 

\begin{dfn}\label{propagator_dfn}
The {\it propagator class} $\Om(\pi)\in H^{d-1}(\cq;R)$ is the unique class satisfying that $\Om(\pi)|_{S^{d-1}_\pi}\in H^{d-1}(S^{d-1}_\pi;R)$ is the Poincar\'e dual of the point class. 
\end{dfn}

The existence and uniqueness of such a class follows from the exact sequence
$$H^{d-1}(\cq,\Sp)\to H^{d-1}(\cq)\to H^{d-1}(\Sp)\to H^d(\cq,\Sp)$$
and the vanishing of its first and last terms: for $n=d,d-1$, $$H^n(\cq,\Sp)\!\approx\!H^n\big((\cq)/\Sp\big)\!=\!H^n(\ov{C}_2(\pi)/\prt\ov{C}_2(\pi))\!\approx\!H^n(\ov{C}_2(\pi),\prt\ov{C}_2(\pi))$$
and the vanishing of the last term follows from the proof of \cite[Lemma 2.10]{WatanabeAdd}. 
The definition of propagator class here is completely analogous to that in \cite{KT}. 

We next show that $\Om(\pi)$ gives us a class $\Om_\Ga(\pi)\in H^{|E(\Ga)|(d-1)}\big(X_\Ga(\pi),S(\pi)\big)$. 

\subsection{Defining \texorpdfstring{$\Om_\Ga$}{Om}}

\label{landinginrelative_subsec}

Define
$$\io: X_\Ga(\pi)\subset\ov{C}^{E(\Ga)}_{2}(\pi)\subset\ov{C}_2(\pi)^{E(\Ga)}\xra{q,\ldots,q}(\cq)^{E(\Ga)}.$$
We remark that $\ccE{2}^{E(\Ga)}$ is the direct product and $\ov{C}^{E(\Ga)}_{2}(\pi)$ is the fiber product. 
For $I\subset{E(\Ga)}$, denote by $p_I\!:\!(\cq)^{E(\Ga)}\to(\cq)^{I}$ the projection to the $I$-factors. If $I=\{e\}$ we also denote $p_e=p_I$. 
For $\si\in\tS_{E(\Ga)}$ and $A\subset\{\i\}\sqcup V(\Ga)$, define $I_{\si,A}=\si(E(\Ga_A))\subset E(\Ga)$, where we implicitly identify edges of $\Ga_A$ also as edges of $\Ga$ and abuse notation to still write $\si$ for its image in $S_{E(\Ga)}$.
\begin{lmm}\label{:(_lmm}
Suppose $A\subset\{\i\}\sqcup V(\Ga)$ and $\si\in\wt{S}_{E(\Ga)}$. Denote $$\ov{V}_A^\si=p_{I_{\si,A}}\big(\iota(\phi(\si)(\ov{\cS}^\Ga_A(\pi)))\big)\subset(\cq)^{I_{\si,A}},$$ 
then $\ov{V}^\si_A\subset (S^{d-1}_\pi)^{I_{\si,A}}$ and it is the image of $\ov{C}^\quo_{V(\Ga_A)}(\R^d)$ under a smooth map. 
\end{lmm}
\begin{proof}
This is a consequence of the framing $F$ on $\pi$ (in the case $\i\notin A$) and the trivialization of $\pi$ near $s_\i(B)$ (in the case $\i\in A$). For simplicity we only consider the case $\si=\id$; the other cases follow easily. 
First assume $\i\notin A$. Recall $\ov{V}_{\Ga_A}\!\subset\!(S^{d-1})^{E(\Ga_A)}$ as in Definition \ref{Vga_dfn}. Since the framing $F$ identifies the vertical tangent space of $E$ at each point not in $s_\i(B)$ with $\R^d$, 
$$\cS^\Ga_A(\pi)\approx C_{\Ga/A}(\pi)\sx\ov{V}_{\Ga_A}\subset C_2(\pi)^{E(\Ga/A)}\sx(S^{d-1}_\pi)^{E(\Ga_A)}\subset (\cq)^{E(\Ga/A)}\sx(\cq)^{E(\Ga_A)},$$
and $(p_{E(\Ga_A)}\circ\io)(\cS^\Ga_A(\pi))=\ov{V}_{\Ga_A}$. Since $\ov{V}_{\Ga_A}$ is closed,  $(p_{E(\Ga_A)}\circ\io)(\ov{\cS}^\Ga_A(\pi))=\ov{V}_{\Ga_A}$ as well.

Now assume $\i\in A$. Since $\pi$ is trivialized near $s_\i(B)$, the vertical tangent spaces of $E$ at points in $s_\i(B)$ are all identified with $T_\i M$. So
$$\ov\cS_A(\pi)\approx \ov{C}_{V(\Ga/A)}(\pi)\sx \ov{C}^\quo_{V(\Ga_A)}(T_\i M).$$
And $(p_{E(\Ga_A)}\circ\io)(\ov\cS^\Ga_A(\pi))=(p_{E(\Ga_A)}\circ\io\circ f_\Ga)(\ov\cS_A(\pi))$; the map $p_{E(\Ga_A)}\circ\io\circ f_\Ga$ factors through the projection to the second factor above. 
\end{proof}


By Lemma \ref{Y_lmm} and Lemma \ref{X_lmm} below (these are two somewhat technical lemmas whose statements and proofs are postponed to below Proposition \ref{naturalclass_prp}, due to their lengths), the cup product 
\begin{gather*}
\cup\!:H^{d-1}(\cq)^{\otimes E(\Ga)}\lra H^{|E(\Ga)|(d-1)}\big((\cq)^{E(\Ga)}\big),\\ \otimes_{i=1}^{|E(\Ga)|}\si_{e^\Ga_i}\lra p_{e^\Ga_1}^*\si_{e^\Ga_1}\cup\ldots\cup p_{e^\Ga_{|E(\Ga)|}}^*\si_{e^\Ga_{|E(\Ga)|}}
\end{gather*}
factors through $H^{|E(\Ga)|(d-1)}\big((\cq)^{E(\Ga)},\bigcup_{\si,A}p^{-1}_{I_{\si,A}}(\vgsa{\si})\big)$, where the union ranges through all $\si\in\wt{S}_{E(\Ga)}$ and $A\subset\{\i\}\sqcup V(\Ga)$ of type 3 (same below). 
\begin{dfn}
\begin{align*}\Om'_\Ga(\pi):=&p_{e^\Ga_1}^*\Om(\pi)\cup\ldots\cup p_{e^\Ga_{|E(\Ga)|}}^*\Om(\pi)\in H^{|E(\Ga)|(d-1)}\big((\cq)^{E(\Ga)},\bigcup_{\si,A}p^{-1}_{I_{\si,A}}(\vgsa{\si})\big)\\
\Om_\Ga(\pi):=&\iota^*\Om'_\Ga(\pi)\in H^{|E(\Ga)|(d-1)}\big(X_\Ga(\pi),S(\pi)\big)
\end{align*}
\end{dfn}

\begin{prp}\label{naturalclass_prp}
Under the assumptions of Theorem \ref{main_thm}, let $\tl{h}_X:X_\Ga(\pi')\to X_\Ga(\pi'')$ be the $\cG$-bundle map as in Lemma \ref{Gbundlemaps_lmm}, then $\tl{h}_X^*(\Om_\Ga(\pi''))=\Om_\Ga(\pi')$. 
\end{prp}
\begin{proof}
By the second commutative diagram in Theorem \ref{main_thm}, $\tl{h}$ induces $$\tl{h}_q:\cqp{'}\lra\cqp{''}$$
which commutes with $q$ and restricts to the homeomorphism 
$h_S:S^{d-1}_{\pi'}\to S^{d-1}_{\pi''}$. 
It follows that $\tl{h}_q^*\Om(\pi'')=\Om(\pi')$. Since $\io$ commutes with $\tl{h}_X$ and $(\tl{h}_q)^{E(\Ga)}$, by the naturality statement in Lemma \ref{Y_lmm} below, $\tl{h}^*_X(\Om_\Ga(\pi''))=\Om_\Ga(\pi')$. 
\end{proof}

\begin{lmm}\label{Y_lmm}
Let $Y=Y_1\sx\ldots\sx Y_n$ be a product of paracompact Hausdorff spaces. Let $s,r\in\Z^{>0}$.  Suppose for all $i=1,\ldots,r$, we have $I_i=\{a^i_1,\ldots,a^i_{m_i}\}\subset\{1,\ldots,n\}$ and closed subset $V_{i}\subset Y_{a^i_1}\sx\ldots\sx Y_{a^i_{m_i}}$ satisfying the following condition: every open cover of $Y$ has a refinement of the form 
$$\cU_1\sx\ldots\sx\cU_n:=\{U_1\sx\ldots\sx U_n|U_j\in\cU_j\}, \tn{ where }\cU_j\tn{ is an open cover of }Y_j$$
such that 
\begin{itemize}
\vspace{-.5cm}
\item[$(\dagger)$]\label{st_condition} for all $i=1,\ldots,r$, $s'_1,\ldots s'_{m_i}\!=s\tn{ or }s\!-\!1$, with at most one of them being $s\!-\!1$,
$$\Big((U^0_{a^i_1}\cap\ldots\cap U^{s'_1}_{a^i_1})\sx\ldots\sx(U^0_{a^i_{m_i}}\cap\ldots\cap U^{s'_{m_i}}_{a^i_{m_i}})\Big)\cap V_{i}=\eset$$
where $U^0_{a^i_j},\ldots,U^{s'_j}_{a^i_j}$ are any pairwise distinct elements of $\cU_{a^i_j}$, for every $j=1,\ldots,m_i$. 
\vspace{-.5cm}
\end{itemize}
Denote $p_{I_i}\!:\!Y\to Y_{a^i_1}\sx\ldots\sx Y_{a^i_{m_i}}$ the projection. Then there is a map 
$$\Xi:H^s(Y_1)\otimes\ldots\otimes H^s(Y_n)\lra H^{sn}\big(Y,\bigcup_{i=1}^rp^{-1}_{I_i}(V_i)\big)$$ such that $\Xi$ composed with the restriction to $H^{sn}(Y)$ is the cup product. And $\Xi$ is natural in the following sense: if $Y'=Y'_1\sx\ldots\sx Y'_n,\{V'_i\subset Y'_{a^i_1}\sx\ldots\sx Y'_{a^i_{m_i}}\}_{i=1}^r$ satisfy the same condition as $Y,\{Y_i\},\{V_i\}$ above, with $\Xi'$ the corresponding map,  
and there are continuous maps $\{f_i:Y'_i\to Y_i\}_{i=1}^r$ such that $(f_{a^i_1}\sx\ldots\sx f_{a^i_{m_i}})(V'_i)\subset V_i$ for all $i$, then $\Xi'\circ(f_1^*\otimes\ldots\otimes f^*_n)=(f_1\sx\ldots\sx f_n)^*\circ\Xi$. 
\end{lmm}
\begin{proof}
First we define $\Xi$. Given $\si_1\in H^s(Y_1),\ldots,\si_n\in H^s(Y_n)$, take an open cover $\cU=\cU_1\sx\ldots\sx\cU_n$ of $Y$ satisfying $(\dagger)$ and such that for all $i$, $\si_i$ is represented by some skew-symmetric \v{C}ech cochain $\al_i\in\ch{C}^{s}_{\cU_i}(Y_i)$ (see, e.g., \cite{skewCech} for the equivalence between usual and skew-symmetric \v{C}ech cohomology; in particular, every (usual) \v{C}ech cohomology class has a skew-symmetric cochain representative). Denote by $p_i:Y\to Y_i$ projection to the $i$-th factor. Define 
\begin{equation}\label{alpha_eqn}
\al:=p_1^*\al_1\cup\ldots\cup p_n^*\al_n\in\ch{C}^{sn}_{\cU}(Y).
\end{equation}
We claim $\al|_{\bigcup_{i=1}^rp^{-1}_{I_i}(V_i)}=0$. Suppose $x\in p^{-1}_{I_i}(V_i)$ for some $i$ and 
$$x\in(U^0_1\sx\ldots\sx U^0_n)\cap(U^1_1\sx\ldots\sx U^1_n)\cap\ldots\cap(U^{sn}_1\sx\ldots\sx U^{sn}_n),\qquad U^j_i\in\cU_i,$$
then 
\begin{align*}
\al\big((U^0_1\sx\ldots\sx U^0_n),(U^1_1\sx\ldots\sx U^1_n),\ldots,(U^{sn}_1\sx\ldots\sx U^{sn}_n)\big)(x)\\
=\al_1(U^0_1,\ldots, U^{s}_1)(p_1(x))\cdot\ldots\cdot\al_n(U^{s(n-1)}_n\,\ldots, U^{sn}_n)(p_n(x))
\end{align*}
contains the following factor 
\begin{equation}\label{factor_eqn}
\al_{a^i_1}\big(U^{s(a^i_1-1)}_{a^i_1},\ldots, U^{sa^i_1}_{a^i_1}\big)(p_{a^i_1}(x))\cdot\ldots\cdot\al_{a^i_{m_i}}\big(U^{s(a^i_{m_i}-1)}_{a^i_{m_i}},\ldots, U^{sa^i_{m_i}}_{a^i_{m_i}}\big)(p_{a^i_{m_i}}(x)).
\end{equation}
Since $\al_1,\ldots,\al_n$ are skew-symmetric, for \eref{factor_eqn} to be non-zero, $U^{s(a^i_j-1)}_{a^i_j},\ldots,U^{sa^i_j}_{a^i_j}$ must be pairwise different for all $j=1,\ldots,m_i$. But $p_{I_i}(x)\in V_i$. So by $(\dagger)$, \eref{factor_eqn}$\neq0$ is impossible. This proves $\al|_{\bigcup_{i=1}^rp^{-1}_{I_i}(V_i)}=0$, so $\al\in\ch{C}^{sn}_{\cU}(Y,\bigcup_{i=1}^rp^{-1}_{I_i}(V_i))$. We define $$\Xi(\si_1,\ldots,\si_n)=\tn{the direct image of }[\al]\tn{ in }H^{sn}(Y,\bigcup_{i=1}^rp^{-1}_{I_i}(V_i)).$$
Then $\Xi$ is clearly linear. 

Next we show that $[\al]$ above does not depend on the choices of $\al_i$. Suppose $\al_i,\al'_i\in\ch{C}^s_{\cU_i}(Y_i)$ are both skew-symmetric cocycles, $[\al_i]=[\al'_i]=\si_i$, then there exists a skew-symmetric $\wt{\al}_i\in\ch{C}^{s-1}_{\cU_i}(Y_i)$, $d\wt{\al_i}=\al_i-\al'_i$. Define 
$$\wt{\al}:=p_1^*\al_1\cup\ldots\cup p_{i-1}^*\al_{i-1}\cup p_i^*\wt{\al}_i\cup p_{i+1}^*\al_{i+1}\cup\ldots\cup p_n^*\al_n\in\ch{C}^{sn-1}_{\cU}(Y).$$
By the same argument as above, $\wt\al$ vanishes on $\bigcup_{i=1}^r p^{-1}_{I_i}(V_i)$, so $\wt\al\in\ch{C}^{sn-1}_\cU(Y,\bigcup_{i=1}^r p^{-1}_{I_i}(V_i))$. Since all $\al_j$ are cocycles, 
$$d\wt\al=(-1)^i\big(p_1^*\al_1\cup\ldots\cup d(p_i^*\wt\al_i)\cup\ldots\cup p_n^*\al_n\big)=(-1)^i\big(p_1^*\al_1\cup\ldots\cup p_i^*(\al_i-\al'_i)\cup\ldots\cup p_n^*\al_n\big).$$
Therefore $[\al]$ does not depend on the choice of $\al_i$. 

We then show that $\Xi(\si_1,\ldots,\si_n)$ does not depend on the choices of $\cU_i$. Suppose $\cU'=\cU'_1\times\ldots\times\cU'_n$ also satisfy $(\dagger)$ and for all $i$, $\cU'_i$ is a refinement of $\cU_i$. Let $\{\mu_i:\cU'_i\to\cU_i\}_{i=1}^n$ be some refinement maps. Denote $\mu=(\mu_1,\ldots,\mu_n):\cU'\to\cU$. Then for $U^i_j\in\cU'_j$, 
\begin{align*}
&\mu^*\al\big((U^0_1\sx\ldots\sx U^0_n),(U^1_1\sx\ldots\sx U^1_n),\ldots,(U^{sn}_1\sx\ldots\sx U^{sn}_n)\big)(x)\\
&=\al\big((\mu_1(U^0_1)\sx\ldots\sx\mu_n(U^0_n)),(\mu_1(U^1_1)\sx\ldots\sx\mu_n(U^1_n)),\ldots,(\mu_1(U^{sn}_1)\sx\ldots\sx\mu_n(U^{sn}_n))\big)(x)\\
&=\al_1\big(\mu_1(U^0_1),\ldots,\mu_1(U^s_1)\big)(p_1(x))\cdot\ldots\cdot\al_n\big(\mu_n(U^{s(n-1)}_n),\ldots,\mu_n(U^{sn}_n)\big)(p_n(x))\\
&=\mu_1^*\al_1(U^0_1,\ldots, U^s_1)(p_1(x))\cdot\ldots\cdot\mu_n^*\al_n(U^{s(n-1)_n},\ldots,U^{sn}_n)(p_n(x)).
\end{align*}
This says $\mu^*\al=p_1^*(\mu_1^*\al_1)\cup\ldots\cup p^*_n(\mu_n^*\al_n)\in\ch{C}^{sn}_{\cU'}(Y,\bigcup_{i=1}^r p^{-1}_{I_i}(V_i))$. So, if we define $\Xi(\si_1,\ldots,\si_n)$ using $\cU'$, then it is the direct image of $\mu^*[\al]=[\mu^*\al]\in H^{sn}_{\cU'}(Y,\bigcup_{i=1}^r p^{-1}_{I_i}(V_i))$ which is the same as the direct image of $[\al]$. Now, if we assume $\cU'$ satisfies $(\dagger)$ but not necessarily a refinement of $\cU$, then by the assumption of the lemma we can find a common refinement of $\cU,\cU'$ that also satisfies $(\dagger)$. 
This implies that $\Xi(\si_1,\ldots,\si_n)$ is independent of the choice of $\cU$, and thus $\Xi$ is well-defined. 

Naturality follows immediately from the definition of $\Xi$: let $Y',\{Y_i'\},\{f_i\},\{V'_i\}$ be as in the lemma and $\cU,\al,\al_i$ as in the first paragraph of the proof, then $\cU':=f_1^{-1}\cU_1\sx\ldots\sx f_n^{-1}\cU_n$ satisfies $(\dagger)$. Take $\al_i'=f_i^*(\al_i)\in\ch{C}^s_{\cU'}(Y')$, then $\al'$ defined using \eref{alpha_eqn} is the same as $(f_1\sx\ldots\sx f_n)^*(\al)$ and the conclusion follows. 
\end{proof}
\begin{lmm}\label{X_lmm}
Every open cover of $(\cq)^{E(\Ga)}$ has a refinement of the form 
$$\cU_{e^\Ga_1}\sx\ldots\sx\cU_{e^\Ga_{|E(\Ga)|}}, \tn{ where each }\cU_{e^\Ga_i}\tn{ is an open cover of }\cq$$
such that 
\begin{itemize}
\vspace{-.5cm}
\item[$(*)$]
for all $A$ of type 3, $\si\in\tS_{E(\Ga)}$, $s'_e=d\!-\!1$ or $d\!-\!2$ with at most one being $d\!-\!2$, 
$$\ov{V}^\si_{\!A}\cap\prod_{e\in I_{\si,A}}(U^0_e\cap\ldots\cap U^{s'_e}_{e})=\eset$$
where $U^0_{e},\ldots,U^{s'_e}_{e}$ are any pairwise distinct elements of $\cU_{e}$, for every $e\in I_{\si,A}$. 
\end{itemize}
\end{lmm}
\begin{proof}
Since $A$ is of type 3, $|V(\Ga_A)|\leq 2/3|E(\Ga_A)|$. Since $d\ge3$, $(2d/3)\le d-1$, and $$\dim\,\ov{C}^\quo_{V(\Ga_A)}(\R^d)=d|V(\Ga_A)|-d-1\leq (2d/3)|E(\Ga_A)|-(d+1)\leq \dim\big((S^{d-1})^{E(\Ga_A)}\big)-4.$$ 
So by Lemma \ref{:(_lmm} $\ov{V}^\si_{\!A}$ is the image of a smooth map 
$f_{\si,A}:\ov{C}^\quo_{V(\Ga_A)}(\R^d)\to(S^{d-1})^{I_{\si,A}}$, where the domain is of codimension at least $4$ with respect to the target. (Indeed, codimension $2$ would suffice for this lemma.)

Fix a metric $D$ on $\cq$. 
Let $K_0\subset K_1\subset\ldots$ be a sequence of compact subsets of $\cq$, such that $K_0$ contains a neighborhood of $S^{d-1}_\pi$ and $\bigcup_{i=1}^\infty K_i=\cq$. 

By Lebesgue's Number Lemma, it suffices to show that for any $\epsilon_0,\epsilon_1,\ldots>0$, there exists $(\cU_e)_{e\in E(\Ga)}$ such that for all $e$ and $U\in\cU_e$, there exists $i$ such that $U\subset K_i$ and $\tn{diameter}_D(U)<\ep_i$, satisfying $(*)$. 

We first reduce the lemma to the following statement: for all $\ep>0$, there exist $(\cU_e)_{e\in E(\Ga)}$, where each $\cU_e$ is an open covers of $S_\pi^{d-1}$, such that for all $U\in\bigcup_e\cU_e$, $\tn{diameter}_D(U)<\ep/4$, satisfying $(*)$. 
Suppose this is true. Plug in $\epsilon_0$ for $\ep$. We can enlarge each $U\in\cU_e$ a little bit to get an open subset $l(U)$ of $\cq$, still contained in $K_0$ and of diameter less than $\ep_0$. (For example, take $l(U)=(\{x\in C_2(\pi)|D(x,U)<\ep_0/2\}\cap \mathring{K}_0)\cup U$.) Denote $l(\cU_e)=\{l(U)\}_{U\in\cU_e}$. 
We can cover $\cq-\bigcup_{U\in\cU_e}l(U)$ by locally finitely many open subsets of $C_2(\pi)$, such that each of these open subset is contained in some $K_i$ and is of diameter less than $\ep_i$; 
call this collection of open sets $\cU'_e$. Then for every $e$, $l(\cU_e)\cup\cU'_e$ is an open cover of $\cq$, in which each open set is contained in some $K_i$ and of diameter less than $\ep_i$. Since for every $A$, $\ov{V}^\si_{\!A}$ is contained in $(S^{d-1}_\pi)^{I_{\si,A}}$ whereas elements in $\cU'_e$ do not intersect $S^{d-1}_\pi$, that $(*)$ is satisfied by $\big(l(\cU_e)\cup\cU'_e\big)_{e\in E(\Ga)}$ follows from that it is satisfied by $(\cU_e)_{e\in E(\Ga)}$. 

We then reduce the statement at the beginning of last paragraph to the following: for any $\ep>0$ there exist triangulations $(T_e)_{e\in E(\Ga)}$ of $S^{d-1}$, compatible with the smooth structure on $S^{d-1}$, such that the diameter of each simplex is less than $\ep/4$, satisfying 
\begin{itemize}
\vspace{-.5cm}
\item[$(**)$]
for all $\si$, $A$ of type 3, and all collections of simplices $(S_{e}:\tn{a simplex in }T_e)_{e\in I_{\si,A}}$ such that $S_e$ is of dimension 0 or 1, at most one of dimension 1, 
$\ov{V}^\si_{\!A}\cap\prod_{e\in I_{\si,A}}\ov{S}_e=\eset.$
\end{itemize}
\vspace{-.5cm}
If this is true, we can take $\cU_e$ to be obtained from $T_e$ by slightly enlarging each top-dimensional simplex $S$ to an open neighborhood $U(S)$ of its closure, so that, 
\begin{itemize}
\vspace{-.5cm}
\item[(1)] (for a non-top dimensional simplex $S$, still denote by $U(S)=\bigcap_{S'}U(S')$ where $S'$ runs through the top-dimensional simplices that $S$ is a face of,) 
$$\ov{V}^\si_{\!A}\cap\prod_{e\in I_{\si,A}}U(S_e)=\eset \ \tn{ whenever }\  \ov{V}^\si_{\!A}\cap\prod_{e\in I_{\si,A}}\ov{S}_e=\eset;$$
\vspace{-.4cm}
\item[(2)] for any finite collection, $S_1,\ldots,S_k$, of simplices in $T_e$, $S_1\cap\ldots\cap S_k=\eset\implies U(S_1)\cap\ldots\cap U(S_k)=\eset$.
\end{itemize}
\vspace{-.5cm}
(These are clearly easy to satisfy. For (2), for every $S$ in $T_e$, take $U(S)$ contained in the union of the stars of the barycenters of its faces in the barycentric subdivision of $T$. For (1), take $U(S)$ contained in the $\ep'$-neighborhood of $S$, where $\ep'$ is the minimal distance between some $\ov{V}^\si_{\!A}$ and the union of products of closed simplices $\ov{V}^\si_{\!A}$ does not intersect.) We also require $U(S)$ to be contained in the $\ep/8$-neighborhood of $S$. Then $\{\cU_e\}$ satisfies $(*)$. 

It remains to prove the statement at the beginning of the last paragraph. Take arbitrary triangulations $(T^0_e)_{e\in E(\Ga)}$ of $S^{d-1}$ with diameter smaller than $\ep/8$, we perturb them simplex by simplex to satisfy $(**)$. The point is that every time we perturb a simplex away from $\ov{V}^\si_{\!A}$, we do it so slightly that no new unwanted intersection appears. The rest of the proof consists of technical details of this. 

Below by ``distance'' we mean the restriction of $D$ to $S^{d-1}_\pi$ and its products. 
By a triangulation $T$ of $S^{d-1}$ we mean a homeomorphism $T:|K_T|\to S^{d-1}$, where $K_T$ is a finite simplicial complex and $|K_T|$ its realization. 
Given $\cT=(T_e)_{e\in E(\Ga)}$ a tuple of triangulations of $S^{d-1}$, denote 
\vspace{-.2cm}
$$\ep(\cT)=\tn{minimal distance between }\prod_{e\in I_{\si,A}}\ov{S}_e\tn{ and }\ov{V}^\si_{\!A},\vspace{-.2cm}$$
where the ``minimal'' is taken over all $\si,A$, $(\ov{S}_e)_{e\in I_{\si,A}}$ such that $\ov{V}^\si_{\!A}\cap\prod_{e\in I_{\si,A}}\!\ov{S}_e=\eset$. There are only finitely many of them. 
And $\ep(\cT)>0$. 
Denote by $\cJ_1(\cT)$ (resp. $\cJ_0(\cT)$) the set of tuples $(\si,A,(x_e)_{e\in I_{\si,A}})$ such that $x_e$ is the image of a 0- or 1-simplex of $T_{e}$, exactly one of them (resp. none of them) being a 1-simplex, and $\vgsa{\si}\cap\prod_{e\in I_{\si,A}}\!x_e\neq\eset$. Then $\cJ_0(\cT),\cJ_1(\cT)$ are finite. 

Now take arbitrary triangulations $\cT^0=(T^0_e)_{e\in E(\Ga)}$ such that the diameter of any simplex is smaller than $\ep/8$. Denote by $\ep'$ the smallest diameter of a simplex of $\cT^0$. 
For the rest of the paragraph, for all $\cT$, define $\ep'(\cT)=\min\{\ep'/4,\ep(\cT)\}$. 
Take an element $(\si,A,(x_e)_{e\in I_{\si,A}})\in\cJ_0(\cT)$. Since $\vgsa{\si}$ cannot cover a neighborhood of $\prod_{e\in I_{\si,A}}\!x_e$, we can find $\prod_{e\in I_{\si,A}}\!x'_e$ in the $\ep'(\cT^0)/2$-neighborhood of it that does not meet $\vgsa{\si}$. For each $e$ we can perturb the triangulation map $T^0_{e}$ in a neighborhood of $(T^0_{e})^{-1}(x_e)$ to get $T^1_{e}$, in such a way that $T^1_{e}\big((T^0_{e})^{-1}(x_e)\big)=x'_e$, and the distance between old and new images of any point is less than $\ep'(\cT^0)/|E(\Ga)|$. For $e\notin I_{\si,A}$, define $T^1_e=T^0_e$. Define $\cT^1=(T^1_e)_{e\in E(\Ga)}$. Then $|\cJ_0(\cT^1)|<|\cJ_0(\cT^0)|$, $|\cJ_1(\cT^1)|\leq|\cJ_1(\cT^0)|$. We do this one by one for all elements of $\cJ_0(\cT^0)$, getting new triangulations $\cT^k=(T^k_e)_{e\in E(\Ga)}$ by the end. Now take $(\si,A,(x_e)_{e\in I_{\si,A}})\in\cJ_1(\cT^k)$ where $e_*\in I_{\si,A}$ is such that $x_{e_*}$ is the (image of a) 1-simplex. Let $p\!:\!\vgsa{\si}\subset (S^{d-1})^{I_{\si,A}}\to (S^{d-1})^{I_{\si,A}-e_*}$ be the projection forgetting the $e_*$-factor. Recall $f_{\si,A}\!:\!\ov{C}^\quo_{V(\Ga_A)}(\R^d)\to(S^{d-1})^{I_{\si,A}}$ defined at the beginning of this proof. By possibly perturbing $(x_e)_{e\in I_{\si,A}-e_*}$ in the same way as above, we can assume that $(x_e)_{e\in I_{\si,A}-e_*}$ is a regular value of $p\circ f_{\si,A}$. So $\vgsa{\si}\cap p^{-1}((x_e)_{e\in I_{\si,A}-e_*})\subset S^{d-1}$ is the smooth image of a manifold of dimension less than $d-2$. So we can perturb $T^k_{e_*}$ in a neighborhood of $(T^k_{e_*})^{-1}(x_{e_*})$, fixing some neighborhoods of the 0-simplices (since $\cJ_0=\eset$ now), to get a new triangulation $T^{k+1}_{e_*}$, so that $T^{k+1}_{e_*}\big((T^k_{e_*})^{-1}(x_{e_*})\big)$ does not intersect $\vgsa{\si}\cap p^{-1}((x_e)_{e\in I_{\si,A}-e_*}),$
and the new and old image of any point has distance less than $\ep'(\cT^k)/|E(\Ga)|$. For $e\in I_{\si,A}$ where $x_e$ is not perturbed, or $e\notin I_{\si,A}$, define $T^{k+1}_{e}=T^k_{e}$. Define $\cT^{k+1}=(T^{k+1}_e)_{e\in E(\Ga)}$. Then $|\cJ_1(\cT^{k+1})|<|\cJ_1(\cT^k)|$. Keep doing this one by one for all elements in $\cJ_1(\cT^k)$. By the end, for some $l$, we obtain a tuple of triangulations $\cT^{k+l}$ satisfying $(**)$. 
\end{proof}

\subsection{Defining Kontsevich's characteristic classes}\label{defK_sec}

To ``push forward'' $\Om_\Ga(\pi)$ to a cohomology class on the base $B$, the Leray-Serre spectral sequence is a convenient tool to formulate it. We follow \cite{HatcherSS} for the definition of Leray-Serre spectral sequence. 
First we make a general definition. 

\subsubsection{Cohomology push-forward}\label{cohomologypushforward_sec}
Suppose $B$ is a CW complex and $X\xra{\pi}B$ a fiber bundle with fiber $F$. Denote by $B_p$ the $p$-skeleton of $B$ and $X_p=\pi^{-1}(B_p)$. 
Suppose there is $k_0>0$ such that $H^{k}(F)=0$ for all $k>k_0$. Then, for any integers $n$ and $p<n-k_0$, in the Leray-Serre spectral sequence for $X_{p}\xra{\pi}B_p$, $E^{a,b}_2=0$ for all $a+b=n$, so $H^n(X_p)=0$. 

The Leray-Serre spectral sequence for $X\xra{\pi}B$ tells us the following (cf. \cite[Theorem 5.15]{HatcherSS}; note here we use the local coefficient version; cf. \cite{SteenrodLocalCoef}): suppose $n\geq k_0\in\Z$,
\begin{itemize}
\vspace{-.5cm}
\item $H^n(X)$ has a filtration by subgroups 
$F^n_p=\ker\big(H^n(X)\to H^n(X_{p-1})\big)$
and $E^{p,n-p}_\i\approx F^n_p/F^n_{p+1}$; 
\item $E_2^{p,q}\approx H^p(B;H^{q}(F))$, where the latter is understood as cohomology with local coefficients; 
\item $d_r\!:E_r^{p,q}\to E_r^{p+r,q+1-r}$.
\end{itemize}
\vspace{-.5cm}
Suppose $r\geq2$. 
Since $H^k(F)=0$ for all $k>k_0$, $$E_2^{n-k_0-r,k_0+r-1}\approx H^{n-k_0-r}\big(B,H^{k_0+r-1}(F)\big)=\{0\}.$$
Since $E_r^{p,q}$ is obtained from $E_2^{p,q}$ by taking subgroups and quotients, $E_r^{n-k_0-r,k_0+r-1}=\{0\}$ for all $r\ge2$. Therefore, all the $d_r$'s mapping into $E_r^{n-k_0,k_0}$ vanish and $E_\i^{n-k_0,k_0}$ is a subgroup of 
$E_2^{n-k_0,k_0}\approx H^{n-k_0}(B;H^{k_0}(F)).$
Since $H^n(X_{n-k_0-1})=0$, $H^{n}(X)=F^{n}_{n-k_0}$. This identifies a map (which we denote by $\pi_*$) \begin{equation}\label{spec0_eqn}
\pi_*:H^{n}(X)\lra H^{n-k_0}(B;H^{k_0}(F)).\end{equation}
\begin{dfn}
We call $\pi_*$ {\it cohomology push-forward} of the fiber bundle $X\xra{\pi}B$. 
\end{dfn}
By the naturality of Leray-Serre spectral sequence \cite[page~537-538]{HatcherSS}, $\pi_*$ does not depend on the choice of the CW structure on $B$, and is natural: suppose $X'\xra{\pi'} B'$ is another fiber bundle with fiber $F'$ such that $H^{k}(F')=0$ for all $k>k_0$ and $(\tl{f}:X'\to X,f:B'\to B)$ is a bundle map (so $H^{k_0}(F')\approx f^*H^{k_0}(F)$ as local systems over $B'$), then $f^*\circ\pi_*=\pi'_*\circ\tl{f}^*$. Using CW approximation, $\pi_*$ can be generalized to the case where $B$ is an arbitrary space. 

The above procedure can be generalized to the relative version: $X\xra{\pi}B$ has a subbundle $Y\xra{\pi}B$ with fiber $A\subset F$. Replacing $F$ with the pair $(F,A)$ and $X$ with $(X,Y)$ everywhere, everything goes through without change. 


\begin{rmk}\label{explicitpushforward_rmk}
(This remark will be used later in Section \ref{comparisonproof_subsec}.)
An explicit description of $\pi_*$ can be obtained by carefully unwinding the definition; we follow \cite{HatcherSS} for the construction of Leray-Serre spectral sequences and specifically, what we do below comes from the diagram on page 526 and $\Phi$ in the proof of Theorem 5.3 in \cite{HatcherSS}. Denote $p_0=n-k_0$ for simplicity. Given an element $\si\in H^n(X)$, first restrict it to $H^n(X_{p_0})$; the image will lie in the image of $H^n(X_{p_0},X_{p_0-1})$. Take such a preimage, say $\si'$. For each $p_0$-cell $e\!:(D^{p_0},\prt D^{p_0})\to (X_{p_0},X_{p_0-1})$ ($D^{p_0}$ being the standard $p_0$-dimensional ball), 
$e^*\si'\in H^n\big(e^*(X),(e|_{\prt D^{p_0}})^*(X)\big)$. 
Since $e^*\pi$ is trivializable, K\"{u}nneth formula gives 
$$H^n\big(e^*(X),(e|_{\prt D^{p_0}})^*(X)\big)\approx H^{p_0}(D^{p_0},\prt D^{p_0})\otimes H^{k_0}(F_b),$$
where $F_b=\pi^{-1}(e(b))$ is the fiber over an arbitrarily fixed point $b\in\mathring D^{p_0}$. Other fibers over $D^{p_0}$ are identified with $F_b$ via the trivialization. Notice that the K\"{u}nneth isomorphism does not depend on the choice of the trivialization of $e^*\pi$: the K\"{u}nneth map is determined by the projection map $e^*(X)\to F_b$ (note that $\prt D^{p_0}$ is not involved here); any two trivializations give homotopic projection maps since $D^{p_0}$ is contractible. Let us denote by $\si(e)\in H^{k_0}(F_b)$ the image of $e^*\si'$ under the K\"{u}nneth map, where $H^{p_0}(D^{p_0},\prt D^{p_0})\approx R$ identified using the canonical orientation on $D^{p_0}$. Then $\{e\to\si(e)\}_{e}$, where $e$ ranges through all $p_0$-cells of $B$, gives a cellular cochain on $B$ with coefficients in the local system $H^{k_0}(F)$. It is a cocycle and represents a cohomology class in $H^{p_0}(B;H^{k_0}(F))$, which is $\pi_*(\si)$. The relative version is similar. 
\end{rmk}
\subsubsection{Defining Kontsevich's characteristic classes}
\label{Kdfn_subsubsec}
Applying the construction above to the fiber bundle $(X_\Ga(\pi),S(\pi))\xra{\pi_X} B$ with $k_0=d|V(\Ga)|$, we get
\begin{equation}\label{spec_eqn}
{\pi_X}_*:H^{|E(\Ga)|(d-1)}\big(X_\Ga(\pi),S(\pi)\big)\lra H^{|E(\Ga)|(d-1)-d|V(\Ga)|}\big(B;H^{d|V(\Ga)|}(\tcM,S)\big).\end{equation}
The map $\rho:H^{d|V(\Ga)|}(X_\Ga,S)\to R$ in Definition \ref{rho_dfn} induces the corresponding map of local systems on $B$. 
So we get an induced map 
\begin{equation}\label{coeffientprojection_eqn}
H^{*}\big(B;H^{d|V(\Ga)|}\!(\tcM,S)\big)
\lra H^{*}(B;R).
\end{equation}

\begin{dfn}\label{K_dfn}
Define 
$K_{\Ga,\pi,F}\in H^{|E(\Ga)|(d-1)-d|V(\Ga)|}(B;R)$, to be the image of $\Om_\Ga(\pi)$ under \eref{spec_eqn} and \eref{coeffientprojection_eqn}. 
\end{dfn}
The corollary below is a direct consequence of Proposition \ref{naturalclass_prp} and the naturality of cohomology push-forward. 
\begin{crl}\label{K_crl}
Under the assumptions of Theorem \ref{main_thm}, $K_{\Ga,\pi',F'}=h_B^*K_{\Ga,\pi'',F''}$. 
\end{crl}

\section{Equivalence with the original definition}\label{equiv_sec}

In this section the coefficient ring $R=\R$. All open covers are assumed to be locally finite. The goal of this section is to prove the following 
\begin{prp}\label{compare_prp}
Suppose $(E\xra{\pi}B,F)$ is a framed smooth $(M,\i)$ bundle over a smooth manifold $B$, then $K_{\Ga,\pi,F}$ defined above agrees, up to scaling by a constant depending only on $\Ga$, with the usual definition of Kontsevich's characteristic classes for $(E\xra{\pi}B,F)$; see, for example, \cite{WatanabeAdd} for definition. 
\end{prp}

\subsection{\texorpdfstring{\v{C}}{C}ech to de Rham preliminary}\label{cechtoderham_sec}
First we state some general facts translating \v{C}ech to de Rham cohomology.
Let $Y$ be a smooth manifold. Denote by $\cA^q_Y$ the sheaf of differential $q$-form germs on $Y$, $\cZ^q_Y$ the subsheaf of closed $q$-form germs and $A^q(Y)$ the space of global $q$-forms on $Y$. Let $\cU$ be an open cover of $Y$. Let $\underline{l}=\{l_U:Y\to\R^{\ge0}\}_{U\in\cU}$ be a partition of unity subordinate to $\cU$. For any $p,q\in\Z^{\geq0}$ define
\begin{gather*}
h_{p,q}^{\underline{l}}:\ch{C}^p_\cU(Y;\cA^q_Y)\lra\ch{C}^{p-1}_\cU(Y;\cA^{q+1}_Y)\\
h_{p,q}^{\underline{l}}(\si)(U_0,\ldots, U_{p-1})=(-1)^p\sum_{U\in\cU} d\big(l_U\cdot\si(U,U_0,\ldots,U_{p-1})\big). 
\end{gather*}
By $l_U\cdot\si(U,\ldots,U_{p-1})$ we mean a form on $U_0\cap\ldots\cap U_{p-1}$ which is given by this formula on $U\cap U_0\cap\ldots\cap U_{p-1}$ and 0 elsewhere; it is smooth since $l_U$ vanishes in a neighborhood of $\prt U$. Clearly $\tn{image}(h_{p,q}^{\underline{l}})\subset\ch{C}^{p-1}_\cU(Y;\cZ^{q+1}_Y)$. 
For each $p$, define
$$h^{\underline{l}}:\ch{C}^p_\cU(Y;\R)\lra\ch{C}^0_{\cU}(Y;\cZ^p_Y),\qquad h^{\underline{l}}(\si)=(-1)^p\,(h^{\underline{l}}_{1,p-1}\circ h^{\underline{l}}_{2,p-2}\ldots\circ h^{\underline{l}}_{p,0})(\si).$$
By \cite[Proposition 9.8]{BottTu}, if $\si\in\ch{C}^p_\cU(Y;\R)$ is a \v{C}ech cocycle, then $h^\ul(\si)$ is a global closed form; if $\cU$ is such that any finite intersection of elements has trivial cohomology, then $\si\to h^\ul(\si)$ induces the canonical isomorphism between $\ch{H}^p(Y;\R)$ and $H^p_{\tn{de Rham}}(Y)$. If $\Phi$ is a family of supports on $Y$, the arguments in \cite[Section 8]{BottTu} still go through if all differential forms are assumed to have supports in $\Phi$ (i.e.,~ $\ch{C}^p_\cU$ is replaced with its subspace of $\Phi$-supported cochains $\ch{C}^p_{\cU,\Phi}(Y;\cA^q_Y)$); so $h^\ul$ still induces the canonical isomorphism between $\ch{H}^p_\Phi(Y;R)$ and $H^p_{\tn{de Rham},\Phi}(Y)$, where $H^p_{\tn{de Rham},\Phi}(Y)$ is the cohomology of the cochain complex of $\Phi$-supported differential forms on $Y$. 

Suppose $\cU,\cU'$ are two open covers of $Y$ and $\mu:\cU\to\cU'$ is a refinement. Let $\ul=\{l_U\}_{U\in\cU}$ be a partition of unity subordinate to $\cU$, then 
$$\mu_*\ul:=\Big\{l_{U'}:=\sum_{\begin{subarray}{c}U\in\cU\\ \mu(U)=U'\end{subarray}}l_U:Y\lra\R\Big\}_{U'\in\cU'}$$
is a partition of unity subordinate to $\cU'$. 
It is easy to check that 
\begin{equation}\label{hmucommute_eqn}
h^{\mu_*\ul}_{p,q}(\mu^*\si)=\mu^*(h^\ul_{p,q}(\si)) \tn{ for all } p,q \tn{ and }\si\in\ch{C}^p_{\cU'}(Y;\cA^q_Y).
\end{equation}

Define cup product 
\begin{gather*}
\cup:\ch{C}^{p_1}_\cU(Y;\cA^{q_1}_Y)\otimes\ch{C}^{p_2}_\cU(Y;\cA^{q_2}_Y)\lra\ch{C}^{p_1+p_2}_\cU(Y;\cA^{q_1+q_2}_Y)\\
(\si_1\cup\si_2)(U_0,\ldots,U_{p_1+p_2})=
(-1)^{q_1p_2}
\si_1(U_0,\ldots, U_{p_1})\wedge\si_2(U_{p_1},\ldots, U_{p_1+p_2}),
\end{gather*}
where the two forms on the RHS are restricted to $U_0\cap\ldots\cap U_{p_1+p_2}$. For simplicity we omit the notation for restriction; same below. 
When restricted to $\ch{C}^{p_1}_\cU(Y;\cZ^{0}_Y)\otimes\ch{C}^{p_2}_\cU(Y;\cZ^{0}_Y)$ this is the usual cup product for \v{C}ech cochains.

\begin{lmm}\label{hcup_lmm}
Let $Y=Y_1\times\ldots\times Y_m$ be a product of smooth manifolds. Denote by $\pi_i:Y\to Y_i$ the projection to the $i$-th factor. For every $i=1,\ldots,m$, let $\cU_i$ be an open cover of $Y_i$ and $\ul_i=\{l_U\}_{U\in\cU_i}$ be a partition of unity. Denote $\cU=\cU_1\times\ldots\times\cU_m$ the product open cover of $Y$, then 
$$\ul=\big\{l_{U_1\times\ldots\times U_m}:=(l_{U_1}\circ\pi_1)\cdot\ldots\cdot (l_{U_m}\circ\pi_m):Y\lra\R\big\}_{U_1\in\cU_1,\ldots,U_m\in\cU_m}$$
is a partition of unity of $Y$ subordinate to $\cU$. 
Let $p=\sum_{i=1}^mp_i$, $q=\sum_{i=1}^mq_i$ be non-negative integers. 
Let $\si_i\in\ch{C}^{p_i}_{\cU_i}(Y_i;\cZ^{q_i}_{Y_i})$ be \v{C}ech cocycles. Define 
$$\si=\pi_1^*\si_1\cup\ldots\cup\pi_m^*\si_m\in\ch{C}^p_\cU(Y;\cZ^{q}_Y).$$
Then, 
\begin{enumerate}
\vspace{-.5cm}
\itemsep=0em
\item if $m=2$,  
$$h^{\ul}_{p,q}(\si)=\begin{cases}
\pi_1^*h^{\ul_1}_{p_1,q_1}(\si_1)\cup\pi_2^*(\si_2), \tn{ if }p_1>0,\\
\pi_1^*\si_1\cup\pi_2^*h^{\ul_2}_{p_2,q_2}(\si_2), \tn{ if }p_1=0;
\end{cases}$$
\item if $q_i=0$ for all $i$, $h^\ul(\si)=\pi_1^*h^{\ul_1}(\si_1)\wedge\ldots\wedge \pi_m^*h^{\ul_m}(\si_m).$
\vspace{-.5cm}
\end{enumerate}
\end{lmm}
\begin{proof}
This is direct computation. For (1), when $p_1>0$, 
\begin{align*}
&h^\ul_{p,q}(\si)(U^0_1\times U^0_2,\ldots, U^{p-1}_1\times U^{p-1}_2)\\
&=(-1)^p\sum_{U_1\in\cU_1,U_2\in\cU_2}\big(
(\pi_1^*dl_{U_1})(l_{U_2}\circ\pi_2)+(l_{U_1}\circ\pi_1)(\pi^*_2dl_{U_2})\big)\wedge\\
&\hspace{3cm}(-1)^{q_1p_2}\pi_1^*\si_1(U_1, U_1^0,\ldots, U_1^{p_1-1})\wedge\pi_2^*\si_2(U_2^{p_1-1},\ldots, U_2^{p-1})\\
&=(-1)^{p+q_1p_2}\bigg(\Big(
\sum_{U_1}(\pi_1^*dl_{U_1})\wedge\pi_1^*\si_1(U_1,\ldots,U_1^{p_1-1})\Big)\wedge\Big(\sum_{U_2}(l_{U_2}\circ\pi_2)\pi_2^*\si_2(U_2^{p_1-1},\ldots, U_2^{p-1})\Big)\\
&+\Big(\sum_{U_1}(-1)^{q_1}(l_{U_1}\circ\pi_1)\pi_1^*\si_1(U_1,\ldots, U_1^{p_1-1})\Big)\wedge\Big(\underbrace{\sum_{U_2}(\pi_2^*dl_{U_2})}_{=0}\wedge\pi_2^*\si_2(U_2^{p_1-1},\ldots, U_2^{p-1})\Big)
\bigg)\\
&=(-1)^{p+p_1+q_1p_2}\pi_1^*h^{\ul_1}_{p_1,q_1}(\si_1)(U^0_1,\ldots, U^{p_1-1}_1)\wedge\pi_2^*\si_2(U_2^{p_1-1},\ldots, U_2^{p-1})\\
&=(\pi_1^*h^{\ul_1}_{p_1,q_1}(\si_1)\cup\pi_2^*\si_2)(U^0_1\times U^0_2,\ldots,U^{p-1}_1\times U^{p-1}_2). 
\end{align*}
When $p_1=0$, 
\begin{align*}
&h^\ul_{p,q}(\si)((U^0_1\times U^0_2),\ldots, (U^{p-1}_1\times U^{p-1}_2))\\
&=(-1)^{p+q_1p_2}\hspace{-.5cm}\sum_{U_1\in\cU_1,U_2\in\cU_2}\hspace{-.5cm}\big(
(\pi_1^*dl_{U_1})(l_{U_2}\circ\pi_2)+(l_{U_1}\circ\pi_1)(\pi^*_2dl_{U_2})\big)\wedge
\pi_1^*\si_1(U_1)\wedge\pi_2^*\si_2(U_2, U_2^{0},\ldots, U_2^{p-1})\\
&=(-1)^{p+q_1p_2}\bigg(\Big(
\underbrace{\sum_{U_1}(\pi_1^*dl_{U_1})}_{=0}\wedge\pi_1^*\si_1(Y_1)\Big)\wedge\Big(\sum_{U_2}(l_{U_2}\circ\pi_2)\pi_2^*\si_2(U_2, U_2^{0},\ldots, U_2^{p-1})\Big)\\
&\quad+\Big(\sum_{U_1}(-1)^{q_1}(l_{U_1}\circ\pi_1)\pi_1^*\si_1(Y_1)\Big)\wedge\Big(\sum_{U_2}(\pi_2^*dl_{U_2})\wedge\pi_2^*\si_2(U_2, U_2^{0},\ldots, U_2^{p-1})\Big)
\bigg)\\
&=(-1)^{p+q_1+p_2+q_1p_2}\pi_1^*\si_1(Y_1)\wedge\pi_2^*h^{\ul_2}_{p_2,q_2}(\si_2)(U_2, U_2^{0},\ldots, U_2^{p-1})\\
&=(\pi_1^*\si_1\cup\pi_2^*h^{\ul_2}_{p_2,q_2}(\si_2))(U^0_1\times U^0_2,\ldots,U^{p-1}_1\times U^{p-1}_2);
\end{align*}
notice that we can write $\si_1(Y_1)$ because $\si_1$ is a degree 0 \v{C}ech cocycle. 

For (2), in the case $m=2$, 
\begin{align*}
h^\ul(\si)=&(-1)^p\,(h^\ul_{1,p-1}\circ\ldots\circ h^\ul_{p,0})(\pi_1^*\si_1\cup\pi_2^*\si_2)\\
=&(-1)^{p}\,(h^\ul_{1,p-1}\circ\ldots\circ h^\ul_{p_2,p-p_2})\big(\pi_1^*(h^{\ul_1}_{1,p_1-1}\circ\ldots\circ h^{\ul_1}_{p_1,0})(\si_1)\cup\pi_2^*\si_2\big)\\
=&(-1)^{p+p_1}\,(h^\ul_{1,p-1}\circ\ldots\circ h^\ul_{p_2,p-p_2})\big(\pi_1^*h^{\ul_1}(\si_1)\cup\pi_2^*\si_2\big)\\
=&(-1)^{p+p_1}\,\big(\pi_1^*h^{\ul_1}(\si_1)\cup\pi_2^*(h^{\ul_2}_{1,p_2-1}\circ\ldots\circ h^{\ul_2}_{p_2,0})\si_2\big)\\
=&\pi_1^*h^{\ul_1}(\si_1)\cup \pi_2^*h^{\ul_2}(\si_2)=\pi_1^*h^{\ul_1}(\si_1)\wedge \pi_2^*h^{\ul_2}(\si_2). 
\end{align*}
The general case follows by induction. 
\end{proof}

We next check that $h$ is natural. Let $f:X\to Y$ be a smooth map between smooth manifolds. For an open cover $\cU$ on $Y$ with partition of unity $\ul=\{l_U\}_{U}$, $f^*(\cU):=\{f^{-1}(U)\}_{U\in\cU}$ is an open cover of $X$ and $f^*\ul:=\{l_{f^{-1}(U)}:=l_U\circ f:X\to\R^{\ge0}\}_{U\in\cU}$ is a partition of unity. We have the pull-back map $f^*:\ch{C}^p_\cU(Y;\cZ^q_Y)\to\ch{C}^p_{f^*\cU}(X;\cZ^q_X)$, and 
\begin{align}\label{hnatural_eqn}
(f^*h^\ul\si)(f^{-1}(U_0),\ldots,& f^{-1}(U_{p-1}))=\sum_{U\in\cU}(-1)^pd\big((l_U\circ f)\wedge f^*(\si(U, U_0,\ldots, U_{p-1}))\big)\nonumber\\
&=\sum_{U\in\cU}(-1)^pd\big(l_{f^{-1}(U)}\cdot(f^*\si)(f^{-1}(U),f^{-1}(U_0),\ldots,f^{-1}(U_{p-1}))\big)\nonumber\\
&=(h^{f^*\ul}f^*\si)(f^{-1}(U_0),\ldots,f^{-1}(U_{p-1})).
\end{align}
Notice that everything in this subsection works for cohomology with supports as well. 

\subsection{Proof of Proposition \texorpdfstring{\ref{compare_prp}}{5.1}}
\label{comparisonproof_subsec}

\subsubsection{Careful choices of open cover refinements in \texorpdfstring{$\ov{C}^{E(\Ga)}_2(\pi)$}{the ambient space}}
\label{521_subsubsec}
Let $\cU'$ be an open cover of $\cq$ such that there exists $\ch\om\in\ch{C}^{d-1}_{\cU'}(\cq)$ a cocycle representative of $\Om(\pi)$ ($\Om(\pi)$ is defined in Definition \ref{propagator_dfn}).
We will make a careful choice of an open cover and a partition of unity on $\ov{C}_2(\pi)$. Note the following commutative square, where $\hat\io,\io$ are inclusion maps: 
\[ \begin{tikzcd}
\prt^v\ov{C}_2(\pi)\rar["\hat\io"]\dar["F"]&\ov{C}_2(\pi)\dar["q"]\\
S^{d-1}\rar["\io"]&\cq. 
\end{tikzcd}\]
\begin{lmm}\label{C2cover_lmm}
There exist an open cover $\cU$ of $\ov{C}_2(\pi)$ such that all non-empty intersections of its elements are contractible,
a refining map $\mu:\cU\to q^*\cU'$, and partitions of unity $\ul$ on $\ov{C}_2(\pi)$ subordinate to $\cU$, $\ul_S$ on $S^{d-1}$ subordinate to $\io^*\cU'$, such that $\hat\io^*(\mu_*\ul)=F^*\ul_S$ (they are both subordinate to $\hat\io^*q^*\cU'=F^*\io^*\cU'$ on $\prt^v\ov{C}_2(\pi)$). 
\end{lmm}
\begin{proof}
By the way $F$ is defined, it can be extended to a smooth map $F:N_\prt\to S^{d-1}$, where $N_\prt\subset\ov{C}_2(\pi)$ is a neighborhood of $\prt^v\ov{C}_2(\pi)$. Still denote by $\hat\io:N_\prt\to\ov{C}_2(\pi)$ the inclusion. Let $\ul_S$ be a partition of unity on $S^{d-1}$ subordinate to $\io^*\cU'$. 
Then $F^*\ul_S=:\{F^*\ul_S(U)\}_{U\in\cU'}$, where $F^*\ul_S(U)$ is supported on $q^{-1}(U)\cap N_\prt$,  
is a partition of unity on $N_\prt$ subordinate to $\hat\io^*q^*\cU'$. 
Our next goal is to find a partition of unity $\ul'$ on $\ov{C}_2(\pi)$ subordinate to $q^*\cU'$, such that $\hat\io^*\ul'=F^*\ul_S$. 

We can find a smooth function $g: \ov{C}_2(\pi)\to \R^{\ge0}$ supported in $N_\prt$ such that $g|_{\prt^v{C}_2(\pi)}\equiv1$: 
let $K\subset C_2(\pi)$ be a compact subset containing $C_2(\pi)-N_\prt$, and let $N'_\prt\subsetneq\ov{C}_2(\pi)-K$ be a neighborhood of $\prt^v\ov{C}_2(\pi)$; 
let $\{V_i\subset \ov{C}_2(\pi)-N'_\prt\}_{i\in I}$ be an open cover of $K$; 
let $\{{g_i}\}_{i\in I}\sqcup\{g'\}$, where $g_i$ is supported in $V_i$ and $g'$ is supported in $C_2(\pi)-K$, be a partition of unity on $C_2(\pi)$ subordinate to $\{V_i\}_i\sqcup\{C_2(\pi)-K\}$, 
then $(\sum_ig_i)|_K\equiv1$ and $(\sum_ig_i)|_{N'_\prt\cap C_2(\pi)}\equiv0$; 
so $g:=1-\sum_ig_i$, extended by 1 to $\prt^v\ov{C}_2(\pi)$, satisfies the requirement. 
For each $U\in\cU'$, define $h_U=g\cdot F^*\ul_S(U):N_\prt\to\R$, then it can be smoothly extended to the entire $\ov{C}_2(\pi)$, taking value 0 out of $N_\prt$; 
so  $h_U|_{\prt^v\ov{C}_2(\pi)}=F^*\ul_S(U)|_{\prt^v\ov{C}_2(\pi)}$, $h_U$ is supported in $q^{-1}(U)$ and $(\sum_Uh_U)|_{N'_\prt}\equiv1$. 

Let $K'\subset C_2(\pi)$ be a compact subset containing $C_2(\pi)-N'_\prt$. Let $\{G_U\subset q^{-1}(U)\cap C_2(\pi)\}_{U\in\cU'}$ be compact subsets that still cover $K'$. For each $U\in\cU'$, take $\phi_U:\ov{C}_2(\pi)\to\R^{\ge0}$ that is supported in $q^{-1}(U)\cap C_2(\pi)$ and $\phi_U|_{G_U}\equiv1$. Then $\sum_{U\in\cU'}(\phi_U+h_U)$, as a function on $\ov{C}_2(\pi)$, is positive everywhere and equals to $1$ on $\prt^v\ov{C}_2(\pi)$. Define $l'_U=(\phi_U+h_U)/(\sum_U(\phi_U+h_U))$. Then $l':=\{l'_U\}_{U\in\cU'}$ is a partition of unity as required. 

Next we define $\cU$. Fix a Riemannian metric on $\ov{C}_2(\pi)$ (which is a smooth manifold with boundary and corners). 
For every $U\in\cU'$, let $\{V^U_i\subset q^{-1}(U)\subset\ov{C}_2(\pi)\}_{i\in I_U}$ be a finite collection of geodesically convex open subsets, 
such that $\tn{supp}(l'_U)\subset\bigcup_{i\in I_U}V^U_i$. Then $\cU:=\{V^U_i\}_{U\in \cU',i\in I_U}$ is an open cover of $\ov{C}_2(\pi)$ whose intersections are all contractible. 
Define $\mu:\cU\to q^{*}\cU'$, $\mu(V^U_i)=q^{-1}(U)$; it is a refinement map. 

Now, we construct a partition of unity $\ul$ on $\ov{C}_2(\pi)$ subordinate to $\cU$ such that $\mu_*\ul=\ul'$.
Using the same argument used to find the $g_i$s above, for every $U\in\cU'$ we can find smooth functions $\psi^U_i:\ov{C}_2(\pi)\to\R^{\ge0}$ supported in $V^U_i$, for every $i\in I_U$, such that $(\sum_{i\in I_U}\psi^U_i)|_{\tn{supp}(l'_U)}\equiv1$. Define $l^U_i=\psi^U_i\cdot l'_U$. Then $\sum_{i\in I_U}l^U_i=l'_U$. So $\ul:=\{l^U_i\}_{U\in\cU',i\in I_U}$ is a partition of unity as required. 
\end{proof}
\begin{crl}
$h^\ul(\mu^*q^*\ch\om)$ is a closed $(d\!-\!1)$-form on $\ov{C}_2(\pi)$ such that  $\hat\io^*h^\ul(\mu^*q^*\ch\om)=F^*\al$ for some closed form $\al\in A^{d-1}(S^{d-1})$, $\int_{S^{d-1}}\al=1$. In other words, $h^\ul(\mu^*q^*\ch\om)$ is a propagator. 
\end{crl}

\begin{proof}
That it is a closed $(d\!-\!1)$-form is clear. And 
$$\hat\io^*h^\ul(\mu^*q^*\ch\om)\stackrel{(\ref{hmucommute_eqn})}{=}
\hat\io^*h^{\mu_*\ul}(q^*\ch\om)\stackrel{(\ref{hnatural_eqn})}{=}
h^{\hat\io^*\mu_*\ul}(\hat\io^*q^*\ch\om)\stackrel{\tn{Lemma }\ref{C2cover_lmm}}{=}
h^{F^*\ul_S}(F^*\io^*\ch\om)\stackrel{(\ref{hnatural_eqn})}{=}
F^*h^{\ul_S}(\io^*\ch\om).$$
Since $[\ch\om]=\Om(\pi)\in H^{d-1}(\cq)$ and $h^{\ul_S}$ induces the canonical isomorphism between \v{C}ech and de Rham cohomology, $[h^{\ul_S}(\io^*\ch\om)]\in H^{d-1}(S^{d-1};\R)$ is the Poincar\'e dual of the point class by the definition of $\Om(\pi)$. So $\int_{S^{d-1}}h^{\ul_S}(\io^*\ch\om)=1$. 
\end{proof}

Now, let $\{\cU'_e\}_{e\in E(\Ga)}$ be a collection of open covers of $\cq$ given by Lemma \ref{X_lmm}. For every $e$, applying the above argument with $\cU'$ replaced by $\cU'_e$, we get an open cover $\cU_e$ with refinement map $\mu_e:\cU_e\to q^*\cU'_e$ and partition of unity $\ul^e$ on $\ccEt$ subordinate to $\cU_e$, such that $h^{\ul^e}(\mu_e^*q^*\ch{\om})$ is a propagator. Denote $\om_e=h^{\ul^e}(\mu_e^*q^*\ch{\om})$. 
Denote by $\pr_e:\ov{C}_2(\pi)^{E(\Ga)}\to\ov{C}_2(\pi)$
the projection to the $e$-th factor. (Recall $\ov{C}^{E(\Ga)}_2(\pi)$ is the fiber product and $\ov{C}_2(\pi)^{E(\Ga)}$ is the direct product of the total space.) 
Denote 
$$\wt\cU:=\pr_{e^\Ga_1}^*{\cU_{e^\Ga_1}}\times\ldots\times \pr^*_{e^\Ga_{|E(\Ga)|}}\!\!\cU^\Ga_{e^\Ga_{|E(\Ga)|}}$$
the product open cover on $\ov{C}_2(\pi)^{E(\Ga)}$. 
Define (that it is supported away from $S(\pi)$ follows from the choice of $\{\cU'_e\}$ given by Lemma \ref{X_lmm})
$$\ch\om_\Ga':=\pr_{e^\Ga_1}^*\mu_{e^\Ga_1}^*\com\cup \pr_{e^\Ga_2}^*\mu_{e^\Ga_2}^*\com\cup\ldots\cup \pr^*_{e^\Ga_{|E(\Ga)|}}\mu_{e^\Ga_{|E(\Ga)|}}^*\com\in
\ch{C}^{|E(\Ga)|(d-1)}_{\wt\cU}\big(\ccEt^{E(\Ga)},S(\pi)\big);$$
its restriction to $X_\Ga$ represents the class $\Om_\Ga(\pi)\in H^{|E(\Ga)|(d-1)}(X_\Ga(\pi),S(\pi))$.
Define
$$\om_\Ga=\pr_{e^\Ga_1}^*\om_{e^\Ga_1}\wedge\ldots\wedge\pr_{e^\Ga_{|E(\Ga)|}}^*\om_{e^\Ga_{|E(\Ga)|}}\in A^{|E(\Ga)|(d-1)}(\ccEt^{E(\pi)});$$
then the push-forward of $\om_\Ga|_{C_\Ga(\pi)}$ to $B$  represents Kontsevich's class in the usual definition. 
By Lemma \ref{hcup_lmm}, $\om_\Ga=h^{\tl\ul}\ch{\om}'_\Ga$, 
where $\tl\ul$ is the partition of unity on $\ccEt^{E(\Ga)}$ subordinate to $\wt\cU$ given by taking ``product'' of the $\ul^e$s as in Lemma \ref{hcup_lmm}. 
Below we denote the restriction of $\om_\Ga$ to $\ccEft$ still by $\om_\Ga$. 

Fix a triangulation on $B$ and denote by $B_p$ the $p$-skeleton of $B$ with respect to this triangulation. 
Denote 
$$p_0=|E(\Ga)|(d-1)-d|V(\Ga)|=\deg \Om_\Ga(\pi)-\dim\, X_\Ga.$$
Recall that $f=(f_e)_{e\in E(\Ga)}:\ov{C}_{V(\Ga)}(\pi)\to\ccEft$ is the forgetful map, and $\phi$ is the factor-permuting action of $\wt{S}_{E(\Ga)}$ on $\ccEft$. 
Denote by 
$$\pi_V:\ov{C}_{V(\Ga)}(\pi)\lra B,\qquad\wh\pi:\ov{C}_2^{E(\Ga)}(\pi)\lra B,\qquad \pi_X:X_\Ga(\pi)\lra B$$
the bundle projection maps (the first two were both denoted by $\pi$; here we want to distinguish them to avoid confusion).  
Define $$\rxg=\bigcup_{\si\in\wt{S}_{E(\Ga)}}\phi(\si)\big(C_\Ga(M,\i)\big), \quad T=X_{\Ga}-\rxg\quad\subset\ccM{2}^{E(\Ga)},$$
and denote $\rxg(\pi),T(\pi)\subset\ccEft$ the bundle version of them. For all $p$, denote $X_p=\wh\pi^{-1}(B_p)\cap X_\Ga(\pi)$.\footnote{
The reason for defining $T$ instead of using the previously defined $T_1$ is the following: we did not establish any regarding smoothness in Section \ref{structureX_sec} when gluing the various copies of $\ccM{\Ga}$ together to form $\wt{X}_\Ga$; now we are working with differential forms, we want our \v{C}ech cochains to be supported away from this ``joint'' part, as shown in the next paragraph.}

By Lemma \ref{refinecover_lmm}, we can find an open cover $\wh\cU$ of $\ccEt^{E(\Ga)}$ refining $\wt\cU$, such that there exists a neighborhood $N$ of $\big(T(\pi)\cap X_{p_0}\big)\cup X_{p_0-1}$ in $\ccEt^{E(\Ga)}$, satisfying 
\begin{equation*}
U_0\cap\ldots\cap U_{|E(\Ga)|(d-1)}\cap N=\eset, \quad\forall\, U_0\neq\ldots\neq U_{|E(\Ga)|(d-1)}\in\wh\cU.
\end{equation*}
Let $\hat\mu:\wh\cU\to\wt\cU$ be a refinement map. Then $\hat\mu^*\ch{\om}'_\Ga$ is supported away from $N$. Define $\ch{\om}_\Ga=(\hat\mu^*\ch{\om}'_\Ga)|_{X_\Ga(\pi)}$. 

Let $\hat\ul$ be a partition of unity on $\ccEt^{E(\Ga)}$ subordinate to $\wh\cU$; then $\hat\mu_*\hat\ul$ is a partition of unity subordinate to $\wt\cU$. Define $\bar\om_\Ga:=h^{\hat\mu_*\hat\ul}(\ch\om'_\Ga)=h^{\hat\ul}(\hat\mu^*\ch\om'_\Ga)$. 
Then $\bar\om_\Ga|_{X_\Ga(\pi)}=h^{\hat\ul|_{X_\Ga(\pi)}}(\ch{\om}_\Ga)$. 
Since all the intersections of elements of $\wt\cU$ are contractible (since the same is true for each $\cU_e$ and $\wt\cU$ is their product), both $h^{\hat\mu_*\hat\ul}$ and $h^{\wt\ul}$ induce the isomorphism between \v{C}ech and de Rham cohomology (here we let the family of supports be the collection of compact subsets in $\ccEt^{E(\pi)}$ that do not intersect $S(\pi)$), 
so $$[\bar\om_\Ga]=[\ch\om'_\Ga]=[\om_\Ga]\in H^{*}(\ccEt^{E(\Ga)},S(\pi)).$$
Denote the restriction of $\bar\om_\Ga$ to $\ccEft$ still by by $\bar\om_\Ga$, then, pulling back the above equation by restriction, $[\bar\om_\Ga]=[\om_\Ga]\in H^{*}(\ccEft,S(\pi))$.

\subsubsection{Passing to the simplicial cohomology on \texorpdfstring{$B$}{B}}
\label{522_subsubsec}

Define $s(\si)=(d-1)\sgn(\si)+d\,\sgn'(\si)$. For a differential form $\al\in A^m(\ccEft)$, define $\pi_*^s(\al)$ to be the degree $(m-d|V(\Ga)|)$ simplicial cochain on $B$ that sends a dimension $m-d|V(\Ga)|$ simplex $\De$ to 
$$
\sum_{\si\in\wt{S}_{E(\Ga)}}(-1)^{s(\si)}\int_{\pi_V^{-1}(\De)
}(\phi(\si)\circ f)^*\al
=\sum_{\si\in\wt{S}_{E(\Ga)}}(-1)^{s(\si)}\int_{\pi_V^{-1}(\De)}f^*\phi(\si)^*\al.
$$
Then, since $\phi(\si)^*\om_\Ga=(-1)^s\om_\Ga$, 
$[\pi^s_*\om_\Ga]\in H^{p_0}(B)$ is $2^{|E(\Ga)|}|E(\Ga)|!$ times Kontsevich's class in the usual definition. 

Recall that $\pi_{\wt{X}}:\wt{X}_\Ga(\pi)\to B$ is the bundle version of Definition \ref{Xtl_dfn}, and so is $\wt{f}:\wt{X}_\Ga(\pi)\to\ccEft$. 
In Definition \ref{Xtl_dfn}, by choosing a collar neighborhood when gluing the copies of $C'_\Ga(M,\i)$, $\wt{X}_\Ga$ can be given a smooth structure so that $\wt{f}:\wt{X}_\Ga(\pi)\to\ccEft$ is piecewise smooth (smooth away from $\wt{f}^{-1}(T(\pi))$); so pulling back a differential form by $\wt{f}$ is still well defined (the result would be a piecewise-smooth form), and the usual differential form push-forward $(\pi_{\wt X})_*$ is also well defined for these forms (by integrating along each piece of the fiber and summing up). 
It follows from definition that $\pi^s_*(\al)=(\pi_{\wt X})_*\wt{f}^*\al$, 

\begin{lmm}\label{Kclasscoho_lmm}
Suppose $\al_1,\al_2\in A^*(\ccEft)$ are such that 
there exists $\wt\al\in A^{*-1}(\ccEft)$ supported away from $S(\pi)$ and $d\wt\al=\al_1-\al_2$, then $\pi^s_*\al_1-\pi^s_*\al_2$ is a coboundary. 
\end{lmm}
\begin{proof}
For a simplex $\De$ in $B$, 
\begin{align*}
\sum_{\si\in\wt{S}_{E(\Ga)}}(-1)^{s(\si)}\int_{\pi_V^{-1}(\De)}(\phi(\si)\circ f)^*(\al_1-\al_2)=\sum_{\si\in\wt{S}_{E(\Ga)}}(-1)^{s(\si)}\int_{\prt\pi_V^{-1}(\De)}(\phi(\si)\circ f)^*\wt\al\\
=\sum_{\si\in\wt{S}_{E(\Ga)}}(-1)^{s(\si)}\int_{\pi_V^{-1}(\prt\De)}(\phi(\si)\circ f)^*\wt\al+
\sum_{\si\in\wt{S}_{E(\Ga)}}(-1)^{s(\si)}\sum_{A}\int_{\pi_V^{-1}(\De)\cap\cS_A}(\phi(\si)\circ f)^*\wt\al,
\end{align*}
where $A\subset V(\Ga)\cup\{\i\}$, $|A|\ge2$. The second term vanishes: for $A$ of type 1, $f(\cS_A)$ has smaller dimension; for $A$ of type 2 or 4, they cancel with each other when summed over $\si$; for $A$ of type 3, because $\wt\al$ vanishes on $S(\pi)$ by assumption. So, $$(\pi_*^s\al_1-\pi_*^s\al_2)(\De)=(\pi_*^s\wt\al)(\prt\De)=(\de\,\pi_*^s\wt\al)(\De).$$
\end{proof}

Since $[\bar\om_\Ga]=[\om_\Ga]\in H^{*}(\ccEft,S(\pi))$ (last sentence before Section \ref{522_subsubsec}) and $[\pi^s_*\om_\Ga]\in H^{p_0}(B)$ is $2^{|E(\Ga)|}|E(\Ga)|!$ times Kontsevich's class in the usual definition (last sentence in the first paragraph of Section \ref{522_subsubsec}), 
by Lemma \ref{Kclasscoho_lmm}, $\pi_*^s(\bar\om_\Ga)$ also represents $2^{|E(\Ga)|}|E(\Ga)|!$ times Kontsevich's class in the usual definition. 

We apply Remark \ref{explicitpushforward_rmk} to the situation here. For each $p_0$-simplex $\De^{p_0}$ in $B$, $\ch\om_\Ga|_{\pi_X^{-1}(\De^{p_0})}$ is supported away from $(T(\pi)\cap\pi_X^{-1}(\De^{p_0}))\cup \pi_X^{-1}(\prt\De^{p_0})$. 
So $\ch\om_\Ga|_{\pi_X^{-1}(\De^{p_0})}$ is a cocycle in 
$$\ch{C}^{|E(\Ga)|(d-1)}_{\wh\cU|_{\pi_X^{-1}(\De^{p_0})}}\big(X_\Ga(\pi)|_{\De^{p_0}}, X_\Ga(\pi)|_{\prt\De^{p_0}}\cup T(\pi)|_{\De^{p_0}}\big).$$ 
By K\"{u}nneth formula, 
$$H^{|E(\Ga)|(d-1)}\big(X_\Ga(\pi)|_{\De^{p_0}},X_\Ga(\pi)|_{\prt\De^{p_0}}\!\cup T(\pi)|_{\De^{p_0}}\big)\approx H^{p_0}(\De^{p_0},\prt\De^{p_0})\otimes H^{d|V(\Ga)|}(X_{\Ga,b},T_{b}),$$
where $b\in\mathring\De^{p_0}$ is an arbitrary point and $(X_{\Ga,b},T_{b})$ is the fiber of $(X_\Ga(\pi),T(\pi))$ over $b$. Let $\chi'(\De^{p_0})\in H^{d|V(\Ga)|}(X_{\Ga,b},T_{b})$ be such that $[\ch\om_\Ga|_{\pi_X^{-1}(\De^{p_0})}]\approx1\otimes\chi'(\De^{p_0})$ under the K\"{u}nneth isomorphism, where $1\in\R\approx H^{p_0}(\De^{p_0},\prt\De^{p_0})$ (notice the identification uses the orientation of $B$). 
Then $\{\De^{p_0}\to\chi'(\De^{p_0})\}_{\De^{p_0}}$ is a simplicial cocycle on $B$ with coefficients in the local system $H^{d|V(\Ga)|}(X_\Ga,T)$. Its restriction to $H^{d|V(\Ga)|}(X_\Ga,T_1)$ (note $T_1\subset T$) represents $(\pi_X)_*\Om_\Ga(\pi)$. Let $\chi(\De^{p_0})=\wt{f}^*\chi'(\De^{p_0})\in H^{d|V(\Ga)|}_c(\wt{X}_{\Ga,b})\approx\R$ (the last identification uses the orientation on $\wt{X}_\Ga$), 
then $\{\De^{p_0}\to\chi(\De^{p_0})\}_{\De^{p_0}}$ is a representative of $K_{\Ga,\pi,F}$ by the definition of $K_{\Ga,\pi,F}$ in Section \ref{Kdfn_subsubsec}. 

Note the bundle map 
\[
\begin{tikzcd}
\wt{X}_\Ga(\pi)\rar["\wt{f}"]\drar["\pi_{\wt X}"']&X_\Ga(\pi)-T_1(\pi)\rar[hook]\dar["\pi_X"]&X_\Ga(\pi)\dlar["\pi_X"]\\
&B&
\end{tikzcd}.
\]
Since $\ch\om_\Ga=0$ in a neighborhood of $T(\pi)\cap X_{p_0}$, $\wt{f}^*\ch\om_\Ga|_{\pi_{\wt X}^{-1}(B_{p_0})}$ is compactly supported. 
For each $p_0$-simplex $\De^{p_0}$ of $B$, by the naturality of the K\"{u}nneth formula, $[\wt{f}^*\ch\om_\Ga|_{\pi_{\wt X}^{-1}(\De^{p_0})}]\approx1\otimes\wt{f}^*\chi'(\De^{p_0})=1\otimes\chi(\De^{p_0})$ under the K\"unneth isomorphism
$$H^{|E(\Ga)|(d-1)}_c\big(\wt X_\Ga(\pi)|_{\De^{p_0}},\wt X_\Ga(\pi)|_{\prt\De^{p_0}}\big)\approx H^{p_0}(\De^{p_0},\prt\De^{p_0})\otimes H^{d|V(\Ga)|}_c(\wt{X}_{\Ga,b}).$$

Since $\bar\om_\Ga=h^{\hat\ul}(\hat\mu^*\ch{\om}'_\Ga)$ and $\hat\mu^*\ch{\om}'_\Ga$ is supported away form $N$, $\bar\om_\Ga$ is also supported away from $N$. 
So $[\bar\om_\Ga]=[\ch\om'_\Ga]\in H^{*}(\ccEt^{E(\Ga)},N)$, 
and $\wt{f}^{*}\bar\om_\Ga$ is a smooth form.
For every simplex $\De^{p_0}$ in $B$, pulling back the cohomology equality to $\wt{X}_\Ga(\pi)$ and restricting to $\De^{p_0}$ in $B$, 
\begin{align*}
[\wt{f}^*\bar\om_\Ga|_{\pi_{\wt X}^{-1}(\De^{p_0})}]=[\wt{f}^*\ch\om_\Ga|_{\pi_{\wt X}^{-1}(\De^{p_0})}]=
1\otimes\chi(\De^{p_0})\in H^{|E(\Ga)|(d-1)}_c\big(\wt X_\Ga(\pi)|_{\De^{p_0}},\wt X_\Ga(\pi)|_{\prt\De^{p_0}}\big)\\
\approx H^{p_0}(\De^{p_0},\prt\De^{p_0})\otimes H^{d|V(\Ga)|}_c(\wt{X}_{\Ga,b})\approx \R\otimes\R.
\end{align*}
Therefore, $\int_{\pi_{\wt X}^{-1}(\De^{p_0})}\wt{f}^*\bar\om_\Ga=\chi(\De^{p_0})$. 
This completes the proof of Proposition \ref{compare_prp}. 

\section{Some remarks about the condition in Theorem \texorpdfstring{\ref{main_thm}}{mainthm}}\label{rmk_sec}

Theorem \ref{main_thm} can potentially be formulated in a different way. 

For open subsets $U,V\subset\R^d$ and a continuous map $f:U\to V$, say $f$ is {\it almost differentiable} if the map $(f,f):U\sx U\to V\sx V$ lifts to a continuous map\footnote{Although the notation $\tl{f}$ has a completely different meaning in Sections 3-5, Section \ref{rmk_sec} is relatively independent from the others, so this abuse of notation should not cause confusion.} 
$\tl{f}:Bl_\De(U\sx U)\to Bl_\De(V\sx V),$ where $\De$ denotes the diagonal in $U\sx U$ and $V\sx V$, respectively, and $Bl_\De$ denotes real oriented blow-up along $\De$. 
We can define an {\it almost differentiable manifold} to be a topological manifold together with a maximal collection of charts where the transition maps are almost differentiable homeomorphisms whose inverses are also almost differentiable. 
So, if $M$ is an almost differentiable manifold, then $Bl_\De(M\sx M)$ is well-defined. The corresponding automorphism group $\tn{Aut}^{ad}(M)$ in this category consists of homeomorphisms $f:M\to M$ such that $(f,f):M\sx M\to M\sx M$ lifts to a homeomorphism $Bl_\De(M\sx M)\to Bl_\De(M\sx M)$. Denote by $\pi:Bl_\De(M\sx M)\to M\sx M$ the blow down map. 
\begin{rmk}
(This is pointed out by the anonymous referee.)
The notion of almost differentiability can be rephrased as follows: 
for open subsets $U,V\subset\R^d$, a continuous map $f:U\to V$ is {\it almost differentiable} if
\begin{itemize}
\vspace{-.6cm}
\itemsep0em 
\item for each $a\in U$ and $v\in S^{d-1}$, the limit
$$SDf_a(v)=\lim_{t\to0,t>0}\frac{f(a+tv)-f(a)}{|f(a+tv)-f(a)|}\in ST_{f(a)}V$$
exists, and
\item $SDf: STU\to STV$, $SDf(a,v):=SDf_a(v)$ is continuous. 
\end{itemize}
\vspace{-.6cm} 
\end{rmk}

Given two real vector spaces $T_1,T_2$ and a linear isomorphism $f:T_1\to T_2$, since $f(\la v)=\la f(v)$ for all $v\in T_1-0$, $\la\in\R-0$, $f$ induces a homeomorphism $ST_1\to ST_2$, where $ST_i=(T_i-0)/{\tn{scaling}}$ denotes the unit sphere in $T_i$.

Suppose $M$ is a $d$-dimensional almost differential manifold. Define a {\it framing} $F$ on $M$ to be a continuous map $F:\prt Bl_\De(M\sx M)\to S^{d-1}$ such that for every $x\in M$, $F|_{\pi^{-1}(x,x)}$ satisfies:
if $\phi:\R^d\supset U\xra{\sim} N\subset M$, $\phi(0)=x$ is a chart of $M$ near $x$, then $F|_{\pi^{-1}(x,x)}:=ST_xU\to S^{d-1}$ is a homeomorphism induced from a linear map $T_xU\to\R^d$. By Proposition \ref{frame_prp} below, this condition doesn't depend on the choice of such a chart $\phi$. 

Suppose $M$ is an almost differentiable manifold and $\i\in M$ a fixed point. Then the group $\cG$ defined in Section \ref{thm_sec} can be similarly defined here: 
$$\cG:=\tn{Aut}^{ad}_\infty(M):=\big\{g\in\tn{Aut}^{ad}(M)\big|\ \exists \tn{ neighborhood } N \ni\i \tn{ such that } g|_N=\id\big\}.$$

Define an {\it almost differentiable $(M,\infty)$-bundle} to be a fiber bundle with fiber $M$ and structure group $\tn{Aut}^{ad}_\infty(M)$. 
Such a bundle $\pi:E\to B$ has an associated vertical sphere tangent bundle $\pi': ST^vE\to E$. 
Define a {\it framing} on $\pi:E\to B$ to be a topological trivialization 
$$
\begin{tikzcd}
F:ST^vE|_{E-s_\infty}\ar[rr,"\tn{homeomorphism}"] \ar[rd,"\pi'"']
& & S^{d-1}\times (E-s_\infty) \ar[ld,"\tn{projection}"]\\
& E-s_\infty &
\end{tikzcd}
$$
of $\pi'|_{E-s_{\infty}}$, standard near $s_{\infty}$, and such that for each point $p\in E-s_\infty$, 
$F|_p:ST^v_pE\to S^{d-1}$
is induced from a linear map (with respect to any chart of $E|_p$ near $p$). 
When $M$ is smooth, 
this definition of framing is equivalent to a topological trivialization 
$$T^v(E-\infty)\xra{\approx} \R^d\times(E-s_\infty)$$
that is a linear isomorphism over each point in $E-s_\infty$ but varies only continuously instead of smoothly as the point in $E-s_\infty$ varies. 

The arguments in Sections \ref{space_sec} and \ref{propagator_sec} did not actually use the assumption that $\pi:E\to B$ is a smooth fiber bundle, and everything goes through if $\pi$ is only assumed to be a framed almost differentiable $(M,\infty)$-bundle, where $M$ is a smooth homology sphere.\footnote{
That $M$ is smooth instead of just almost differentiable is needed in the proof, to have the structure of the Fulton-MacPherson compactification of configuration spaces.
However, this is mainly used in Section \ref{structureX_sec}; specifically, to get Corollary \ref{dimT12_crl}. The argument there can probably go through by considering everything locally---trying to show that for every point $p\in T_2$ (resp. $T_1$), there is a neighborhood $U$ of $p$ such that $\dim_t(T_2\cap U)$ (resp. $\dim_t(T_1\cap U)$) is bounded as desired.
If this works, the statement of the Theorem \ref{restatement_thm} can potentially be generalized to the case when $M$ is only almost differentiable.} 
Therefore, we have\footnote{I would like to thank the anonymous referee for pointing this out.} 
\begin{thm}[restatement of Theorem \ref{main_thm}]
\label{restatement_thm}
Suppose $M$ is a smooth homology sphere and $\infty\in M$ is a point, then
Kontsevich's characteristic classes can be defined for framed almost differentiable $(M,\infty)$-bundles; when the bundle is smooth, the definition agrees with the original one. 
\end{thm}


\subsection{Auxiliary observations on almost differentiability}

The almost differentiability condition is actually quite strong. I do not know at the time of writing if an almost differentiable manifold (respectively, an almost differentiable bundle) necessarily has a unique $C^1$ structure. We close this section with three auxiliary observations. Example \ref{eg} below shows that almost differentiability does not imply $C^1$. Proposition \ref{conf_prp} below shows that almost differentiability implies quasi-conformal. Proposition \ref{frame_prp} below shows that an almost differentiable map induces a linear map between tangent bundles, modulo scaling by a positive smooth function. 
\begin{eg}\label{eg}
Let $B^n_{\ep}$ be the standard ball of radius $\ep<1/(2e)$ in $\R^n$. Define 
$$f:B^n_\ep\to\R^n,\quad f(x)=-2\log(|x|)\cdot x,$$
then $f$ maps $B^n_\ep$ homeomorphically onto its image, and $f$ is almost differentiable, but not continuously differentiable at $0$. See \cite{MOanswer} for a detailed proof -- the point being that the function $-2\log|x|$ approaches $\i$ slow enough as $x\to0$. 
\end{eg}

The following definition of quasi-conformal of a map is copied from \cite{quasiconformal}. A homeomorphism $f:U\to\R^d$ from an open subset $U$ to its image is {\it $k$ quasiconformal} if for all $x\in U$
$$H_f(x)=\lim_{r\to0}\sup\frac{\max\{|f(y)-f(x)|\ |\ |y-x|=r\}}{\min\{|f(y)-f(x)|\ |\ |y-x|=r\}}\le k.$$
$f$ is {\it quasiconformal} if it is $k$ quasiconformal for some $k\ge1$. 

\begin{prp}\label{conf_prp}
Let $U\subset\R^d$ be open and $f:U\to\R^d$ be an almost differentiable homeomorphism to its image. Then for every compact subset $K\subset U$, $f$ is quasiconformal on $K$. 
\end{prp}
\begin{proof}
First notice that $f$ being almost differentiable implies: for every point $x\in U$ there is a map $f'_x:S^{d-1}\to S^{d-1}$ such that for a sequence of pairs of points $\{(x_n,y_n)\in U\sx U-\Delta\}_{n=1}^\i$, 
$$\lim_{n\to\i}(x_n,y_n)=(x,x),\ \lim_{n\to\i}\frac{y_n-x_n}{|y_n-x_n|}=v\implies\lim_{n\to\i}\frac{f(y_n)-f(x_n)}{|f(y_n)-f(x_n)|}=f'_x(v).$$

Suppose $f$ is not quasiconformal on some $K$. Then for every $k>0$ there exists $x_k\in K$ such that, for all $\ep>0$, there exist $y^b_k,y^s_k\in U$ (the superscripts stand for ``big'' and ``small'') satisfying
$$|y^b_k-x_k|=|y^s_k-x_k|<\ep,\quad\frac{|f(y^b_k)-f(x_k)|}{|f(y^s_k)-f(x_k)|}>k.$$
Let $k$ range over $\Z^{>0}$. Since $K$ is compact, by possibly passing to a subsequence, we can assume $x_k\to x$ as $k\to\i$ for some $x\in K$. 
Plugging in $\ep=1/k$ above, for every $k$ we get a tuple of points $x_k,y^b_k,y^s_k$, all limit to $x$ as $k\to\i$. For each $k$, denote by $S_k$ the sphere centered at $x_k$ on which $y^s_k,y^b_k$ lie. 
Define $z_k\in S_k$ to be the midpoint of the shortest geodesic (if not unique, take an arbitrary one) between $y^b_k,y^s_k$ on $S_k$. 
This implies that the angle between the vectors $y^b_k-x_k$ and $y^b_k-z_k$ is at least $\pi/4$; same for the vectors $z_k-y^s_k$ and $z_k-x_k$.
\begin{center}
\begin{tikzpicture}
\coordinate (xk) at (0,0);
\draw (xk) circle (1);
\draw [fill] (xk) circle (0.03);
\node at (-0.3,0) {$x_k$};
\draw [fill] (0,1) coordinate (ybk)
circle (0.03); 
\node at (0,1.4) {$y^b_k$}; 
\coordinate (ysk) at ($(xk)+(-50:1)$);
\draw [fill] (ysk) circle (0.03); 
\node [xshift=0.3cm] at (ysk) {$y^s_k$}; 
\coordinate (zk) at ($(xk)+(20:1)$); 
\draw [fill] (zk) circle (0.03); 
\node [xshift=0.3cm] at (zk) {$z_k$}; 

\draw [gray] (xk) -- (zk)
(zk) -- (ybk)
(ybk) -- (xk)
(xk) -- (ysk)
(ysk) -- (zk); 

\node at (3,0) {$\xrightarrow{\ \ \ f\ \ \ }$};

\coordinate (fxk) at (6,0); 
\draw (fxk) ellipse (0.5 and 1.5);
\draw [fill] (fxk) circle (0.03);
\node [xshift=-0.5cm] at (fxk) {$f(x_k)$};
\coordinate (fybk) at ($(fxk)+(0,1.5)$); 
\draw [fill] (fybk) circle (0.03); 
\node [yshift=0.3cm] at (fybk) {$f(y^b_k)$};
\coordinate (fysk) at ($(fxk)+(0.5,0)$); 
\draw [fill] (fysk) circle (0.03); 
\node [xshift=0.4cm] at (fysk) {$f(y^b_k)$};
\coordinate (fzk) at ($(fxk)+(45:0.5 and 1.5)$);
\draw [fill] (fzk) circle (0.03); 
\node [xshift=0.4cm] at (fzk) {$f(z_k)$};

\draw [gray] (fxk) -- (fzk)
(fzk) -- (fybk)
(fybk) -- (fxk)
(fxk) -- (fysk)
(fysk) -- (fzk); 
\end{tikzpicture}
\end{center}
Since 
$$\frac{|f(y^b_k)-f(x_k)|}{|f(z_k)-f(x_k)|}\cdot\frac{|f(z_k)-f(x_k)|}{|f(y^s_k)-f(x_k)|}>k,$$
one of the factors must be bigger than $\sqrt k$. In the case it is the first factor, define $z^b_k=y^b_k$, $z^s_k=z_k$; in the case it is the second factor, define $z^b_k=z_k$, $z^s_k=y^s_k$. Now we have a sequence of tuples $(x_k,z^b_k,z^s_k)$, all converge to $x$ as $k\to\i$, and 
$$\lim_{k\to\i}\frac{f(z^b_k)-f(x_k)}{|f(z^b_k)-f(x_k)|}=\lim_{k\to\i}\frac{f(z^b_k)-f(z^s_k)}{|f(z^b_k)-f(z^s_k)|},$$
because in the triangle $(f(z^b_k),f(z^s_k),f(x_k))$, the ratio between the lengths of the edges $f(z^s_k)f(x_k)$ and $f(z^b_k)f(x_k)$ goes to 0, implying that the angle between the edges $f(z^b_k)f(z^s_k)$ and $f(z^b_k)f(x_k)$ goes to 0. Now, since $f$ is almost differentiable and thus so is $f^{-1}$, 
\begin{align*}
\lim_{k\to\i}\frac{z^b_k-x_k}{|z^b_k-x_k|}&=(f^{-1})'_{f(x)}\Big(\lim_{k\to\i}\frac{f(z^b_k)-f(x_k)}{|f(z^b_k)-f(x_k)|}\Big)\\
&=(f^{-1})'_{f(x)}\Big(\lim_{k\to\i}\frac{f(z^b_k)-f(z^s_k)}{|f(z^b_k)-f(z^s_k)|}\Big)=\lim_{k\to\i}\frac{z^b_k-z^s_k}{|z^b_k-z^s_k|},
\end{align*}
which contradicts that the angle between the two vectors is at least $\pi/4$. 
\end{proof}
The converse to Proposition \ref{conf_prp} is not true. For example, $f:\R^2\to\R^2,f(x)=|x|^{-1/2}\cdot x$ is quasiconformal but not almost differentiable. 

\begin{prp}\label{frame_prp}
Let $U_1,U_2$ be open subsets of $\R^d$ and $f:U_1\to U_2$ be an almost differentiable homeomorphism. For every point $x\in U$, denote $\tl f_x:ST_xU_1\to ST_{f(x)}U_2$ the homeomorphism given by restricting $\tl f$ to $\pi^{-1}(x,x)$, where $ST$ denotes the unit sphere in the tangent space. Then
$\tl f_x$ is induced from a linear isomorphism $T_xU_1\to T_xU_2$. 
\end{prp}
\begin{proof}
Given a directed, ordered graph $\Ga$ and $U\subset\R^d$ an open subset, define
$$g_\Ga^U:U^{V(\Ga)}-\De_{\tn{big}}\to (U\sx U-\De)^{E(\Ga)},\qquad (x_v)_{v\in V(\Ga)}\lra\big((x_{v^\Ga_-(e)},x_{v^\Ga_+(e)})\big)_{e\in E(\Ga)};$$
then, by the construction of Fulton-MacPherson compactification, $g^U_\Ga$ extends to a map 
$$\bar g^U_\Ga:\ov{C}_{V(\Ga)}(U)\lra Bl_\De(U\sx U)^{E(\Ga)},$$
where $\ov{C}_{V(\Ga)}(U)$ is the Fulton-MacPherson configuration space of $V(\Ga)$-labeled marked points in $U$. 
Similar to the proof of Lemma \ref{1_lmm}, $\tn{image}(\bar g^U_\Ga)$ is the closure of $\tn{image}(g^U_\Ga)$ in $Bl_\De(U\sx U)^{E(\Ga)}$. Let $\pi:Bl_\De(U\sx U)^{E(\Ga)}\to (U\sx U)^{E(\Ga)}$ be the blow-down map. For a point $x\in U$, 
$\pi^{-1}\big(((x,x))_{e\in E(\Ga)}\big)=(ST_xU)^{E(\Ga)}=(S^{d-1})^{E(\Ga)}$ and, by the construction of Fulton-MacPherson compactification,  
$$\pi^{-1}\big(((x,x))_{e\in E(\Ga)}\big)\cap \tn{image}(\bar g^U_\Ga)=\ov{V}_\Ga\subset (S^{d-1})^{E(\Ga)};$$
recall $\ov{V}_\Ga$ is defined in Definition \ref{Vga_dfn}. Now plug in $U_1,U_2$ for $U$. From the definition of $g^U_\Ga$, $(f,f)^{E(\Ga)}\circ g^{U_1}_\Ga=g^{U_2}_\Ga\circ f^{V(\Ga)}$, so $\tn{image}(g^{U_2}_\Ga)=(f,f)^{E(\Ga)}(\tn{image}(g^{U_1}_\Ga))$. Passing to their closures in $Bl_\De(U_i\sx U_i)^{E(\Ga)}$, we have $\tn{image}(\bar g^{U_2}_\Ga)=\tl{f}^{E(\Ga)}(\tn{image}(\bar g^{U_1}_\Ga))$. Therefore, for each $x\in U_1$,
$$(\tl{f}_x)^{E(\Ga)}:(S^{d-1})^{E(\Ga)}=(ST_xU_1)^{E(\Ga)}\lra(ST_{f(x)}U_2)^{E(\Ga)}=(S^{d-1})^{E(\Ga)}$$
maps $\ov{V}_\Ga$ to $\ov{V}_\Ga$. The conclusion of the proposition follows from Lemma \ref{lattice_lmm} below. 
\end{proof}
\begin{lmm}\label{lattice_lmm}
Suppose $f:S^{d-1}\to S^{d-1}$ is a homeomorphism such that for any directed, ordered graph $\Ga$, $\ov{V}_\Ga$ is invariant under $f^{E(\Ga)}:=(f,\ldots,f):(S^{d-1})^{E(\Ga)}\to(S^{d-1})^{E(\Ga)}$, then $f$ is induced by a map $F\in GL(d)$. 
\end{lmm}
\begin{proof}(This proof is given by Fabian Gundlach.) 
The strategy is to define an increasing sequence of subsets $\{A_n\subset S^{d-1}\}_{n=1}^\i$, $A_n\subset A_{n+1}$, such that $\cup_nA_n$ is dense in $S^{d-1}$, and show that for each $n$, $f|_{A_n}$ is induced by a map $F_n\in GL(d)$. 
For each $n\in\Z^{>0}$, let $\Ga_n$ be the complete graph with $(n+1)^d$ vertices labeled by elements in $L_n:=\{0,\ldots,n\}^d$. 
Then, putting the vertex of $\Ga_n$ labeled by $(m_1,\ldots,m_d)\in L_n$ at $(m_1,\ldots,m_d)\in\R^d$ gives an element in $C^\quo_{V(\Ga_n)}(\R^d)$. 
Since $\ov{V}_{\Ga_n}$ is invariant under $f^{E(\Ga_n)}$, there is an element $x=(x_{\underline{m}}\in\R^d)_{\underline{m}\in L_n}\in(\R^d)^{V(\Ga_n)}$ such that for any $\underline{m}_1\neq\underline{m}_2\in L_n$, $x_{\underline m_1}\neq x_{\underline m_2}$ and 
\begin{equation}\label{latticegraphmap_eqn}
\frac{x_{\underline m_1}-x_{\underline m_2}}{|x_{\underline m_1}-x_{\underline m_2}|}=f\Big(\frac{\underline m_1-\underline m_2}{|\underline m_1-\underline m_2|}\Big).
\end{equation}
Denote $e_j=(0,\ldots,0,1,0\ldots,0)\in L_n$ where $1$ is at the $j$-th place. We also view $e_j$ as an element in $S^{d-1}$. For all $\underline{m}\in L_n$ and $j\in\{1,\ldots,d\}$, $x_{\underline{m}+e_j}-x_{\underline{m}}$ has direction $f({e_j})$ by \eref{latticegraphmap_eqn}. We next show that $|x_{\underline{m}+e_j}-x_{\underline m}|$ does not depend on $\underline{m}$ either. For $k\neq j$, since $x_{\underline{m}\pm e_k}-x_{\um}$ is parallel to $x_{\um+e_j\pm e_k}-x_{\um+e_j}$, the points $x_{\um},x_{\um+e_j},x_{\um\pm e_k},x_{\um+e_j\pm e_k}$ form a parallelogram, so we must have $x_{\um\pm e_k+e_j}-x_{\um\pm e_k}=x_{\um+e_j}-x_{\um}$. Therefore, plugging in $k$ for $j$, we have $x_{\um+e_k}-x_{\um}=x_{\um-e_j+e_k}-x_{\um-e_j}$ as well. So, in the two triangles $(x_\um,x_{\um+e_j},x_{\um+e_k})$ and $(x_{\um-e_j},x_{\um},x_{\um-e_j+e_k})$, one of the pairs of corresponding edges are equal as vectors. The other two pairs of corresponding edges are both parallel by \eref{latticegraphmap_eqn}, so they must both be equal. This shows $x_{\um+e_j}-x_\um=x_\um-x_{\um-e_j}$. Therefore, $x_{\um+e_j}-x_\um$ does not depend on the choice of $\um$. Without loss of generality we can assume $x_{(0,\ldots,0)}=(0,\ldots,0)\in\R^d$. Then $x_{\um+e_j}-x_\um=x_{e_j}$ for all $\um,j$. So, $x_{\um_1+\um_2}=x_{\um_1}+x_{\um_2}$ for all $\um_1,\um_2$. This shows that the map $F_n\in GL(d)$ defined by ``$\forall i\ F_n(e_i)=x_{e_i}$'' maps $\um$ to $x_\um$ for all $\um\in L_n$. 
For $F\in GL(d)$, denote by $\hat F:S^{d-1}\to S^{d-1}$ the homeomorphism induced by $F$. 
Now define $$A_n=\Big\{\frac{\um_1-\um_2}{|\um_1-\um_2|}\Big\}_{\um_1\neq\um_2\in L_n}\subset S^{d-1},$$
then $f|_{A_n}=\hat F_n|_{A_n}$. On the other hand, the condition $\hat F_{A_n}=f|_{A_n}$ uniquely determines $F$ up to scaling: since $\{e_i\}_{i=1}^d\subset A_n$, there exists $(0\neq\la_i\in\R)_{i=1}^d$ such that $F_n(e_i)=\la_if(e_i)$; since $e_i+e_j\in A_n$, the direction of $F(e_i+e_j)=\la_if(e_i)+\la_jf(e_j)$ is determined by $f$, so $\la_i/\la_j$ is determined. Therefore, for different $n$, $F_n$ differ only by scaling. This determines a map $F\in GL(d)$ up to scaling, satisfying $\hat F|_{\cup_nA_n}=f|_{\cup_nA_n}$. Since $\cup_nA_n$ is dense in $S^{d-1}$, $\hat F=f$. 
\end{proof}

\texttt{xujiachen@math.harvard.edu}

\end{document}